\documentclass[12pt]{amsart}
\usepackage{float}
\usepackage[colorlinks,
            linkcolor=black,
            anchorcolor=blue,
            citecolor=black]{hyperref}
\usepackage{graphicx}
\usepackage{cleveref}
\usepackage{amsmath}
\usepackage{amssymb}
\usepackage{setspace}
\usepackage{amsthm}

\usepackage{geometry}
\allowdisplaybreaks[4]

\newtheorem{thm}{Theorem}[section]
\newtheorem{cor}[thm]{Corollary}

\newtheorem{lem}[thm]{Lemma}

\theoremstyle{definition}
\newtheorem{defn}[thm]{Definition}
\theoremstyle{remark}
\newtheorem{rem}[thm]{Remark}
\theoremstyle{conclusion}

\theoremstyle{conjecture}

\numberwithin{equation}{section}

\newcommand{\be}{\begin{equation}}
\newcommand{\ee}{\end{equation}}

\newcommand{\R}{\mathbb{R}}

\begin{document}
\title[Nonlocal $p$-Laplace equations with Hardy potential]{Radial symmetry and sharp asymptotic behaviors of nonnegative solutions to 
weighted doubly $D^{1,p}$-critical quasi-linear nonlocal elliptic equations with Hardy potential}

\author{Daomin Cao, Wei Dai, Yafei Li}

\address{Institute of Applied Mathematics, Chinese Academy of Sciences, Beijing 100190, and University of Chinese Academy of Sciences, Beijing 100049, P. R. China}
\email{dmcao@amt.ac.cn}

\address{School of Mathematical Sciences, Beihang University (BUAA), Beijing 100191, and Key Laboratory of Mathematics, Informatics and Behavioral Semantics, Ministry of Education, Beijing 100191, P. R. China}
\email{weidai@buaa.edu.cn}

\address{School of Mathematical Sciences, Beihang University (BUAA), Beijing 100191, P. R. China}
\email{yafeili@buaa.edu.cn}

\thanks{Daomin Cao is supported by National Key R\&D Program of China (Grant 2023YFA1010001) and NNSF of China (Grant 11271354). Wei Dai is supported by the NNSF of China (No. 12222102), the National Science and Technology Major Project (2022ZD0116401) and the Fundamental Research Funds for the Central Universities. Yafei Li is supported by the Fundamental Research Funds for the Central Universities.}

\begin{abstract}
In this paper, we mainly consider nonnegative weak solutions $u\in D^{1,p}(\R^{N})$ to the doubly $D^{1,p}(\R^{N})$-critical nonlocal quasi-linear Schr\"{o}dinger-Hartree equation:
\begin{align*}
-\Delta_p u- \mu \frac{u^{p-1}}{|x|^p}=\left(|x|^{-2p}\ast |u|^{p}\right)|u|^{p-2}u \qquad &\mbox{in} \,\, \mathbb{R}^N,
\end{align*}
where $N\geq3$, $0\leq\mu< \bar{\mu}:=\left( (N-p)/p \right)^p$ and $1<p<\frac{N}{2}$. When $\mu>0$, due to appearance of the Hardy potential, the equation has singularity at $0\in\mathbb{R}^{N}$ and hence is not translation invariant, so sharp asymptotic estimates near the origin must be involved. First, we establish regularity and the sharp estimates on asymptotic behaviors near the origin and the infinity for any positive solution $u\in D^{1,p}(\R^{N})$ (and $|\nabla u|$) to more general equation
$-\triangle_p u - \mu \frac{1}{|x|^p}u^{p-1}=V(x)\frac{1}{|x|^s}u^{p-1}$ with $N\geq2$, $0\leq\mu< \bar{\mu}$, $1<p<N$, $0\leq s < p$ and $0\leq V(x)\in L^\frac{N}{p-s}(\R^N)$. Then, as a consequence, we can apply the method of moving planes to prove that all the nontrivial nonnegative solutions in $D^{1,p}(\R^{N})$ are radially symmetric and strictly radially decreasing about the origin $0\in\mathbb{R}^{N}$. The sharp asymptotic
estimates and radial symmetry for more general weighted doubly $D^{1,p}$-critical nonlocal quasi-linear equations were also derived. Our results extend the results in \cite{DLL} from the special case $\mu=0$ to general cases $0\leq\mu<\bar{\mu}$.
\end{abstract}

\maketitle {\small {\bf Keywords:} Quasi-linear nonlocal elliptic equations with $p$-Laplacian; Nonlocal Hartree type nonlinearity; Hardy potential; Doubly critical problem; Radial symmetry; Sharp estimates on asymptotic behaviors; The method of moving planes.\\

}

\section{Introduction}
\subsection{Background and setting of the problem}
In this paper, we mainly consider the nonnegative $D^{1,p}(\R^{N})$-weak solutions to the following doubly $D^{1,p}(\R^{N})$-critical quasi-linear static Schr\"{o}dinger-Hartree equation with Hardy potential and nonlocal nonlinearity:
\begin{align}\label{eq1.1}
\left\{ \begin{array}{ll} \displaystyle
-\Delta_p u - \frac{\mu}{|x|^p} u^{p-1}=\left(|x|^{-2p}\ast u^{p}\right)u^{p-1}  \quad\,\,\,\,&\mbox{in}\,\, \R^N, \\ \\
u \in D^{1,p}(\R^N),\quad\,\,\,\,\,\,u\geq0 \,\,\,  \mbox{in}\,\, \R^N,&
\end{array}
\right.\hspace{1cm}
\end{align}
where $N\geq3$, $1<p<\frac{N}{2}$, $0\leq\mu<\bar{\mu}:=\left(\frac{N-p}{p}\right)^p$ with $\bar{\mu}$ being the best constant in Hardy inequality. Let $V_1(x):=\left(|x|^{-2p}\ast u^{p}\right)(x):=\int_{\R^N} \frac{u^p(y)}{|x-y|^{2p}}\mathrm{d}y$. Due to the kernel $|x|^{-2p}$ in the nonlocal convolution term $V_{1}$, the full range of $p$ for the quasi-linear nonlocal problem is $1<p<\frac{N}{2}$. For $1<p<N$, the $p$-Laplace operator $\Delta_p$ is defined by $\Delta_p u:=div (|\nabla u|^{p-2} \nabla u)$. For $1<p<2$ or $p>2$, the $p$-Laplacian $\Delta_p$ is singular elliptic or degenerate elliptic, respectively. The function space
$$ D^{1,p}(\R^N):=\left\{ u\in L^{p^{\star}}(\R^N) \, \bigg| \, \int_{\R^N} |\nabla u|^p \mathrm{d}x <+\infty \right\}$$
is the completion of $C_0^\infty$ with respect to the norm $\| u \|:=\| \nabla u \|_{L^p(\R^N)}$. Since equation \eqref{eq1.1} is quasi-linear, nonlocal and doubly $D^{1,p}(\R^{N})$-critical, it would be a quite interesting and challenging problem to study the sharp asymptotic estimates, radial symmetry, strictly radial monotonicity and classification of solutions to \eqref{eq1.1}.

\medskip

The nonlocal quasi-linear equation \eqref{eq1.1} is $D^{1,p}(\R^{N})$-critical in the sense that both \eqref{eq1.1} and the $D^{1,p}$-norm $\| \nabla u \|_{L^p(\R^N)}$ are invariant under the scaling $u\mapsto u_{\lambda}(\cdot):=\lambda^{\frac{N-p}{p}}u(\lambda\cdot)$. Equation \eqref{eq1.1} is also invariant under arbitrary rotations, but is not invariant with respect to translations if $\mu\neq0$ due to the appearance of the critical Hardy potential. For equation \eqref{eq1.1}, the finite $D^{1,p}(\R^N)$-energy assumption $u\in D^{1,p}(\R^N)$ is necessary to guarantee $V_1(x)=|x|^{-2p}\ast u^{p}\in L^{\frac{N}{p}}(\mathbb{R}^{N})$ and hence $V_1(x)$ is finite and well-defined almost everywhere. In fact, due to Hardy-Littlewood-Sobolev inequality (see \cite{FL1,FL2,Lieb,Lions2}), one has $V_1(x)=\left(|x|^{-2p}\ast |u|^{p}\right)(x)\in L^{\frac{N}{p}}(\mathbb{R}^{N})$. More precisely,
\begin{equation}\label{conv}
  \|V_1(x)\|_{L^{\frac{N}{p}}(\mathbb{R}^{N})}\leq C\|u\|^{p}_{L^{p^{\star}}(\mathbb{R}^{N})}, \qquad \forall \,\, u\in D^{1,p}(\mathbb{R}^{N}),
\end{equation}
where $C=C(N,p)>0$. Then it follows from H\"{o}lder's inequality and the critical Sobolev embedding inequality that, for any $u\in D^{1,p}(\mathbb{R}^{N})$,
\begin{equation}\label{a4}
  \int_{\mathbb{R}^{N}}\left(\frac{1}{|x|^{2p}}\ast |u|^{p}\right)(x) \cdot |u|^{p}(x)\mathrm{d}x\leq \|V_{1}\|_{L^{\frac{N}{p}}(\R^N)}\|u\|^{p}_{L^{p^{\star}}(\R^N)}\leq C\|\nabla u\|^{2p}_{L^{p}(\R^N)}.
\end{equation}
We say \eqref{eq1.1} is doubly $D^{1,p}(\mathbb{R}^{N})$-critical in the sense that the Hardy potential $\frac{u^{p-1}}{|x|^p}$ therein is also $D^{1,p}(\mathbb{R}^{N})$-critical. Indeed, by the Hardy-Sobolev inequality (see e.g. \cite{CKN,NHMSM}):
\begin{equation}\label{eq-HDSI}
  \left(\int_{\R^N}\frac{|u|^{p_s}(x)}{|x|^s}\mathrm{d}x\right)^{\frac{1}{p_s}} \leq C_{p,s} \| \nabla u \|_{L^p(\R^N)},  \qquad \forall \,\, u\in D^{1,p}(\mathbb{R}^{N}),
\end{equation}
where $0\leq s\leq p$, $p_s:=\frac{p(N-s)}{N-p}$ is the Hardy-Sobolev critical exponent, $C_{p,s}>0$ is the best constant for the Hardy-Sobolev embedding, and $C_{p,p}=\bar{\mu}^{-1}=\left(\frac{p}{N-p}\right)^p$ can not be attained when $s=p$. \eqref{a4}, combining with  \eqref{eq-HDSI} with $s=p$, deduces that, for $0\leq \mu<\bar{\mu}$,
\begin{equation}\label{eq-HLS}
  \|u\|_{V_{1}}:=\left(\int_{\mathbb{R}^{N}}\int_{\mathbb{R}^{N}}\frac{|u|^{p}(x)|u|^{p}(y)}{|x-y|^{2p}}\mathrm{d}x\mathrm{d}y\right)^{\frac{1}{2p}}\leq C_{N,p,\mu}\left(\| \nabla u \|^p_{L^p(\R^N)} - \mu\int_{\R^N}\frac{|u|^p(x)}{|x|^p}\mathrm{d}x \right)^{\frac{1}{p}},
\end{equation}
where $C_{N,p,\mu}$ is the best constant for the critical embedding inequality \eqref{eq-HLS}, which can be characterized by the following minimization problem
\begin{equation}\label{mini-hartree}
  C_{N,p,\mu}:=\inf\left\{\left(\| \nabla u \|^p_{L^p(\R^N)} -\mu\int_{\R^N}\frac{|u|^p(x)}{|x|^p}\mathrm{d}x \right)^{\frac{1}{p}} \,\Big|\,u\in D^{1,p}(\mathbb{R}^{N}), \,\, \|u\|_{V_{1}}=1\right\},
\end{equation}
where $0\leq\mu<\bar{\mu}:=\left(\frac{N-p}{p}\right)^p$.  The doubly $D^{1,p}(\mathbb{R}^{N})$-critical nonlocal quasi-linear equation \eqref{eq1.1} is the Euler-Lagrange equation for the minimization problem \eqref{mini-hartree}.

\smallskip

The general Choquard-Pekar equations use the limiting embedding inequality of the type \eqref{eq-HLS} with $\mu=0$ (c.f. e.g. \cite{LE76,LPL80,Lions1}). When $\mu=0$, Lions \cite{Lions1} proved that the minimization problem \eqref{mini-hartree} has a minimum via the concentration-compactness principle (c.f. iv) in page 169 of \cite{Lions1}). For the general minimization problem \eqref{mini-hartree} with $0\leq\mu<\bar{\mu}$, Su and Chen in \cite{SYCH} obtained that the minimization problem \eqref{mini-hartree} and hence the best constant $C_{N,p,\mu}$ in \eqref{eq-HLS} can be attained by a radially symmetric, non-increasing and nontrivial nonnegative function in $D^{1,p}(\R^N)$. Consequently, the nonlocal quasi-linear Hartree equation \eqref{eq1.1} possesses (at least) a nontrivial nonnegative weak solution $u\in D^{1,p}(\mathbb{R}^{N})$. For more literatures on existence and multiplicity of solutions to elliptic equations with Hardy potentials, c.f. e.g. \cite{CH,CP,CY,MMPS} and the references therein.

\smallskip

In the semi-linear elliptic case $p=2$, \eqref{eq1.1} becomes the following $H^{1}$-energy-critical nonlocal nonlinear Hartree equation with Hardy potential:
\begin{align}\label{ciHeq}
-\Delta u -\mu\frac{u}{|x|^2}=\left(|x|^{-4}\ast u^{2}\right)u  \quad\,\,\,\,&\mbox{in}\,\, \R^N,
\end{align}
where $N \geq 5$ and $0\leq\mu<\left(\frac{N-2}{2}\right)^2$. Equation \eqref{eq1.1} is the quasi-linear counterpart of \eqref{ciHeq}. PDEs similar to \eqref{ciHeq} arise in the Hartree-Fock theory of the nonlinear Schr\"{o}dinger equations (see \cite{LS}). The solution $u$ to \eqref{ciHeq} with $\mu=0$ is also a ground state or a stationary solution to the following focusing energy-critical dynamic Schr\"{o}dinger-Hartree equation (see e.g. \cite{LMZ,MXZ3})
\begin{equation}\label{Hartree}
i\partial_{t}u+\Delta u=-\left(|x|^{-4}\ast u^{2}\right)u, \qquad (t,x)\in\mathbb{R}\times\mathbb{R}^{N}.
\end{equation}
The Schr\"{o}dinger-Hartree equations have many interesting applications in the quantum theory of large systems of non-relativistic bosonic atoms and molecules (see, e.g. \cite{FL}) and have been quite intensively studied, c.f. e.g. \cite{LMZ,MXZ3} and the references therein, in which the ground state solution can be regarded as a crucial criterion or threshold for global well-posedness and scattering in the focusing case. Therefore, the classification of solutions to \eqref{ciHeq} plays an important and fundamental role in the study of the focusing dynamic Schr\"{o}dinger-Hartree equations \eqref{Hartree}. For classification and other related quantitative and qualitative properties of solutions to energy-critical Hartree type equations like \eqref{ciHeq} involving fractional, second, higher and arbitrary order Laplacians $(-\Delta)^{s}$ ($0<s<+\infty$) and Choquard equations, we refer to \cite{AY,CD,CDZ,DFHQW,DFQ,DL,DLQ,DQ,GHPS,Lei,Liu,LE76,LPL80,MZ,MS} and the references therein.

\smallskip

In the quasi-linear case (i.e. $1<p<N$, $p \neq 2$), the doubly energy-critical nonlocal equation \eqref{eq1.1} becomes much more difficult and challenging. First, since the $p$-Laplacian $-\Delta_{p}$ is quasi-linear, we can not simply deduce the comparison principles from the maximum principles and the Kelvin type transforms are not available. Moreover, it is necessary for us to discuss the singular elliptic case $1<p<2$ and the degenerate elliptic case $p>2$ separately. Second, due to the lack of Kelvin type transforms, in order to carry out the method of moving planes, we need first to derive the sharp asymptotic estimates of nontrivial nonnegative solutions $u$ (and $|\nabla u|$) to equation \eqref{eq1.1}. Third, comparing to the case $\mu=0$ studied in \cite{DLL}, since the Hardy potential $\frac{u^{p-1}}{|x|^p}$ is singular at the origin $0\in\mathbb{R}^{N}$, equation \eqref{eq1.1} is not translation invariant and hence the sharp asymptotic estimates near the origin $0$ must be established. Forth, due to the nonlocal virtue of the Hartree nonlinearity in \eqref{eq1.1}, the methods and techniques for point-wise local nonlinearity do not work anymore, we need to develop and use some new ideas and techniques to derive the sharp asymptotic estimates near $0$ and $\infty$ and the radial symmetry.

\smallskip

\subsection{Main results}
In this paper, we aim to overcome all the above mentioned difficulties caused by the quasi-linear $p$-Laplacian $-\Delta_{p}$, the singularity of $D^{1,p}(\R^N)$-critical Hardy potential and the nonlocal nonlinear interaction, and prove the sharp asymptotic estimates, the radial symmetry and strictly radial monotonicity of positive solutions to \eqref{eq1.1}. These results will lead to a complete classification of nonnegative weak solutions to \eqref{eq1.1} provided that the radial solution is unique.

\medskip

To this end, we will first generalize the equation \eqref{eq1.1} and consider the following more general doubly $D^{1,p}(\R^{N})$-critical quasi-linear equation with Hardy potential:
\begin{equation}\label{gen-eq}
-\Delta_p u - \frac{\mu}{|x|^p}|u|^{p-2}u= V(x)|x|^{-s} |u|^{p-2}u   \qquad \text{in}\,\, \R^N,
\end{equation}
where $0\leq\mu< \bar{\mu}=\left( \frac{N-p}{p}\right)^p$, $1<p<N$, $0\leq s < p$ and $0\leq V(x)\in L^\frac{N}{p-s}(\R^N)$. Equation \eqref{gen-eq} generalizes  equations \eqref{eq1.1}, \eqref{eq1.2} and \eqref{eq1.3}. Equation \eqref{eq1.1} is a typical special case of \eqref{gen-eq} with $s=0$ and $V:=|x|^{-2p}\ast |u|^{p}$.

\smallskip

\begin{defn}[$D^{1,p}(\R^{N})$-weak solution]\label{weak}
A $D^{1,p}(\R^{N})$-weak solution of the generalized equation \eqref{gen-eq} is a function $u\in D^{1,p}(\R^N)$ such that
\begin{align*}
\int_{\R^N} |\nabla u|^{p-2} \langle \nabla u, \nabla\psi\rangle \mathrm{d}x
&-\mu \int_{\R^N}\frac{1}{|x|^{p}}|u|^{p-2}(x)u(x)\psi(x)\mathrm{d}x\\
&= \int_{\R^N}V(x)|x|^{-s}|u|^{p-2}(x)u(x)\psi(x)\mathrm{d}x, \,\,\,\,\quad \forall \,\,\psi \in C_0^\infty(\R^N).
\end{align*}
\end{defn}

\smallskip

For any nonnegative weak solution $u$ of the generalized equation \eqref{gen-eq}, by exploiting the De Giorgi-Moser-Nash iteration technique (c.f. \cite{TMOSER}) as in \cite[Lemma B.3]{CDPSYS} and \cite{DLL,HL,JSLB,Trud}, we prove in Lemmas \ref{lm:l-b-r} that $u\in L_{loc}^q(\R^N\setminus\{0\})$ for some $N/(p-s) < q < +\infty$, and that $u\in L^{\infty}_{loc}(\R^N \setminus \{0\})$ if $V\in L^{q}_{loc}(\R^{N}\setminus \{0\})$ for some $N/(p-s)<q<+\infty$. Next, using the standard $C^{1,\alpha}$ estimates (see \cite{ED83,KM,PT}), we deduce that $u\in C_{loc}^{1,\alpha}(\R^N\setminus \{0\})$ for some $0<\alpha<\min\{1,1/(p-1)\}$. Moreover, it follows from the strong maximum principle (see Lemma \ref{hopf}) that each nontrivial nonnegative solution to \eqref{eq1.1} is actually strictly positive provided that $V\in L^{\infty}_{loc}(\mathbb{R}^{N}\setminus \{0\})$.

\smallskip

For $\mu=0$ and $V=u^{p_s-p}$ with nonnegative $u \in D^{1,p}(\R^N)$, the general equation \eqref{gen-eq} becomes the following weighted $D^{1,p}(\R^N)$-critical $p$-Laplace equation:
\begin{align}\label{eq-cri-har-0}
 -\Delta_p u = |x|^{-s} u^{p_s-1}, \quad\, u\geq0 \quad\,\,\,\,&\mbox{in}\,\, \R^N,
\end{align}
where $1<p<N$, $0\leq s <p$ and $p_s=\frac{p(N-s)}{N-p}$. In particular, if $s=0$, \eqref{eq-cri-har-0} becomes
\begin{align}\label{criticeqution}
 -\Delta_p u = u^{p^\star-1},  \quad\, u\geq0 \quad\,\,\,\,&\mbox{in}\,\, \R^N,
\end{align}
where $u \in D^{1,p}(\R^N)$, $1<p<N$ and $p^\star:=\frac{Np}{N-p}$ is the critical Sobolev embedding exponent. Serrin \cite{S} and Serrin and Zou \cite{SJZH02}, among other things, proved regularity and Liouville type theorem for general quasi-linear equations with point-wise local nonlinearity and sharp lower bound on the asymptotic estimate of super-$p$-harmonic functions. Guedda and Veron \cite{GV} proved the uniqueness of radially symmetrical positive solution $U(x)=U(r)$ with $r=|x|$ to \eqref{criticeqution} satisfying $U(0)=b>0$ and $U'(0)=0$, and hence reduced the classification of solutions to the radial symmetry of solutions.

\smallskip

In the singular elliptic case $1<p<2$, under the $C^{1}\cap W^{1,p}$-regularity assumptions on the positive solutions and several assumptions on the local nonlinear term, the radial symmetry and strictly radial monotonicity of positive solutions to \eqref{criticeqution} was obtained in \cite{DLPFRM,LDMR} by using the moving plane technique. Later, in \cite{LDSMLMSB}, Damascelli, Merch\'{a}n, Montoro and Sciunzi refined the moving plane method used in \cite{DLPFRM,LDMR} and derived the radial symmetry of positive $D^{1,p}(\R^N)$-weak solutions to \eqref{criticeqution} without regularity assumptions on solutions. However, they still required in \cite{LDSMLMSB} that the nonlinear term is locally Lipschitz continuous, i.e., $p^{\star} \geq 2$. Subsequently, by applying scaling arguments and the doubling Lemma (see \cite{PPPQPS}), V\'{e}tois \cite{VJ16} first derived a preliminary estimate on the decay rate of positive $D^{1,p}(\R^N)$-weak solutions $u$ to \eqref{criticeqution} (i.e., $u(x)\leq \frac{C}{|x|^{\frac{N-p}{p}}}$ for $|x|$ large). Then, by combining the preliminary estimate with scaling arguments, the Harnack-type inequalities (c.f. \cite{LDBS,JSLB}), regularity estimates and the boundedness estimate of the quasi-norm $L^{t,\infty}(\mathbb{R}^{N})$ with $t=\frac{N(p-1)}{N-p}$ (Lemma 2.2 in \cite{VJ16}), V\'{e}tois \cite{VJ16} obtained the sharp asymptotic estimates of the positive $D^{1,p}(\R^N)$-weak solutions $u$ (i.e., $u\sim\left(1+|x|^{\frac{N-p}{p-1}}\right)^{-1}$) and the sharp upper bound on decay rate of $|\nabla u|$ (i.e., $|\nabla u|\leq \frac{C}{|x|^{\frac{N-1}{p-1}}}$ for $|x|$ large) for $1<p<N$. These sharp asymptotic estimates combined with the results in \cite{DLPFRM,LDMR} successfully extended the radial symmetry results for \eqref{criticeqution} in \cite{LDSMLMSB} to the full singular elliptic range $1<p<2$.

\smallskip

For equation \eqref{criticeqution} in the degenerate elliptic case $p>2$,  based on scaling arguments, the results for limiting profile at infinity, the uniqueness up to multipliers of $p$-harmonic maps in $\mathbb{R}^{N}\setminus\{0\}$ and the sharp asymptotic estimates on $u$ proved in \cite{VJ16},Sciunzi \cite{BS16} developed a new technique and derived the sharp lower bound on decay rate of $|\nabla u|$ (i.e., $|\nabla u|\geq \frac{c}{|x|^{\frac{N-1}{p-1}}}$ for $|x|$ large). Thanks to the sharp asymptotic estimates on $u$ and $|\nabla u|$, Sciunzi \cite{BS16} finally proved the classification of positive $D^{1,p}(\R^N)$-weak solution $u$ to \eqref{criticeqution} via the moving planes technique in conjunction with the weighted Poincar\'{e} type inequality and Hardy's inequality and so on. For radial symmetry of positive $D^{1,p}(\R^N)$-weak solutions to $D^{1,p}(\mathbb{R}^{N})$-critical quasi-linear equations with local nonlinearity of type \eqref{criticeqution} via the method of moving planes, c.f. \cite{DLL,LD,DP,LDSMLMSB,DLPFRM,LDMR,LDBS04,LPLDHT,OSV,BS16,VJ16} and the references therein.

\smallskip

Ciraolo, Figalli and Roncoroni \cite{CFR} obtained classification of positive $D^{1,p}(\R^N)$-weak solutions to $D^{1,p}(\R^N)$-critical anisotropic $p$-Laplacian equations of type \eqref{criticeqution} in general convex cones. Dai, Gui and Luo \cite{DGL} classified weak solutions with finite total mass to anisotropic Finsler $N$-Laplacian Liouville equations in general convex cones and hence extended the classification results in \cite{CFR} from $1<p<N$ to the limiting case $p=N$. Under some energy growth conditions or suitable control of the solutions at $\infty$ in some cases, Catino, Monticelli and Roncoroni \cite{CMR} established classification of general positive solutions to \eqref{criticeqution} (possibly having infinite $D^{1,p}(\R^N)$-energy) for $1<p<N$. For $\frac{N+1}{3}\leq p<N$, Ou \cite{Ou} classified general positive solutions to \eqref{criticeqution} (possibly having infinite $D^{1,p}(\R^N)$-energy). It should be noted that, for our nonlocal quasi-linear equation \eqref{eq1.1}, the finite $D^{1,p}(\R^N)$-energy assumption that $u\in D^{1,p}(\R^N)$ is natural and necessary, since we need it to guarantee that $V_{1}(x):=|x|^{-2p}\ast u^{p}\in L^{\frac{N}{p}}(\mathbb{R}^{N})$ and hence $V_{1}(x)$ is finite and well-defined almost everywhere. For classification of solutions to weighted $D^{1,p}(\R^N)$-critical $p$-Laplace equation \eqref{eq-cri-har-0}, refer to \cite{CY,LPLDHT}. In \cite{LM}, Lin and Ma obtained the classification of positive solutions for weighted $p$-Laplace equation of the type \eqref{eq-cri-har-0} with singular weight on partial variables, and hence derived the best constant and extremal functions for a class Hardy-Sobolev-Maz'ya inequalities. For more literatures on quantitative and qualitative properties of solutions for quasi-linear equations involving $p$-Laplacians with local nonlinearities, please c.f. \cite{A,AY,B,CDPSYS,CV,DLL,LD,DFSV,DP,LDBS04,LDBS,DYZ,ED83,Di,DPZZ,FSV,GL,LPLDHT,Lieberman,LM,OSV,BSR14,JSLB,SJZH02,Tei,PT,JLV,ZL,Z} and the references therein.

\medskip

When $1<p<N$ and $s=0$, if $V(x)=|u|^{p^\star-p}$ with $p^\star=\frac{Np}{N-p}$ and nonnegative $u \in D^{1,p}(\R^N)$, the generalized equation \eqref{gen-eq} becomes the following doubly $D^{1,p}(\R^N)$-critical $p$-Laplace equation with Hardy potential:
\begin{align}\label{criticeqution-hardy}
\left\{ \begin{array}{ll} \displaystyle
-\Delta_p u - \frac{\mu}{|x|^p} |u|^{p-2}u=|u|^{p^{\star}-2}u  \quad\,\,\,\,&\mbox{in}\,\, \R^N, \\ \\
u \in D^{1,p}(\R^N), \quad   u\geq 0\,\,  \mbox{in}\,\, \R^N,&
\end{array}
\right.\hspace{1cm}
\end{align}
where $0\leq\mu< \bar{\mu}=\left( (N-p)/p \right)^p$.
By employing the weak comparison principle, De Giorgi-Moser-Nash iteration technique and the Caccioppoli inequality, Xiang \cite{XCL,XCL2} obtained the upper and lower bound asymptotic behavior of the positive solution $u\in D^{1,p}(\R^N)$ for equation \eqref{criticeqution-hardy} and the upper bound asymptotic estimates of $|\nabla u|$ near the origin and at infinity. Using the scaling arguments, some asymptotic estimates proved in \cite{XCL,XCL2} and the strictly radial decreasing property of positive solution $u\in C_{loc}^{1,\alpha}(\R^N\setminus\{0\})$ to the following equation
\begin{equation}\label{eq-p-h-c}
-\Delta_p u - \frac{\mu}{|x|^p}|u|^{p-2}u= 0 \,\,\,\,\,\  \mbox{in}\,\ \R^N\setminus\{0\}
\end{equation}
under suitable asymptotic conditions for $u$ near the origin and at infinity, Oliva, Sciunzi and Vaira \cite{OSV} proved the sharp lower bound on the decay rate of $|\nabla u|$ to equation \eqref{criticeqution-hardy} at infinity, and hence the radial symmetry of positive solutions to equation \eqref{criticeqution-hardy} in the case $1<p<N$ and $0\leq\mu< \bar{\mu}$ via a refined method of moving planes. Recently, Esposito, Montoro, Sciunzi and Vuono \cite{EMSV} derived the complete asymptotic estimates for the anisotropic doubly $D^{1,p}(\R^N)$-critical equation, by using the method in \cite{OSV,XCL,XCL2} and the uniqueness of positive solution $u\in C_{loc}^{1,\alpha}(\R^N\setminus\{0\})$ to the equation \eqref{eq-p-h-c} under suitable asymptotic conditions for $u$ and $|\nabla u|$ in $\R^N \setminus\{0\}$.

\smallskip

When $\mu=0$ and $s=0$, if $V(x)=|x|^{-2p}\ast |u|^p$ with nonnegative $u \in D^{1,p}(\R^N)$, the generalized equation \eqref{gen-eq} becomes the following $D^{1,p}(\R^N)$-critical quasi-linear nonlocal equation:
\begin{equation}\label{eq-cri-hartree}
 -\Delta_p u =\left(|x|^{-2p}\ast u^{p}\right)u^{p-1},\,u\geq0, \quad\,\mbox{in}\, \R^N,\quad u \in D^{1,p}(\R^N),
\end{equation}
where the full range of $p$ is $1<p<\frac{N}{2}$ and $N\geq3$. In \cite{DLL}, Dai, Li and Liu first proved the sharp asymptotic estimates, radial symmetry and strictly radial monotonicity of positive solutions to \eqref{eq-cri-hartree}. Since \eqref{eq-cri-hartree} is a typical special case of equation \eqref{eq1.1} with $\mu=0$,inspired by \cite{DLL,OSV}, in this paper we will establish the sharp asymptotic estimates, the radial symmetry and strictly radial monotonicity of positive solution to the doubly energy-critical quasi-linear nonlocal problem \eqref{eq1.1}. Our results can be regarded as the nonlocal counterpart of \cite{OSV} and the extension of \cite{DLL} from $\mu=0$ to general case $0\leq\mu<\bar{\mu}:=\left(\frac{N-p}{p}\right)^p$.

\medskip

First, inspired by \cite{CDPSYS,DLL,SJZH02,XCL}, using the De Giorgi-Moser-Nash iteration method, we can derive the following regularity and a better preliminary estimates on upper bounds of asymptotic behaviors of $u\in D^{1,p}(\R^N)$ to the generalized doubly $D^{1,p}$-critical quasi-linear equation \eqref{gen-eq} near the origin as well as at infinity.
\begin{thm}[Local integrability, boundedness and regularity of solutions to \eqref{gen-eq}]\label{th2.1.1}
Assume $0\leq\mu<\bar{\mu}$, $1<p<N$, $0\leq s<p$, $0\leq V\in L^{\frac{N}{p-s}}(\mathbb{R}^{N})$ and let $u$ be a nonnegative $D^{1,p}(\R^{N})$-weak solution to the generalized equation \eqref{gen-eq}. Then, we have \\
{\bf $(i)$} $u\in L^{r}_{loc}(\mathbb{R}^{N}\setminus\{0\})$ for any $0<r<+\infty$; \\
{\bf $(ii)$} for any $\bar{p}\geq p^\star$, and for any $\rho_{1}\in(0,1)$ such that $\|V\|_{L^\frac{N}{p-s} (B_{\rho_{1}}(0))}\leq \epsilon_1$ and $\|V\|_{L^\frac{N}{p-s}\Big(\R^N\setminus B_{\frac{1}{\rho_1}}(0)\Big)} \leq \epsilon_1$, there hold
\begin{equation}\label{loc-est-3+}
  \|u\|_{L^{\bar{p}}(B_{4R}(0)\setminus B_{R/4}(0))} \leq \frac{C}{R^{\frac{N-p}{p}-\frac{N}{\bar{p}}}} \|u\|_{L^{p^\star}(B_{8R}(0)\setminus B_{R/8}(0))}, \,\,\,\quad \forall \,\, R\leq \frac{\rho_1}{8},
\end{equation}
and
\begin{equation}\label{loc-est-4+}
  \|u\|_{L^{\bar{p}}(\R^N\setminus B_{2R}(0))} \leq \frac{C}{R^{\frac{N-p}{p}-\frac{N}{\bar{p}}}} \|u\|_{L^{p^\star}(\R^N\setminus B_{R}(0))}, \,\,\,\quad \forall \,\, R\geq \frac{1}{\rho_1},
\end{equation}
where $C=C(N,p,s,\mu)>0$ and $\epsilon_1>0$ is given by Lemma \ref{lm:l-b-2};\\
{\bf $(iii)$} for any $\rho_{2}\in(0,1)$ such that $\|V\|_{L^\frac{N}{p-s} (B_{\rho_{2}}(0))}\leq \epsilon_2$ and $\|V\|_{L^\frac{N}{p-s} (\R^N\setminus B_{1/{\rho_{2}}}(0))} \leq \epsilon_2$, there hold
\begin{equation}\label{loc-est-5+}
  \|u\|_{L^{p^\star}(B_R (0))} \leq C R^{\tau_1}, \,\,\,\quad \forall \,\, R \leq \rho_2,
\end{equation}
and
\begin{equation}\label{loc-est-6+}
  \|u\|_{L^{p^\star}(\R^N\setminus B_{R}(0))} \leq C R^{-\tau_2}, \,\,\,\quad \forall \,\, R\geq \frac{1}{\rho_2},
\end{equation}
where $C=C(N,\mu,p,s,\rho_2,\|u\|_{L^{p^\star}(\R^N)})>0$, $\tau_i =\tau_i(N,p,s,\mu)>0$ ($i=1,2$) and $\epsilon_2>0$ is given by Lemma \ref{lm:l-b-3}.

\smallskip

Moreover, assume further that $V\in L^{q}_{loc}(\mathbb{R}^{N}\setminus \{0\})$ for some $\frac{N}{p}<q<+\infty$ when $s=0$ (no further assumption if $0<s<p$),
then $u\in C^{1,\alpha}(\R^N\setminus \{0\})\cap L^{\infty}_{loc}(\mathbb{R}^{N}\setminus \{0\})$ for some $0<\alpha<\min\{1,\frac{1}{p-1}\}$.
\end{thm}

\begin{thm}[Preliminary estimates of solutions to \eqref{gen-eq}]\label{th2.1++}
Let $u\in D^{1,p}(\R^N)$ be a nonnegative weak solution of the general equation \eqref{gen-eq} with $0\leq \mu < \bar{\mu}$, $1<p<N$, $0\leq s < p$ and $0\leq V(x)\in L^\frac{N}{p-s}(\R^N)$. Let $\rho>0$ be small enough such that $ \|V\|_{L^\frac{N}{p-s} (B_{\rho}(0))}\leq \min\{\epsilon_1,\epsilon_2\}$ and $\|V\|_{L^\frac{N}{p-s} (\R^N\setminus B_{1/{\rho}}(0))} \leq \min\{\epsilon_1,\epsilon_2\}$, where $\epsilon_1>0$ and $\epsilon_2>0$ are given by Lemmas \ref{lm:l-b-2} and \ref{lm:l-b-3} respectively. If $s=0$, assume that $\rho$ is smaller if necessary such that, there exists some $l\in(0,p)$ satisfying
\begin{equation}\label{est-bdd-1+}
\|V\|_{L^{\frac{N}{p-l}} (B_{2r}(y))} \leq \widetilde{C} r^{-l}, \,\,\,\,\,\,\,\,\,\, \forall \,\, |y|\leq \frac{\rho}{16} \,\,\,\,\, \mbox{or} \,\,\,\,\, |y|\geq \frac{16}{\rho},
\end{equation}
where $r=\min\left\{1,\frac{|y|}{4}\right\}$ and the positive constant $\widetilde{C}$ is independent of $y$ (no assumption is needed when $0<s<p$). Then, there exists a positive constant $C=C(N,p,\mu,s,\rho,u,\tilde{C})>0$ such that
$$ u(y) \leq C |y|^{-\frac{N-p}{p}+\tau_1}, \,\,\,\,\,\,\,\, \forall \ |y|\leq \frac{\rho}{16},$$
and
$$ u(y) \leq C |y|^{-\frac{N-p}{p}-\frac{\tau_2}{2}}, \,\,\,\,\,\,\,\, \forall \ |y|\geq \frac{16}{\rho},$$
where $\tau_1, \tau_2 >0$ are given in Lemma \ref{lm:l-b-3}.
\end{thm}

In order to prove the radial symmetry and strictly radial monotonicity of positive solutions to the generalized equation \eqref{gen-eq}, one of the key ingredients is to prove the sharp asymptotic estimates for positive $D^{1,p}(\R^N)$-weak solution $u$ and $|\nabla u|$ near the origin and at infinity. In fact, using the arguments in \cite{DLL,EMSV,BS16,SJZH02,VJ16,XCL}, by scaling arguments, the asymptotic behaviors of $u$, the regularity results in DiBenedetto \cite{ED83} and Tolksdorf \cite{PT}, and the uniqueness up to the equation \eqref{eq-p-h-c} under suitable asymptotic conditions for $u$ and $|\nabla u|$ in $\R^N\setminus \{0\}$ (see \cite[Theorem 1.3]{EMSV}), we can deduce the following complete sharp asymptotic estimates on both upper and lower bounds for $u$ and $|\nabla u|$ provided that $V(x)$ and $u$ satisfy a better preliminary asymptotic estimates at the origin and at infinity (guaranteed by Theorem \ref{th2.1++}).
\begin{thm}[Sharp asymptotic estimates of solutions to \eqref{gen-eq}]\label{th2.1}
Assume that $0\leq\mu<\bar{\mu}$, $1<p<N$, $0\leq s<p$, $0\leq V\in L^{\frac{N}{p-s}}(\mathbb{R}^{N})$, and let $u$ be a nontrivial nonnegative $D^{1,p}(\R^{N})$-weak solution to the generalized equation \eqref{gen-eq}. Suppose further that $V(x) \leq C_{V,1} |x|^{-\beta_1+s}$ for $|x|$ small and some $C_{V,1}>0$, $\beta_1<p$ provided that $u(x)\leq C|x|^{-\alpha_1}$ for some $C>0$, $\alpha_1<\frac{N-p}{p}$ and all $|x|$ small, and $V(x) \leq C_{V,2} |x|^{-\beta_2+s}$ for $|x|$ large and some $C_{V,2}>0$, $\beta_2>p$ provided that $u(x)\leq C|x|^{-\alpha_2}$ for some $C>0$, $\alpha_2>\frac{N-p}{p}$ and all $|x|$ large. Then there exist positive constants $0<R_0<1<R_1$ depending on $N,p,\mu,s,V$ and $u$, such that
\begin{equation}\label{eq0806}
   c_0 |x|^{-\gamma_1} \leq u(x) \leq C_0 |x|^{-\gamma_1}, \qquad \forall \,\,\,  |x| < R_0,
\end{equation}
\begin{equation}\label{eq0806++}
 c_0 |x|^{-\gamma_2} \leq u(x) \leq C_0 |x|^{-\gamma_2}, \qquad \forall \,\,\,  |x| > R_1,
\end{equation}
\begin{align}\label{eq0806+++}
c_0 |x|^{-\gamma_1-1} \leq |\nabla u(x)| \leq C_0 |x|^{-\gamma_1-1}, \qquad \forall \,\,\, |x| < R_0,
\end{align}
\begin{align}\label{eq0806+}
c_0 |x|^{-\gamma_2-1} \leq |\nabla u(x)| \leq C_0 |x|^{-\gamma_2-1}, \qquad \forall \,\,\, |x| > R_1,
\end{align}
where $\gamma_1<\gamma_2$ are two different roots of
\[\gamma^{p-2}[(p-1)\gamma^2-(N-p)\gamma]+\mu=0,\]
$C_0$ is positive constant depending on $N,p,\mu,s,V$ and $u$, $c_0$ is positive constant depending on $N,p,\mu,s,V,R_0,R_1$ and $u$.
\end{thm}
\begin{rem}
In particular, if $\mu=0$, then $\gamma_1=0$ and $\gamma_2=\frac{N-p}{p-1}$ in Theorem \ref{th2.1}, and hence \cite[Theorem 1.3]{DLL} can be deduced as a corollary from Theorem \ref{th2.1}. Moreover, it is easy to see that
\[0\leq \gamma_1 < \frac{N-p}{p} < \gamma_2\leq \frac{N-p}{p-1}.\]
\end{rem}

\begin{rem}\label{loc-u-rem}
When $V(x)=u^{p_{s}-p}(x)$ with the critical Hardy-Sobolev exponent $p_s:=\frac{p(N-s)}{N-p}$, the general doubly $D^{1,p}(\R^{N})$-critical quasi-linear equation \eqref{gen-eq} degenerates into the following doubly $D^{1,p}(\R^{N})$-critical quasi-linear equation with Hardy potentials and local $D^{1,p}$-critical Hardy-Sobolev nonlinearity:
\begin{align}\label{eq1.2}
\left\{ \begin{array}{ll} \displaystyle
-\Delta_p u - \frac{\mu}{|x|^p} u^{p-1}=|x|^{-s} u^{p_s-1}  \quad\,\,\,\,&\mbox{in}\,\, \R^N, \\ \\
u \in D^{1,p}(\R^N),\quad\,\,\,\,\,\,u\geq0 \,\, \mbox{in}\,\, \R^N,&
\end{array}
\right.\hspace{1cm}
\end{align}
where $N\geq2$, $1<p<N$, $0\leq\mu<\bar{\mu}:=\left(\frac{N-p}{p}\right)^p$ and $0\leq s<p$. When $s=0$, one has $p_{s}=p^{\star}$.

\smallskip

Define $V_2(x):=u^{p_s-p}(x)$ with $p_s-p=p(p-s)/(N-p)$. We will show that $V_2(x)=u^{p_s-p}(x)$ satisfies all the assumptions on $V$ in Theorems \ref{th2.1.1}, \ref{th2.1++} and \ref{th2.1}.

\smallskip

In fact, it follows easily from $u\in D^{1,p}(\R^N)$ and the Sobolev embedding that $V_2(x)\in L^{\frac{N}{p-s}}(\R^N)$. Thus Theorem \ref{th2.1.1} implies that $u\in C^{1,\alpha}(\R^N\setminus \{0\})\cap L^{\infty}_{loc}(\mathbb{R}^{N}\setminus \{0\})$ for some $0<\alpha<\min\{1,\frac{1}{p-1}\}$. Let $\rho>0$ be small enough such that $ \|V_{2}\|_{L^\frac{N}{p-s} (B_{\rho}(0))}\leq \min\{\epsilon_1,\epsilon_2\}$ and $\|V_{2}\|_{L^\frac{N}{p-s} (\R^N\setminus B_{1/{\rho}}(0))} \leq \min\{\epsilon_1,\epsilon_2\}$, where $\epsilon_1>0$ and $\epsilon_2>0$ are given by Lemmas \ref{lm:l-b-2} and \ref{lm:l-b-3} respectively. When $s=0$, we will show that condition \eqref{est-bdd-1+} in Theorem \ref{th2.1++} holds. Indeed, fixing any $l\in(0,p)$, taking $R=|y|/4$ and using $(ii)$ in Theorem \ref{th2.1.1} with $\bar{p}=\frac{Np^2}{(N-p)(p-l)}=\frac{N(p^{\star}-p)}{p-l}$, we deduce that, there exists a constant $C>0$ independent of $y$ such that for any $|y|\leq \rho/16$ or $|y|\geq 16/\rho$
\begin{equation}\label{060107}
\begin{aligned}
     &\quad \|V_{2}\|_{L^{\frac{N}{p-l}}(B_{2R}(y))}=\|u^{p^{\star}-p}\|_{L^{\frac{N}{p-l}}(B_{2R}(y))} \leq \|u^{p^{\star}-p}\|_{L^{\frac{N}{p-l}}(B_{8R}(0)\setminus B_{R/8}(0))} \\
     &=\|u\|^{p^{\star} -p}_{L^{(p^{\star} -p)\cdot\frac{N}{p-l}}(B_{8R}(0)\setminus B_{R/8}(0))}\leq C \left( R^{\frac{N(N-p)(p-l)}{Np^2}-\frac{N-p}{p}} \|u\|_{L^{p^\star}(B_{16R}(0)\setminus B_{R/16}(0))}  \right)^{p^{\star}-p} \\
     &= C R^{-l} \|u\|^{p^{\star}-p}_{L^{p^\star}(B_{16R}(0)\setminus B_{R/16}(0))}\leq CR^{-l}
\end{aligned}
\end{equation}
 i.e., \eqref{est-bdd-1+} holds. Therefore, from Theorem \ref{th2.1++}, we deduce that, there exists a positive constant $C=C(N,p,s,\mu,\rho,u)>0$ such that for $\tau_1, \tau_2>0$ given in Lemma \ref{lm:l-b-3}, there hold
$$ u(y) \leq C |y|^{-\frac{N-p}{p}+\tau_1}, \,\,\,\,\,\,\,\quad \forall \,\, |y|\leq \frac{\rho}{16},$$
and
$$ u(y) \leq C |y|^{-\frac{N-p}{p}-\frac{\tau_2}{2}}, \,\,\,\,\,\,\,\quad \forall \,\, |y|\geq \frac{16}{\rho}.$$
 As a consequence, for $|x|<\frac{\rho}{16}$,
\begin{equation*}
  V_2(x)=u^{p_s-p}(x)\leq C|x|^{-\beta_1} \quad \text{with} \quad \beta_1:=\left(\frac{N-p}{p}-\tau_1\right)(p_s-p)=p-s-\tau_1 (p_s-p)<p-s,
\end{equation*}
and for $|x|>\frac{16}{\rho}$,
\begin{equation*}
  V_2(x)=u^{p_s-p}(x)\leq C|x|^{-\beta_2} \quad \text{with} \quad \beta_2:=\left(\frac{N-p}{p}+\frac{\tau_2}{2}\right)(p_s-p)=p-s+\frac{\tau_2}{2} (p_s-p)>p-s.
\end{equation*}
Consequently, Theorems \ref{th2.1.1} and \ref{th2.1} can be applied to equation \eqref{eq1.2}, and hence the positive weak solution $u$ to \eqref{eq1.2} satisfies the regularity $u\in C^{1,\alpha}(\R^N\setminus \{0\})\cap L^{\infty}_{loc}(\mathbb{R}^{N}\setminus \{0\})$ in Theorem \ref{th2.1.1} and the sharp asymptotic estimates \eqref{eq0806}-\eqref{eq0806+} in Theorem \ref{th2.1}. In particular, when $s=0$, the regularity and sharp asymptotic estimates in \cite{OSV,XCL,XCL2} can be deduced directly from our generalized Theorems \ref{th2.1.1} and \ref{th2.1} as well.
\end{rem}

For equation \eqref{eq1.1},  one can also easily verify that all the assumptions on $V$ and $u$ in Theorem \ref{th2.1} are fulfilled (for details, see the proof of the asymptotic estimates in Theorem \ref{th2} in Subsection \ref{sc3.4}). Thanks to the regularity result and the complete sharp asymptotic estimates of the positive solution $u$ to equation \eqref{eq1.1} and $|\nabla u|$, by using a refined method of moving planes in \cite{OSV}, we can prove the radial symmetry and strictly radial monotonicity of positive $D^{1,p}(\mathbb{R}^{N})$-weak solutions to equation \eqref{eq1.1}. Our result is the following theorem.
\begin{thm}[Sharp asymptotic estimates and radial symmetry of solutions to \eqref{eq1.1}]\label{th2}
Assume that $0\leq\mu<\bar{\mu}$, $1<p<\frac{N}{2}$ and let $u$ be a nontrivial nonnegative $D^{1,p}(\mathbb{R}^{N})$-weak solution of the doubly $D^{1,p}(\mathbb{R}^{N})$-critical nonlocal quasi-linear equation \eqref{eq1.1}. Then $u \in C^{1,\alpha} (\R^N\setminus \{0\}) \cap L_{loc}^\infty(\R^N\setminus \{0\})$ for some $0<\alpha<\min\{1,\frac{1}{p-1}\}$, and the sharp asymptotic estimates \eqref{eq0806}-\eqref{eq0806+} hold for any positive $D^{1,p}(\mathbb{R}^{N})$-weak solution $u$ to \eqref{eq1.1}. Moreover, either $u\equiv0$; or $u>0$ is radially symmetric and strictly decreasing about the origin, i.e., positive solution $u$ must assume the form
\[u(x)=\lambda^{\frac{N-p}{p}}U(\lambda x)\]
for $\lambda:=u(1)^{\frac{p}{N-p}}>0$, where $U(x)=U(r)$ with $r=|x|$ is a positive radial solution to \eqref{eq1.1} with $U(\lambda)=1$ and $U'(r)<0$ for any $r>0$ ($U$ may not be unique).
\end{thm}
\begin{rem}
In fact, by the regularity $u \in C^{1,\alpha} (\R^N\setminus \{0\}) \cap L_{loc}^\infty(\R^N\setminus \{0\})$ in Theorem \ref{th2}, one can easily verify that $|x|^{-2p}\ast u^{p}\in L^{\infty}_{loc}(\mathbb{R}^{N}\setminus \{0\})$ for any nonnegative weak solution $u\in D^{1,p}(\mathbb{R}^{N})$ to \eqref{eq1.1} (see \eqref{eq10.26.30.03}). Thus it follows from the strong maximum principle in Lemma \ref{hopf} that, for any nonnegative $D^{1,p}(\mathbb{R}^{N})$-weak solution $u$ to \eqref{eq1.1}, we have either $u\equiv0$ or $u>0$ in $\mathbb{R}^{N}\setminus \{0\}$.
\end{rem}
\begin{rem}\label{rm.2.1}
Our results can be regarded as the extension of \cite{DLL} from $\mu=0$ to general $0\leq\mu<\bar{\mu}:=\left(\frac{N-p}{p}\right)^p$ and the nonlocal counterpart of \cite{OSV}. Furthermore, the ranges of $N$, $p$ and $\mu$ for equation \eqref{eq1.1} in Theorem \ref{th2} are complete. Thus, in Theorem \ref{th2}, we have derived a complete result on the sharp asymptotic estimates, radial symmetry and strictly radial monotonicity of the positive weak solutions to the doubly $D^{1,p}$-critical quasi-linear nonlocal equation \eqref{eq1.1}.
\end{rem}


For more literatures on the classification results and other quantitative and qualitative properties of solutions for various PDE and IE problems via the methods of moving planes, please refer to \cite{CGS,CD,CDZ,CL,CL1,CLL,CLO,CDQ,DFHQW,DFQ,DL,DQ,LD,DP,LDSMLMSB,DLPFRM,LDMR,LDBS04,DMPS,GNN,LPLDHT,Lin,Liu,MZ,BS16,S,VJ16,WX} and the references therein. For classification of solutions to semi-linear equations on Heisenberg group or CR manifolds via Jerison-Lee identities and invariant tensor techniques, see \cite{MO,MOW} and the references therein.

\subsection{Extensions to more general doubly $D^{1,p}(\R^N)$-critical nonlocal quasi-linear equations}
Through exactly the same arguments, we can also study the following general doubly $D^{1,p}(\R^{N})$-critical nonlocal quasi-linear equation with generalized (weighted) nonlocal nonlinearity and Hardy potential:
\begin{align}\label{eq1.3}
\left\{ \begin{array}{ll} \displaystyle
-\Delta_p u - \frac{\mu}{|x|^p} u^{p-1}=|x|^{-s} \left(|x|^{-\sigma}\ast u^{p_{s,\sigma}}\right)u^{p_{s,\sigma}-1}  \quad\,\,\,\,&\mbox{in}\,\, \R^N, \\ \\
u \in D^{1,p}(\R^N),\quad\,\,\,\,\,\,u\geq0 \,\,  \mbox{in}\,\, \R^N,&
\end{array}
\right.\hspace{1cm}
\end{align}
where $0\leq\mu<\bar{\mu}=\left(\frac{N-p}{p}\right)^p$, $1<p<N$, $0\leq s <p$, $p_{s,\sigma}:=\frac{(2N-\sigma-s)p}{2(N-p)}$, $s<\sigma<N$ and $0<\sigma+s\leq 2p$. Define $V_3(x):=\left(|x|^{-\sigma}\ast u^{p_{s,\sigma}}\right)(x) u^{p_{s,\sigma}-p}(x)$, then the  general nonlocal quasi-linear equation \eqref{eq1.3} is a special case of the generalized equation \eqref{gen-eq}. In the endpoint case $\sigma=2p$, $s=0$ and $p_{\sigma}^{\star}=p$ with $1<p<\frac{N}{2}$, the general nonlocal quasi-linear equation \eqref{eq1.3} degenerates into the doubly $D^{1,p}(\R^{N})$-critical nonlocal $p$-Laplacian static Hartree equation \eqref{eq1.1}.

\medskip

For \eqref{eq1.3}, the finite $D^{1,p}(\R^N)$-energy assumption $u\in D^{1,p}(\R^N)$ is necessary to guarantee $V_{3}(x)=\left(|x|^{-\sigma}\ast u^{p_{s,\sigma}}\right)(x) u^{p_{s,\sigma}-p}(x)\in L^{\frac{N}{p-s}}(\mathbb{R}^{N})$ and hence $V_{3}(x)$ is finite and well-defined almost everywhere. Indeed, by H\"{o}lder's inequality  and Hardy-Littlewood-Sobolev inequality in Lemma \ref{HLSI} with $p=\sigma$, $q=2N/(\sigma-s)$ and $m=2N/(2N-\sigma-s)$, one has
\begin{equation}\label{060108-gen}
\begin{aligned}
     &\quad \|V_3(x)\|_{L^{\frac{N}{p-s}}(\R^N)}
       = \left[ \int_{\R^N} |K(x)|^{\frac{N}{p-s}}  |u(x)|^{(p_{s,\sigma}-p) \cdot \frac{N}{p-s}}  \mathrm{d}x \right]^{\frac{p-s}{N}} \\
     &\leq \left[ \left( \int_{\R^N} |K(x)|^{\frac{N}{p-s}\cdot \frac{2(p-s)}{\sigma-s}} \mathrm{d}x \right)^{\frac{\sigma-s}{2(p-s)}} \left(\int_{\R^N} |u(x)|^{\frac{2p^2-(\sigma+s)p}{2(N-p)} \cdot \frac{N}{p-s} \cdot \frac{2(p-s)}{2p-(\sigma+s)}} \mathrm{d}x \right)^{\frac{2p-(\sigma+s)}{2(p-s)}} \right]^{\frac{p-s}{N}} \\
     &\leq \left( \int_{\R^N} |K(x)|^{\frac{2N}{\sigma-s}} \mathrm{d}x \right)^{\frac{\sigma-s}{2N}} \left\|u \right\|_{L^{p^\star}(\R^N)}^{\frac{p(2p-\sigma-s)}{2(N-p)}}\\ &\leq \bar{C} \|u^{p_{s,\sigma}}\|_{L^{\frac{2N}{2N-\sigma-s}}(\R^N)} \left\|u \right\|_{L^{p^\star}(\R^N)}^{p_{s,\sigma}-p} = \bar{C} \left\|u \right\|_{L^{p^\star}(\R^N)}^{2p_{s,\sigma}-p}, \qquad \forall \,\, u\in D^{1,p}(\mathbb{R}^{N}),
\end{aligned}
\end{equation}
where $K(x):=\left(|x|^{-\sigma}\ast u^{p_{s,\sigma}}\right)(x)$. Thus it follows from Sobolev embedding inequality and double weighted Hardy-Littlewood-Sobolev inequality in Theorem \ref{HLSI+} with $\alpha=s$, $\beta=0$, $\sigma=\sigma$ and $r=t=(2N)/(2N-s-\sigma)$ that
\begin{equation}\label{b1}
  \|u\|_{V_{3}}:=\left(\int_{\mathbb{R}^{N}}\int_{\mathbb{R}^{N}}\frac{|u|^{p_{s,\sigma}}(x)|u|^{p_{s,\sigma}}(y)}{|x|^{s}|x-y|^{\sigma}}\mathrm{d}x\mathrm{d}y\right)^{\frac{1}{2p_{s,\sigma}}}
  \leq C_{N,p,s,\sigma}\| \nabla u \|_{L^p(\R^N)},  \quad\, \forall \,\, u\in D^{1,p}(\mathbb{R}^{N}),
\end{equation}
where $C_{N,p,\sigma}$ is the best constant for the double weighted Hardy-Littlewood-Sobolev inequality. The general $D^{1,p}(\mathbb{R}^{N})$-critical nonlocal quasi-linear equation \eqref{eq1.3} is the Euler-Lagrange equation for the following minimization problem
\begin{equation}\label{gmini}
  C_{N,p,\mu,s,\sigma}:=\inf\left\{\left(\| \nabla u \|^p_{L^p(\R^N)}-\mu\int_{\R^N}\frac{1}{|x|^p}u^{p}\mathrm{d}x\right)^{\frac{1}{p}} \,\Big|\,u\in D^{1,p}(\mathbb{R}^{N}), \,\, \|u\|_{V_{3}}=1\right\},
\end{equation}
where $0\leq\mu<\bar{\mu}$. The general doubly $D^{1,p}(\R^{N})$-critical nonlocal quasi-linear equation \eqref{eq1.3} is the Euler-Lagrange equation for the minimization problem \eqref{gmini}.

\medskip

By using the double weighted Hardy-Littlewood-Sobolev inequality instead of Hardy-Littlewood-Sobolev inequality, through entirely similar arguments as Su and Chen \cite{SYCH}, one can also derive that the minimization problem \eqref{gmini} and hence the best constant $C_{N,p,\mu,s,\sigma}$ in \eqref{b1} can be attained by a radial symmetric, non-increasing and nontrivial nonnegative function in $D^{1,p}(\R^N)$. Consequently, the general nonlocal quasi-linear equation \eqref{eq1.3} possesses (at least) a nontrivial nonnegative weak solution $u\in D^{1,p}(\mathbb{R}^{N})$.

\medskip

We aim to prove sharp asymptotic estimates and radial symmetry for positive weak solutions to \eqref{eq1.3} and derive the following theorem.
\begin{thm}[Sharp asymptotic estimates and radial symmetry of solutions to \eqref{eq1.3}]\label{gth2}
Let $u$ be a nonnegative $D^{1,p}(\mathbb{R}^{N})$-weak solution to the general doubly $D^{1,p}(\R^N)$-critical quasi-linear nonlocal equation \eqref{eq1.3}. Then either $u\equiv0$ or $u>0$ in $\mathbb{R}^{N}$, $u \in C^{1,\alpha} (\R^N\setminus\{0\}) \cap L^\infty(\R^N\setminus\{0\})$ for some $0<\alpha<\min\{1,\frac{1}{p-1}\}$, and the sharp asymptotic estimates \eqref{eq0806}-\eqref{eq0806+} on $u$ and $|\nabla u|$ hold for any positive $D^{1,p}(\mathbb{R}^{N})$-weak solution $u$ to \eqref{eq1.3}. Moreover, assume further that $\sigma+s\geq2$ or $\sigma+s<2$ and $\left(\frac{N-1}{p_{s,\sigma}}\right)^{p-1}\left[(p-1)\frac{N-1}{p_{s,\sigma}}-(N-p)\right]+\mu<0$, then any positive solution $u$ must be radially symmetric and strictly decreasing about the origin, i.e., positive solution $u$ must assume the form
\[u(x)=\lambda^{\frac{N-p}{p}}U(\lambda x)\]
for $\lambda:=u(1)^{\frac{p}{N-p}}>0$, where $U(x)=U(r)$ with $r=|x|$ is a positive radial solution to \eqref{eq1.3} with $U(\lambda)=1$ and $U'(r)<0$ for any $r>0$ ($U$ may not be unique).
\end{thm}


\begin{rem}\label{rem1}
One should note that, in Theorem \ref{gth2}, the assumption
\[\sigma+s\geq2, \quad\,\,\, \text{or} \quad\,\,\, \sigma+s<2 \,\,\, \text{and} \,\,\, \left(\frac{N-1}{p_{s,\sigma}}\right)^{p-1}\left[(p-1)\frac{N-1}{p_{s,\sigma}}-(N-p)\right]+\mu<0\]
is equivalent to $p_{s,\sigma}\gamma_1+1<N$.
\end{rem}

\smallskip

The rest of our paper is organized as follows. In Section \ref{sc2}, we will recall some preliminaries on important inequalities, (strong) comparison principles, strong maximum principle and H\"{o}pf's Lemma, related regularity results and properties of the critical set etc. existing in literature. Section \ref{sc3} is devoted to the proof of the key ingredients: Theorems \ref{th2.1.1}, \ref{th2.1++} and \ref{th2.1}, and the asymptotic estimates in Theorem \ref{th2}, i.e., the sharp asymptotic estimates for positive weak solution $u$ and its gradient $|\nabla u|$ to the generalized equation \eqref{gen-eq} and the
nonlocal equation \eqref{eq1.1}. The proof of the radial symmetry results of Theorem \ref{th2} will be carried out in Section \ref{sc4}. Finally, in Section \ref{sc5}, we prove Theorem \ref{gth2}.

\smallskip

In the following, we will use the notation $B_{R}$ to denote the ball $B_{R}(0)$ centered at $0$ with radius $R$, and use $C$ to denote a general positive constant that may depend on $N$, $p$, $\mu$, $s$, $\sigma$ and $u$, and whose value may differ from line to line.

\section{Preliminaries}\label{sc2}
The aim of this section is to recall some useful tools will be used in our proofs of Theorems \ref{th2.1} and \ref{th2}, including some important inequalities, (strong) comparison principles, strong maximum principle and H\"{o}pf's Lemma, related regularity results and properties of the critical set and so on. For clarity of presentation, we will split this section into two sub-sections.

\subsection{Some key inequalities}\label{sc2.1}

We will use the following Hardy-Littlewood-Sobolev, double weighted Hardy-Littlewood-Sobolev and Caffarelli-Kohn-Nirenberg inequalities.
\begin{thm}[Hardy-Littlewood-Sobolev inequality, c.f. e.g. \cite{FL1,FL2,Lieb,Lions2}]\label{HLSI}
Assume $t,r>1$ and $0<p<N$ with $\frac{1}{t}+\frac{p}{N}+\frac{1}{r}=2$. Let $f\in L^t(\R^N)$ and $h\in L^r(\R^N)$. Then there exists a sharp constant $C(N,p,r,t)$, independent of $f$ and $h$, such that
\begin{align*}
\left| \int_{\R^N} \int_{\R^N} \frac{f(x) h(y)}{|x-y|^{p}}\mathrm{d}x\mathrm{d}y \right|\leq C(N,p,r,t)  \| f \|_{L^t(\R^N)} \| h \|_{L^r(\R^N)}.
\end{align*}
Moreover, if $\frac{N}{q}=\frac{N}{m}-N+p$ and $1<m<q<\infty$, we have
\begin{align*}
\left\| \int_{\R^N} \frac{g(x)}{|x-y|^{p}}\mathrm{d}x \right\|_{L^q(\R^N)}\leq \overline{C}(N,p,q,m)  \| g \|_{L^m(\R^N)}
\end{align*}
for all $g\in L^m(\R^N)$.
\end{thm}
Hardy and Littlewood also introduced the following double weighted inequality, which was later generalized by Stein and Weiss in \cite{SMGW}, see also \cite{CL2}.
\begin{thm}[Double weighted Hardy-Littlewood-Sobolev inequality, c.f. \cite{SMGW}, see also \cite{CL2}]\label{HLSI+}
Assume $\alpha, \beta\geq0$, $1<t,r<\infty$ and $0<\sigma<N$. Let $f\in L^t(\R^N)$ and $h\in L^r(\R^N)$. Then there exists a sharp constant $C(N,\sigma,\alpha,\beta,r,t)$, independent of $f$ and $h$, such that
\begin{align*}
\left| \int_{\R^N} \int_{\R^N} \frac{f(x) h(y)}{|x|^{\alpha} |x-y|^{\sigma} |y|^{\beta}}\mathrm{d}x\mathrm{d}y \right|\leq C(N,\sigma,\alpha,\beta,r,t)  \| f \|_{L^t(\R^N)} \| h \|_{L^r(\R^N)},
\end{align*}
where $\alpha+\beta\geq0$,
$$ 1-\frac{1}{r} -\frac{\sigma}{N} < \frac{\alpha}{N} <1-\frac{1}{r} \,\,\quad \mbox{and} \,\,\quad \frac{1}{t}+\frac{\sigma+\alpha+\beta}{N}+\frac{1}{r}=2.$$
\end{thm}

\begin{thm}[Caffarelli-Kohn-Nirenberg inequality, c.f. \cite{CKN}, see also \cite{NHMSM}]\label{HDSI}
Let $q\geq1$, $m\geq1$, $\tau>0$, $0<a\leq1$, $\gamma, \alpha, \beta \in\R$ be such that
$$ \frac{1}{\tau} + \frac{\gamma}{N} = a \left(\frac{1}{q} + \frac{\alpha-1}{N}\right) + (1-a)\left( \frac{1}{m} +\frac{\beta}{N} \right),$$
and, with $\gamma=a\sigma+(1-a)\beta$,
$$0\leq\alpha-\sigma$$
and
\begin{equation*}
\alpha-\sigma\leq1 \,\,\,\ \mbox{if} \,\,\ a>0 \,\ \mbox{and} \,\,\ \frac{1}{\tau} + \frac{\gamma}{N} = \frac{1}{q} + \frac{\alpha-1}{N}.
\end{equation*}
We have, for $u\in C_{c}^1(\R^N)$, \\
$(i)$ if $\frac{1}{\tau} + \frac{\gamma}{N}>0$, then
$$ \left(\int_{\R^N} |u|^\tau |x|^{\gamma\tau} \mathrm{d}x\right)^{1/\tau} \leq C_H \| |x|^\alpha \nabla u\|^a_{L^q(\R^N)}  \| |x|^\beta  u\|^{1-a}_{L^m(\R^N)};$$
$(ii)$ if $\frac{1}{\tau} + \frac{\gamma}{N}<0$ and $supp(u) \subset \R^N\setminus\{0\}$, then
$$ \left(\int_{\R^N} |u|^\tau |x|^{\gamma\tau} \mathrm{d}x\right)^{1/\tau} \leq C_H \| |x|^\alpha \nabla u\|^a_{L^q(\R^N)} \| |x|^\beta  u\|^{1-a}_{L^m(\R^N)},$$
where $C_H$ is a positive constant independent of $u$.
\end{thm}
\begin{rem}\label{re-HDSI}
By taking $\alpha=\beta=0$, $a=1$ and $q=p$ in $(i)$ of Theorem \ref{HDSI}, and using the density of $C_0^\infty (\R^N)$ in $D^{1,p}(\R^N)$ and Fatou's Lemma, we can obtain that
$$ \left(\int_{\R^N} |u|^\tau |x|^{\gamma\tau} \mathrm{d}x\right)^{1/\tau} \leq C_H \| \nabla u\|_{L^p(\R^N)}, \,\,\,\,\,\,\,\  \forall \ u\in D^{1,p}(\R^N),$$
where $\frac{1}{\tau} + \frac{\gamma}{N}=\frac{N-p}{Np}$ and $-1\leq\gamma\leq0$, i.e., Hardy-Sobolev inequality \eqref{eq-HDSI}.
\end{rem}

Now, we will study the following basic point-wise estimates for the convolution term of the type $|x|^{-\nu}*g$.
\begin{lem}\label{lm.6}
Let $0<\nu<N$ and $R>1>R_0>0$, then the following statements concerning
the convolution $|x|^\nu*g$ hold.\\
{\bf $(i)$} If $g(x)\geq0$ in $\mathbb{R}^{N}$ and $g(x)\geq C_1 |x|^{-\beta_0}$ in $\R^N\setminus B_R(0)$ for some $C_1,\beta_0>0$, then for some constant $C>0$, we have
\begin{align}\label{eq10.26.18}
\left\{ \begin{array}{ll}
|x|^{-\nu}*g=\infty,  \,\,\,\,\,&\mbox{if}\,\, \beta_0\leq N-\nu \\
|x|^{-\nu}*g\geq C |x|^{N-\nu-\beta_0},  \,\,\,\,&\mbox{if}\,\, \beta_0>N-\nu
\end{array}
\right. \,\,\,\,\,\,\, &\mbox{in}\,\, \R^N\setminus B_R(0).
\end{align}
{\bf $(ii)$} If $0 \leq g(x) \leq C_1 |x|^{-\beta_1}$ in $B_{R_0}(0)$ and $ 0\leq g(x) \leq \frac{C_2}{1+|x|^{\beta_2}}$ in $\R^N\setminus B_{R_0}(0)$ for some $C_1, C_2 >0$, $\beta_1<N$ and $\beta_2>N-\nu$, then
\begin{align}\label{eq10.26.19}
|x|^{-\nu}*g  \leq C |x|^{-\nu} + C |x|^{N-\nu-\beta_2} \,\,\,\,\,\,\,\mbox{in}\,\, \R^N\setminus B_{2R}(0),
\end{align}
\begin{align}\label{eq10.26.19+}
|x|^{-\nu}*g \leq C |x|^{N-\nu-\beta_1}+C \,\,\,\,\,\,\mbox{in}\,\, B_{R_0/2}(0),
\end{align}
\begin{align}\label{eq10.26.19++}
|x|^{-\nu}*g \leq \bar{C} \,\,\,\,\,\,\mbox{in}\,\, B_{3R}(0)\setminus B_{R_0/4}(0)
\end{align}
for some constants $C=C(R_0)>0$ and $\bar{C}=\bar{C}(R,R_0)>0$.\\
{\bf $(iii)$} 
If $g(x)\geq0$  in $\mathbb{R}^{N}$ and $g(x)\geq C_{1} |x|^{-\beta_0}$ in $\R^N \setminus B_{R}(0)$ for some $\beta_0>N-\nu$ and constant $C_{1}>0$,  then
\begin{equation}\label{eq10.26.30}
|x|^{-\nu}*g\geq \widetilde{C} \,\,\mbox{in}\,\, B_{R/2}(0)
\end{equation}
for some constant $\widetilde{C}=\widetilde{C}(R)>0$.
\end{lem}
\begin{proof}
{\bf $(i)$} Since $|x-y|\leq |x|+|y| \leq 2|y|$ if $|x|\leq |y|$, we get, for $x\in\R^N\setminus B_R(0)$,
$$
|x|^{-\nu}*g \geq C_1 \int_{|x|\leq |y|} \frac{|y|^{-\beta_0}}{|x-y|^{\nu}} \mathrm{d}y \geq C \int_{|x|\leq |y|} |y|^{-\beta_0-\nu}\mathrm{d}y \geq C \int_{|x|}^\infty \tau^{N-1-\beta_0-\nu} \mathrm{d}\tau,
$$
which leads to the conclusion in $(i)$.

\smallskip

{\bf $(ii)$} Since $N>\beta_1$, $0<\nu<N$ and $\beta_2>N-\nu$, we have, for $|x|\geq 2R>2$,
\begin{align}\label{eq10.28.30}
&\int_{\R^N} \frac{g(y)}{|x-y|^{\nu}}\mathrm{d}y =\left\{\int_{|y|\leq\frac{|x|}{2}} + \int_{\frac{|x|}{2} \leq |y| \leq 2|x|} +\int_{|y|\geq 2|x|}  \right\} \frac{g(y)}{|x-y|^{\nu}}\mathrm{d}y \\
&\leq C |x|^{-\nu}\int_{|y|\leq\frac{|x|}{2}} g(y) \mathrm{d}y + C_2 |x|^{-\beta_2} \int_{\frac{|x|}{2} \leq |y| \leq 2|x|} \frac{\mathrm{d}y}{|x-y|^{\nu}} + C \int_{|y|\geq 2|x|} \frac{\mathrm{d}y}{|y|^{\nu+\beta_2}}  \nonumber\\
&\leq  C |x|^{-\nu} \left\{\int_{|y|\leq R_0 } + \int_{R_0<|y|\leq\frac{|x|}{2}} \right\} g(y) \mathrm{d}y
+ C_2 |x|^{-\beta_2} \int_{|y-x|\leq 3|x|} \frac{\mathrm{d}y}{|x-y|^{\nu}}
+ C |x|^{N-\nu-\beta_2} \nonumber\\
&\leq  C |x|^{-\nu} \left\{\int_{|y|\leq R_0 }|y|^{-\beta_1} \mathrm{d}y + \int_{R_0< |y|\leq\frac{|x|}{2}} |y|^{-\beta_2} \mathrm{d}y \right\}
+ C |x|^{N-\nu-\beta_2} \nonumber\\
&\leq C |x|^{-\nu}+C |x|^{N-\nu-\beta_2},\nonumber
\end{align}
in which we have used the following facts:
\begin{align*}
|x-y|\geq |x|-|y| \geq \frac{|x|}{2}, \,\,\,\,\,\,\,&\mbox{if}\,\,|y|\leq \frac{|x|}{2},\\
|x-y|\leq |x|+|y| \leq 3|x|, \,\,\,\,\,\,\,&\mbox{if}\,\, \frac{|x|}{2} \leq |y|\leq 2|x|,\\
|x-y|\geq |y|-|x| \geq \frac{|y|}{2}, \,\,\,\,\,\,\,&\mbox{if}\,\,|y|\geq 2|x|.
\end{align*}
Thus \eqref{eq10.26.19} holds true.

\medskip

Next, similar to \eqref{eq10.28.30}, for any $|x|\leq R_0/2<1$, one has
\begin{align}\label{eq10.28.30+}
&\int_{\R^N} \frac{g(y)}{|x-y|^{\nu}}\mathrm{d}y =\left\{\int_{|y|\leq\frac{|x|}{2}} + \int_{\frac{|x|}{2} \leq |y| \leq 2|x|} +\int_{|y|\geq 2|x|}  \right\} \frac{g(y)}{|x-y|^{\nu}}\mathrm{d}y \\
&\leq C |x|^{-\nu}\int_{|y|\leq\frac{|x|}{2}} g(y) \mathrm{d}y + C_1 |x|^{-\beta_1} \int_{\frac{|x|}{2} \leq |y| \leq 2|x|} \frac{\mathrm{d}y}{|x-y|^{\nu}} + C \int_{|y|\geq 2|x|} \frac{g(y)}{|y|^{\nu}}\mathrm{d}y  \nonumber\\
&\leq C |x|^{-\nu} \int_{|y|\leq\frac{|x|}{2}} \frac{\mathrm{d}y}{|y|^{\beta_1}}  + C_1 |x|^{-\beta_1} \int_{|y-x|\leq 3|x|} \frac{\mathrm{d}y}{|x-y|^{\nu}} + C \left\{\int_{2|x| \leq |y| \leq R_0 } + \int_{R_0 <|y|} \right\} \frac{g(y)}{|y|^{\nu}} \mathrm{d}y  \nonumber\\
&\leq C |x|^{N-\nu-\beta_1} + C \left\{\int_{2|x| \leq |y| \leq R_0 } \frac{\mathrm{d}y}{|y|^{\nu+\beta_1}}  + \int_{R_0 <|y|} \frac{\mathrm{d}y}{|y|^{\nu+\beta_2}} \right\} \nonumber\\
&\leq  C |x|^{N-\nu-\beta_1}+C,\nonumber
\end{align}
which yields \eqref{eq10.26.19+}.

\medskip

Finally, for any $R_0/4 \leq |x| \leq 3R$, it holds
\begin{align*}
|x|^{-\nu}*g &= \int_{\R^N} \frac{g(y)}{|x-y|^{\nu}} \mathrm{d}y \\
&= \int_{B_{R_0/8}} \frac{g(y)}{|x-y|^{\nu}} \mathrm{d}y + \int_{B_{4R}\setminus B_{R_0/8}} \frac{g(y)}{|x-y|^{\nu}} \mathrm{d}y +  \int_{\R^N \setminus B_{4R}} \frac{g(y)}{|x-y|^{\nu}} \mathrm{d}y \\
&=:I_1 + I_2 + I_3.
\end{align*}
Since $|x-y|\geq \frac{R_0}{4}-\frac{R_{0}}{8}$ for $y\in B_{R_0/8}$ and $x\in B_{3R}\setminus B_{R_0/4}$, we have
\begin{align*}
I_1 = \int_{B_{R_0/8}} \frac{g(y)}{|x-y|^{\nu}} \mathrm{d}y
    \leq\frac{C}{R_0^\nu} \int_{B_{R_0/8}} \frac{\mathrm{d}y}{|y|^{\beta_1}}  \leq C R_0^{N-\nu-\beta_1}.
\end{align*}
Due to $|y|/4 \leq |y|-3R \leq |x-y| \leq |x|+|y|\leq 2|y|$ for $y\in \R^N \setminus B_{4R}$ and $x\in B_{3R}\setminus B_{R_0/4}$, we get
\begin{align*}
I_3 =  \int_{\R^N \setminus B_{4R}}
       \frac{g(y)}{|x-y|^{\nu}} \mathrm{d}y
       \leq C \int_{\R^N \setminus B_{4R}} \frac{\mathrm{d}y}{|y|^{\beta_2+\nu}} \leq C R^{N-\nu-\beta_2}.
\end{align*}
By assumptions on $g$, one has
\begin{align*}
I_2 = &\int_{B_{4R}\setminus B_{R_0/8}} \frac{g(y)}{|x-y|^{\nu}} \mathrm{d}y \\
      &\leq C R_0^{-\beta_1}\int_{R_0/8\leq |y| \leq R_0} \frac{\mathrm{d}y}{|x-y|^{\nu}} + C R^{-\beta_2} \int_{R_0\leq |y| \leq 4R} \frac{\mathrm{d}y}{|x-y|^{\nu}}
      \leq C.
\end{align*}
Combining estimates for $I_1$, $I_2$ and $I_3$, we see that \eqref{eq10.26.19++} holds. This finishes our proof of $(ii)$.

\smallskip

{\bf $(iii)$} For any $x\in B_{\frac{R}{2}}(0)$,
\begin{align*}
|x|^{-\nu}*g &= \int_{\R^N} \frac{g(y)}{|x-y|^{\nu}} \mathrm{d}y
\geq \int_{\R^N \setminus B_{R}(0)} \frac{g(y)}{|x-y|^{\nu}} \mathrm{d}y.
\end{align*}
Due to $\frac{1}{2}|y| \leq |y|-\frac{R}{2} \leq |x-y| \leq 2|y|$ for $x\in B_{\frac{R}{2}}$ and $y\in B^c_{R}$, we have
\begin{align*}
&\int_{\R^N \setminus B_{R}} \frac{g(y)}{|x-y|^{\nu}} \mathrm{d}y \\
&\geq C \int_{\R^N \setminus B_{R}} \frac{1}{|y|^{\nu}} |y|^{-\beta_0} \mathrm{d}y  \geq C \int_{R}^\infty r^{-\beta_0-\nu+N-1} \mathrm{d}r = C R^{-\beta_0-\nu+N},\,\,\,\,\,\,\,\,\,x\in B_{\frac{R}{2}}(0).
\end{align*}
Therefore, we arrive at $(iii)$. The proof of Lemma \ref{lm.6} is completed.
\end{proof}

\begin{rem}\label{rm.2.7}
Assume that $u$ is a nonnegative weak solution to \eqref{gen-eq} and let the nonlocal nonlinear term $f_k(x):=\mu |x|^{-p} u^{p-1}(x) + V_k(x) |x|^{-s} u^{p-1}(x)$, where $k=1,2,3$ ($s=0$ when $k=1$). We have proved that $u\in C_{loc}^{1,\alpha}(\R^N\setminus\{0\})$ for some $0<\alpha<\min\{1,\frac{1}{p-1}\}$ (see Theorems \ref{th2.1.1}, \ref{th2} and \ref{gth2}). Moreover, if $u$ is positive, by Theorems \ref{th2.1}, \ref{th2} and \ref{gth2}, we obtain that, there exist $0<\gamma_1<(N-p)/p<\gamma_2<(N-p)/(p-1)$, $c_0, C_0>0$ and $0<R_0<1<R_1$ such that
\[ c_0 |x|^{-\gamma_1} \leq  u(x) \leq C_0 |x|^{-\gamma_1} \qquad \mbox{in}\,\,\, B_{R_0}(0),\]
\[ c_0 (1+|x|^{\gamma_2})^{-1} \leq u(x) \leq C_0 (1+|x|^{\gamma_2})^{-1} \,\,\,\,\,\,\,\mbox{in}\,\,\R^N \setminus B_{R_0}(0),\]
\[ c_0 |x|^{-\gamma_1-1} \leq |\nabla u(x)| \leq C_0 |x|^{-\gamma_1-1} \qquad \mbox{in}\,\,\, B_{R_0}(0),\]
\[ c_0 |x|^{-\gamma_2-1} \leq |\nabla u(x)| \leq C_0 |x|^{-\gamma_2-1} \qquad \mbox{in}\,\,\,\R^N \setminus B_{R_1}(0).\]
One can observe that $N>p\gamma_1$, $N>p_{s,\sigma}\gamma_1$, $N> N-p> N-2p$ and $N > p_{s,\sigma} \cdot (N-p)/p > N- \sigma$. Therefore, by  $(ii)$ and $(iii)$ in Lemma \ref{lm.6} with $g(x)=u^p(x)$, $\beta_0=p\gamma_2$, $\beta_1=p\gamma_1$, $\beta_2=N-p$ and $\nu=2p$, and with $g(x)=u^{p_{s,\sigma}}(x)$, $\beta_0=p_{s,\sigma}\gamma_2$, $\beta_1=p_{s,\sigma}\gamma_1$, $\beta_2=p_{s,\sigma} \cdot (N-p)/p$ and $\nu=\sigma$, we have, for any $R>0$,
\begin{align}\label{eq10.26.30.03}
C_1 \leq |x|^{-2p}\ast u^p \leq C_2 \max\{|x|^{-p},1\} \,\,\,\,\,\,\,\mbox{in}\,\, B_R,
\end{align}
\begin{align}\label{eq10.26.30.03+}
C_1 \leq |x|^{-\sigma}\ast u^{p_{s,\sigma}} \leq C_2 \max\{|x|^{-\sigma/2 + s/2},1\} \,\,\,\,\,\,\,\mbox{in}\,\, B_R.
\end{align}
For any $i=1,2,\cdots,N$, by $(ii)$ and $(iii)$ in Lemma \ref{lm.6} with $g(x)=pu^{p-1}(x)|\nabla u(x)|$, $\beta_0=p\gamma_2+1$, $\beta_1=p\gamma_1+1$, $\beta_2=N-p$ and $\nu=2p$, and with $g(x)=p_{s,\sigma}u^{p_{s,\sigma}-1}(x)|\nabla u(x)|$, $\beta_0=p_{s,\sigma}\gamma_2+1$, $\beta_1=p_{s,\sigma}\gamma_1+1<N$ (see Remark \ref{rem1}), $\beta_2=p_{s,\sigma} \cdot (N-p)/p$ and $\nu=\sigma$, we have, for any $R>0$,
\begin{align}\label{eq10.26.30.02}
|(|x|^{-2p}\ast u^p)_{x_i}(x)| \leq \int_{\R^N} \frac{p|u(x-y)|^{p-1} |\nabla u(x-y)|}{|y|^{2p}} \mathrm{d}y \leq C_3\max\{|x|^{-p-1},1\} \,\,\,\,\,\,\mbox{in}\,\, B_R,
\end{align}
\begin{align}\label{eq10.26.30.02+}
|(|x|^{-\sigma}\ast u^{p_{s,\sigma}})_{x_i}| \leq \int_{\R^N} \frac{p_{s,\sigma}|u(y)|^{p_{s,\sigma}-1} |\nabla u(y)|}{|x-y|^{\sigma}} \mathrm{d}y \leq C_3 \max\{|x|^{-\sigma/2 + s/2-1},1\} \,\,\,\,\,\,\mbox{in}\,\, B_R,
\end{align}
where the constants $C_i>0$ $(i=1,2,3)$ depend on $R$ and $R_0$. Furthermore, \eqref{eq10.26.30.03}--\eqref{eq10.26.30.02+} show that, for any $\Omega \subset\subset \R^N\setminus\{0\}$, $f_k(x) =\mu |x|^{-p} u^{p-1}(x) + V_k(x)|x|^{-s}u^{p-1}(x)$ ($k=1,2,3$, with $s=0$ when $k=1$) satisfy
\begin{align}\label{eq26.30.001}
f_k(x)\in L^s (\Omega) \,\,\,\,\,\,\,\,\mbox{with}\,s>N,
\end{align}
and the so-called ``condition ($I_{\alpha}$)" in Lemma \ref{lm.4}, i.e., for some $0\leq\mu<\alpha(p-1)$, there holds:
\begin{align}\label{eq26.30.002}
f_k(x)\in W^{1,m}(\Omega)
\end{align}
for some $m>\frac{N}{2(1-\mu)}$, and, given any $x_0\in \Omega$, there exist $C(x_0,\mu)>0$ and $\rho(x_0,\mu)>0$ such that, $B_{\rho(x_0,\mu)}(x_{0})\subset \Omega$ and
\begin{align}\label{eq26.30.003}
|f_k(x)|\geq C(x_0,\mu) |x-x_0|^\mu \qquad \mbox{in}\,\,B_{\rho(x_0,\mu)}(x_{0}).
\end{align}
\end{rem}

\medskip

In proofs of Lemmas \ref{lm:l-b-r}, \ref{lm:l-b-2} and \ref{lm:l-b-3}, we will need the following inequalities.
\begin{lem}[Lemma 2.3 in \cite{EMSV}]\label{tech-in}
Let $p>1$ and $a,b\geq 0$. Then, for all $\delta>0$ there exists $C_\delta >0$ such that
$$a^p \geq \frac{1}{1+2^{p-1}\delta}(a+b)^p -C_\delta b^p.$$
\end{lem}
\begin{lem}[Lemma 2 in \cite{JSLB}]\label{tech-in-1}
Let $\alpha$ be a positive exponent, and let $a_i, \ \beta_i, \ i=1,2,\cdots,M$ be two sets of $M$ real numbers such that $0<a_i<\infty$ and $0\leq\beta_i<\alpha$. Suppose that $z$ is a positive number satisfying the inequality
\[z^\alpha \leq \sum\limits_{i=1}\limits^{M} a_i z^{\beta_i}.\]
Then
\[z\leq C \cdot \sum\limits_{i=1}\limits^{M} (a_i)^{\frac{1}{\alpha-\beta_i}},\]
where $C>0$ depends only on $M, \alpha$ and $\beta_i$.
\end{lem}

\medskip

In order to apply the moving planes method to $p$-Laplace equations, we will frequently use the following basic point-wise gradient estimate (Lemma \ref{lm.w-g-i}) and the strong comparison principles (Lemmas \ref{th3.1} and \ref{th3.2}).
\begin{lem}[Lemma 3.1 in \cite{OSV}]\label{lm.w-g-i}
Let $u, v$ be positive and $C^1$--functions defined in a neighborhood of some point $x_0\in\R^N$. Then it holds
\begin{equation}\label{050331}
\begin{aligned}
      |\nabla u|^{p-2}\nabla u & \cdot \nabla\left(u-\frac{v^p}{u^p}u\right) +  |\nabla v|^{p-2}\nabla v \cdot \nabla\left(v-\frac{u^p}{v^p}v\right)\\
      & \geq C_p \min\{v^p, u^p\}(|\nabla \log u|+|\nabla \log v|)^{p-2} |\nabla \log u- \nabla \log v|^2, \,\quad \forall p>1,
\end{aligned}
\end{equation}
near $x_0$ for some constant $C_p$ depending only on $p$.
\end{lem}

\begin{lem}[Theorem 1.4 in \cite{LDBS}, c.f. also Theorem 2.1 in \cite{OSV}]\label{th3.1}
Let $\frac{2N+2}{N+2}<p<2$ or $p\geq 2$ and $u,v \in C^1(\overline{\Omega})$, where $\Omega$ is a bounded smooth domain of $\R^N$.
Suppose that $f:\overline{\Omega}\times [0,\infty)\to \R$ is a continuous positive function and of class $C^1$ in $\Omega\times(0,\infty)$. Suppose that either $u_1,u_2$ are two weak solutions of
\begin{align*}
\left\{ \begin{array}{ll}
-\Delta_p u =f(x,u) \,\,\,\,\,\, & \mbox{in}\,\,\Omega, \\
u>0 \,\,\,\,\,\,& \mbox{in}\,\,\Omega, \\
u=0 \,\,\,\,\,\,& \mbox{on}\,\,\partial\Omega \\
\end{array}
\right.\hspace{1cm}
\end{align*}
satisfying
$$ -\Delta_p u_1 + \Lambda u_1 \leq  -\Delta_p u_2 + \Lambda u_2 \qquad \,\,\,\mbox{and} \quad \,u_1\leq u_2\,\quad\mbox{in}\,\,\Omega,$$
where $\Lambda\in\R$. Then $u_1\equiv u_2$ in $\Omega$ unless $u_1<u_2$ in $\Omega$.
\end{lem}

\begin{lem}[Theorem 1.4 in \cite{LD}]\label{th3.2}
Suppose that $\Omega$ is a bounded domain of $\R^N$, $1<p<\infty$ and let $u,v \in C^1(\Omega)$ weakly satisfy
$$ -\Delta_p u + \Lambda u \leq  -\Delta_p v + \Lambda v \quad\,\,\,\mbox{and}\quad\,u\leq v\,\,\,\,\,\mbox{in}\,\,\Omega,$$
and denote by $Z_v^u:=\{ x\in\Omega\,\mid\,|\nabla u(x)|=|\nabla v(x)| =0 \}$. Then if there exists $x_0\in\Omega\setminus Z_v^u$ with $u(x_0)=v(x_0)$, then $u\equiv v$ in the connected component of $\Omega\setminus Z_v^u$ containing $x_0$. The same result still holds if, more generally,
$$ -\Delta_p u - f(u) \leq  -\Delta_p v - f(v) \quad\,\,\,\mbox{and}\,\quad u\leq v \quad\,\mbox{in}\,\,\Omega$$
with $f:\R \to \R$ locally Lipschitz continuous.
\end{lem}

We recall a version of the strong maximum principle and the Hopf Lemma for the $p$-Laplacian, c.f. \cite{JLV}, see also \cite[Theorem 2.1]{LDBS04}, \cite[Theorem 2.4]{LDSMLMSB} or \cite{LD,PS}.
\begin{lem}[Strong Maximum Principle and H\"{o}pf's Lemma, \cite{JLV}, Theorem 2.1 in \cite{LDBS04}]\label{hopf}
Let $\Omega$ be a domain in $\mathbb{R}^{N}$ and suppose that $u\in C^{1}(\Omega)$, $u\geq0$ in $\Omega$ weakly solves
\[-\Delta_{p}u+cu^{q}=g\geq0 \qquad \text{in} \,\,\Omega\]
with $1<p<+\infty$, $q\geq p-1$, $c\geq0$ and $g\in L^{\infty}_{loc}(\Omega)$. If $u\not\equiv0$, then $u>0$ in $\Omega$. Moreover, for any point $x_{0}\in\partial\Omega$ where the interior sphere condition is satisfied and $u\in C^{1}(\Omega\cup\{x_{0}\})$ and $u(x_{0})=0$, we have that $\frac{\partial u}{\partial s}>0$ for any inward directional derivative (this means that if $y$ approaches $x_{0}$ in a ball $B\subseteq\Omega$ that has $x_{0}$ on its boundary, then $\lim\limits_{y\rightarrow x_{0}}\frac{u(y)-u(x_{0})}{|y-x_{0}|}>0$).
\end{lem}

\subsection{Regularity and properties of critical set of solutions}\label{sc2.2}

In this subsection, we will give the integrable properties of $\frac{1}{|\nabla u|}$ and the connectivity of  set of critical point of  nonnegative weak solution $u$ to \eqref{eq1.1} (or \eqref{eq1.2} or \eqref{eq1.3}, resp.), and the weighted Poincar\'{e} type inequality.

\smallskip

From Theorems \ref{th2.1.1}, \ref{th2} and \ref{gth2}, the nonnegative weak solution $u$ of \eqref{eq1.1}, \eqref{eq1.2} and \eqref{eq1.3} (respectively) satisfies $u\in L^{\infty}_{loc}(\mathbb{R}^{N}\setminus\{0\})\cap C^{1,\alpha} (\R^N\setminus\{0\})$ for some $0<\alpha<\min\{1,\frac{1}{p-1}\}$. Moreover, using standard elliptic regularity, the solution $u\in C^{1,\alpha} (\Omega)$ belongs to the class $C^2(\Omega\setminus Z_u)$, where $\Omega\subset\subset\R^N\setminus\{0\}$ and $Z_u=\{ x\in\Omega : |\nabla u(x)|=0 \}$ (see \cite{ED83,DGNT,PT}).

\smallskip

We need the following Hessian and reversed gradient estimates from \cite{BSR14}.
\begin{lem}[Hessian and reversed gradient estimates, \cite{BSR14}] \label{lm.4}
Let $\Omega\subset\subset\R^N$ and $1<p<\infty$. Assume $u\in W^{1,p}(\Omega)$ be a weak solution to
\begin{align}\label{eq10.252252}
-\Delta_p u =f(x) \quad\,\,\,\mbox{in}\,\,\Omega,
\end{align}
where $f(x)$ satisfies \eqref{eq26.30.001} and condition ($I_\alpha$) in Remark \ref{rm.2.7}, i.e., \eqref{eq26.30.002}-\eqref{eq26.30.003}. Then, for any $x_0\in Z_u$ and $B_{2\rho}(x_0)\subset \Omega$, we have
\begin{align}\label{eq10.25.52}
\int_{B_\rho(x_0)} \frac{|\nabla u|^{p-2-\beta}(x)}{|x-y|^\gamma} |\nabla  u_{x_i}(x)|^2 \mathrm{d}x < C,\quad\,\,i=1,\cdots,N
\end{align}
uniformly for any $y\in \Omega$, where $0\leq\beta<1$, $\gamma=0$ if $N=2$, $\gamma<N-2$ if $N \geq 3$ and
$$C=C(N,p,x_0,\gamma,\beta,u,f,\rho)>0.$$
Moreover, we have, uniformly for any $y\in \Omega$,
\begin{align}\label{eq10.25.2.1}
\int_{B_\rho(x_0)} \frac{1}{|\nabla u|^t} \frac{1}{|x-y|^\gamma}\mathrm{d}x <\widetilde{C},
\end{align}
where $\max\{p-2,0\}\leq t < p-1$ and $\gamma=0$ if $N=2$, $\gamma<N-2$ if $N \geq 3$ and $$\widetilde{C}=\widetilde{C}(N,p,x_0,\gamma,t,u,f,\rho)>0.$$
\end{lem}

From Remark \ref{rm.2.7} and Lemma \ref{lm.4}, we can derive the following Corollary.
\begin{cor}\label{re2333}
Let $u$ be a positive weak solution of \eqref{eq1.1} (or \eqref{eq1.2}, or \eqref{eq1.3}, resp.). Setting $\rho(x):=|\nabla u(x)|^{p-2}$, for any $\Omega \subset\subset\R^N\setminus\{0\}$, we have
\begin{align}\label{eq10.25.2}
\int_{\Omega} \frac{1}{\rho^t(x)} \frac{1}{|x-y|^\gamma}\mathrm{d}x < C
\end{align}
uniformly for any $y\in\Omega$, where $\max\{p-2,0\}\leq (p-2)t<p-1$ and $\gamma=0$ if $N=2$, $\gamma<N-2$ if $N \geq 3$.
\end{cor}
\begin{proof}
In Remark \ref{rm.2.7}, we know that $f_k(x) =\mu |x|^{-p} u^{p-1} + V_k(x)|x|^{-s}u^{p-1}(x)$ ($k=1,2,3$, with $s=0$ when $k=1$) satisfy \eqref{eq26.30.001} and the condition ($I_\alpha$) (i.e., \eqref{eq26.30.002}-\eqref{eq26.30.003}) in Lemma \ref{lm.4} with $\Omega \subset\subset\R^N\setminus\{0\}$. Noting that $u\in C^2(\Omega\setminus Z_u)$, it follows from the finite covering theorem and \eqref{eq10.25.2.1} in Lemma \ref{lm.4} that \eqref{eq10.25.2} holds. This finishes the proof of Corollary \ref{re2333}.
\end{proof}

\begin{rem}\label{re2334}
Let $u$ be a positive weak solution of \eqref{eq1.1} (or \eqref{eq1.2}, or \eqref{eq1.3}, resp.) and let $Z_u = \{x \in \R^N: |\nabla u(x)| = 0 \}$. Clearly, it follows from Theorem \ref{th2.1} that the critical set $Z_u\subset B_{R_1}\setminus B_{R_0}$ and $Z_u$ is a closed set. Moreover, \eqref{eq10.25.2} in Corollary \ref{re2333} implies that the Lebesgue measure of $Z_{u}$ is zero, i.e., $|Z_u|=0$.
\end{rem}

From Lemma \ref{th3.2}, we can clearly see that it is very useful to know whether $\Omega\setminus Z_u$ is connected or not. We have the following Lemma on the connectivity of the critical set $Z_{u}$ from \cite[Theorem 4.1 and Remark 4.1]{LDBS04} and \cite[Lemma 5]{MMPS}\footnote{One can clearly see that, with slight modifications on its proof, \cite[Lemma 5]{MMPS} also holds for quasi-linear equations involving Hardy potential without the gradient term $|\nabla u|^{p}$.}.
\begin{lem}[Properties of the critical set, c.f. e.g. \cite{LDBS04,MMPS}]\label{lm.7}
Assume $u \in C^{1,\alpha} (\bar{\Omega}\setminus\{0\})$ be a positive weak solution of \eqref{eq1.1} (or \eqref{eq1.2}, or \eqref{eq1.3}, resp.). Let us define the critical set
$$Z_u = \{ x\in\Omega : |\nabla u(x)|=0 \},$$
where $\Omega$ is a general bounded domain. 
Then $\Omega\setminus Z_u$ does not contain any connected component $C$ such that $\overline{C} \subset \Omega$.
\end{lem}

Finally, we give the weighted Poincar\'{e} type inequality, which will play a key role in the proof of Theorem \ref{th2} (and also Theorem \ref{gth2}) via the method of moving planes in the degenerate elliptic case $p>2$. To this end, we need the following definition.
\begin{defn}[Weighted Sobolev spaces]\label{lm.3}
Assume that $\Omega\subset\R^N$ is a bounded domain and let $\rho\in L^1(\Omega)$ be a positive function. We define $H^{1,p}(\Omega,\rho)$ as the completion of $C^1(\overline{\Omega})$ (or $C^\infty(\overline{\Omega})$) with the norm
$$ \| v \|_{ H^{1,p}(\Omega,\rho) } := \| v \|_{L^p(\Omega)} + \| \nabla u\|_{L^p(\Omega,\rho)},$$
where $\| \nabla u\|^p_{L^p(\Omega,\rho)} = \int_\Omega  \rho  |\nabla u |^p \mathrm{d}x$.

In this way $H^{1,2}(\Omega,\rho)$ is a Hilbert space, and we also define $H_0^{1,p}(\Omega,\rho)$ as the closure of $C_0^\infty(\overline{\Omega})$ in $H^{1,p}(\Omega,\rho)$.
\end{defn}

\begin{lem}[Weighted Poincar\'{e} type inequality, \cite{LDBS04,FV,LPLDHT}]\label{lm.3}
Assume that $\Omega\subset\R^N$ is a bounded domain and let $\rho\in L^1(\Omega)$ be a positive function such that
\begin{align}\label{eq10.25.3}
\int_\Omega \frac{1}{\rho |x-y|^\gamma}\mathrm{d}y < C, \,\,\,\quad \forall \,\, x\in\Omega,
\end{align}
where $\gamma<N-2$ if $N \geq 3$ and $\gamma=0$ if $N = 2$. Let $v\in H^{1,2}(\Omega,\rho)$ be such that
\begin{align}\label{eq10.25.4}
|v(x)|\leq C \int_\Omega \frac{|\nabla v(y)|}{ |x-y|^{N-1}}\mathrm{d}y,\quad\,\,\,\forall \,\, x\in\Omega.
\end{align}
Then we have
\begin{equation}\label{wp}
  \int_\Omega v^2(x) \mathrm{d}x \leq C(\Omega) \int_\Omega \rho |\nabla v(x)|^2 \mathrm{d}x,
\end{equation}
where $C(\Omega)\to 0$ if $|\Omega|\to 0$. The same inequality holds for $v\in H^{1,2}_0(\Omega,\rho)$.
\end{lem}

\begin{rem}\label{rem0}
Assume that $0\leq\mu<\bar{\mu}$ and let $u$ be a positive weak solution of \eqref{eq1.1} (or \eqref{eq1.2}, or \eqref{eq1.3}, resp.). By Theorems \ref{th2.1}, \ref{th2} and \ref{gth2}, we know that, for any $p\geq2$, $\rho:=|\nabla u|^{p-2}\in L^1(\Omega)$ for any $\Omega\subset\subset\R^N\setminus\{0\}$. In addition, from \eqref{eq10.25.2} in Corollary \ref{re2333} and Remark \ref{re2334}, we deduce that \eqref{eq10.25.3} holds for $\rho=|\nabla u|^{p-2}>0$ a.e.. Hence the weighted Poincar\'{e} type inequality \eqref{wp} in Lemma \ref{lm.3} holds with $\Omega\subset\subset\R^N\setminus\{0\}$ and $\rho=|\nabla u|^{p-2}$ for $p\geq2$.
\end{rem}

\section{Proofs of Theorems \ref{th2.1.1}, \ref{th2.1++}, \ref{th2.1} and the sharp asymptotic estimates of positive solutions to equation \eqref{eq1.1}}\label{sc3}

In subsections \ref{sc3.1}--\ref{sc3.3}, we will give proofs of Theorems \ref{th2.1.1}, \ref{th2.1++} and \ref{th2.1}, and establish the sharp estimates on asymptotic behaviors of positive weak solutions near the origin and at infinity to the generalized equation \eqref{gen-eq} respectively. In subsection \ref{sc3.4}, we will show that $V_1(x)=|x|^{-2p}\ast u^{p}$ satisfies all the assumptions on $V(x)$ in Theorem \ref{th2.1} and hence derive the sharp asymptotic estimates in Theorem \ref{th2} from Theorem \ref{th2.1}.

\smallskip

 Since the general problem \eqref{gen-eq} includes problems \eqref{eq1.1}, \eqref{eq1.2} and \eqref{eq1.3} as special cases, we therefore will consider the generalized equation \eqref{gen-eq}, i.e.,
\begin{equation*}
  \left\{\begin{array}{ll} \displaystyle
            -\Delta_p u - \mu \frac{1}{|x|^p}|u|^{p-2}u= V(x)|x|^{-s} |u|^{p-2}u   &\mbox{in}\, \R^N, \\ \\
            u \in D^{1,p}(\R^N),\,\, u\geq0 \,\,\,\,\,\,&  \mbox{in}\, \R^N,
          \end{array}
  \right.\hspace{1cm}
\end{equation*}
where $0\leq\mu< \bar{\mu}=\left( (N-p)/p \right)^p$, $1<p<N$, $0\leq s < p$ and $0\leq V(x)\in L^\frac{N}{p-s}(\R^N)$.

\subsection{Local integrability, boundedness, regularity and asymptotic estimates}\label{sc3.1}
Inspired by \cite{CDPSYS,DLL,ED83,EMSV,HL,JSLB,PT,XCL}, we can prove the local integrability, boundedness, regularity and asymptotic estimates of solutions in the following Lemmas \ref{lm:l-b-r}, \ref{lm:l-b-2} and \ref{lm:l-b-3} via the De Giorgi-Moser-Nash iteration techniques, which goes back to e.g. \cite{TMOSER}.

\begin{lem}[Local integrability, boundedness, regularity and asymptotic estimates]\label{lm:l-b-r}
Let $u\in D^{1,p}(\R^N)$ be a nonnegative solution to the generalized equation \eqref{gen-eq} with $0\leq\mu<\bar{\mu}$, $1<p<N$, $0\leq s < p$ and $0\leq V(x)\in L^\frac{N}{p-s}(\R^N)$.
Then $u\in L_{loc}^{r}(\R^N\setminus \{0\})$ for any $0<r<+\infty$.
If $V\in L^{q}_{loc}(\R^{N}\setminus \{0\})$ for some $\frac{N}{p-s}<q<+\infty$, then $u\in C^{1,\alpha}(\R^N\setminus \{0\})\cap L^{\infty}_{loc}(\R^N \setminus \{0\})$ for some $0<\alpha<\min\{1,\frac{1}{p-1}\}$.
In particular, for any $y \in \R^N \setminus \{0\}$, and for any $0<R\leq\frac{|y|}{3}$, if $R^{p-s-\frac{N}{q}}\|V\|_{L^{q}(B_{2R}(y))}\leq \widetilde{C}$ for some $\frac{N}{p-s}<q<+\infty$ and some constant $\widetilde{C}>0$, then there holds:
\begin{equation}\label{est-1}
  \max_{B_R(y)} u(x) \leq \frac{C}{R^\frac{N}{\bar{p}}} \| u \|_{L^{\bar{p}}(B_{2R}(y))}, \qquad \forall \,\, p^\star\leq\bar{p}<+\infty,
\end{equation}
where $C=C\left(N,p,s,\mu,q,\widetilde{C}\right)>0$.
\end{lem}

\begin{proof}
Let $\Omega$ be any bounded domain such that $\Omega \subset\subset \R^N\setminus \{0\}$ and take arbitrarily $\bar{p}\in[p^\star,+\infty)$. Since $V(x)\in L^\frac{N}{p-s}(\R^N)$, for any $y\in \Omega$, there exists a small $R_1=R_1(\bar{p})>0$ satisfying $B_{3R_1}(y) \subset \R^N\setminus \{0\}$ and
\begin{equation}\label{050101}
  \| V \|_{L^\frac{N}{p-s}(B_{2R_1}(y))} \leq \frac{1}{2^{3+s}|B_{1}(0)|^\frac{s}{N} C_S^p (\bar{p}/p)^{p-1}}.
\end{equation}
First, we aim to prove that
\begin{equation}\label{loc-est-1}
  \|u\|_{L^{\bar{p}}(B_R(y))} \leq \frac{C}{R^{\frac{N-p}{p}-\frac{N}{\bar{p}}}} \|u\|_{L^{p^\star}(B_{2R}(y))}, \,\,\,\quad \forall\,\, R\leq R_1,
\end{equation}
where $C=C(N,p,s,\mu)>0$ and $C_S$ is Sobolev imbedding constant.

\smallskip

In order to show \eqref{loc-est-1}, for any $2R_1\geq h > h^\prime >0$, let us consider a cut--off function $\eta\in C_0^\infty (B_h (y))$ with $\eta=1$ in $B_{h^\prime}(y)$, $0\leq \eta \leq 1$ and $|\nabla \eta|\leq \frac{2}{h-h^\prime}$. Now, by density argument, we can take the test function $\varphi=\eta^p u^{1+p(t-1)}$ in \eqref{gen-eq} (where $1<t\leq \bar{p}/p$) and obtain that
\begin{equation}\label{050102}
  \int_{\R^N} |\nabla u|^{p-2}\nabla u \cdot \nabla \varphi \mathrm{d}x - \mu \int_{\R^N} \frac{1}{|x|^p} u^{p-1} \varphi \mathrm{d}x = \int_{\R^N} V(x) |x|^{-s} u^{p-1} \varphi \mathrm{d}x.
\end{equation}

\smallskip

On one hand, for any fixed $l\in(0,p)$, using H\"{o}lder's inequality and the Sobolev inequality, we have
\begin{equation}\label{050103}
\begin{aligned}
 \mu \int_{\R^N} \frac{1}{|x|^p} u^{p-1} \varphi \mathrm{d}x
      &= \mu \int_{\R^N} \frac{1}{|x|^p} u^{pt} \eta^p \mathrm{d}x \\
      &\leq \mu \left\| \frac{1}{|x|^p} \right\|_{L^{\frac{N}{p-l}}(\R^N \cap supp(\eta))}  \left\| \eta u^t \right\|^{p-l}_{L^{p^\star}(\R^N)}  \left\| \eta u^t \right\|^l_{L^p(\R^N)}\\
       &\leq \mu C_S^{p-l} \left\| \frac{1}{|x|^p} \right\|_{L^{\frac{N}{p-l}}(\R^N \cap supp(\eta))}  \left\| \nabla(\eta u^t) \right\|^{p-l}_{L^p(\R^N)}  \left\| \eta u^t \right\|^l_{L^p(\R^N)},
\end{aligned}
\end{equation}
\begin{equation}\label{050104}
\begin{aligned}
      \int_{\R^N} V(x)|x|^{-s} u^{p-1} &\varphi \mathrm{d}x
      = \int_{\R^N} V(x) |x|^{-s} u^{pt} \eta^p \mathrm{d}x \\
      &\leq  \left\| V(x) \right\|_{L^{\frac{N}{p-s}}(\R^N \cap supp(\eta))}
             \left\| |x|^{-s} \right\|_{L^{\frac{N}{s}}(\R^N \cap supp(\eta))}
             \left\| u^{pt} \eta^p \right\|_{L^{\frac{N}{N-p}}(\R^N)} \\
      &\leq  C_S^p \left\| V(x) \right\|_{L^{\frac{N}{p-s}}(\R^N \cap supp(\eta))}
                  \left\| |x|^{-s} \right\|_{L^{\frac{N}{s}}(\R^N \cap supp(\eta))}
                  \left\| \nabla(\eta u^t) \right\|^p_{L^p(\R^N)}.
\end{aligned}
\end{equation}

\smallskip

On the other hand, the first term in the left hand side of \eqref{050102} can be calculated as follows:
\begin{equation}\label{050105}
\begin{aligned}
      \int_{\R^N} &|\nabla u|^{p-2}  \nabla u \cdot \nabla\varphi \mathrm{d}x =\int_{\R^N} |\nabla u|^{p-2} \nabla u \cdot \nabla\left( \eta^p u^{1+p(t-1)} \right) \mathrm{d}x \\
      &= (1+p(t-1))\int_{\R^N} |\nabla u|^p \eta^p u^{p(t-1)}\mathrm{d}x +p\int_{\R^N} \eta^{p-1} u^{1+p(t-1)} |\nabla u|^{p-2} \nabla u  \cdot \nabla\eta \mathrm{d}x.
\end{aligned}
\end{equation}
By Lemma \ref{tech-in}, for any $\delta\in(0,1)$, there exists a positive constant $C_\delta$ such that
\begin{equation}\label{050106}
\begin{aligned}
      &(1+p(t-1))\int_{\R^N} |\nabla u|^p \eta^p u^{p(t-1)}\mathrm{d}x \\
      &\geq \frac{1+p(t-1)}{t^p} \frac{1}{1+2^{p-1}\delta} \int_{\R^N} \left| \nabla(\eta u^t) \right|^p \mathrm{d}x - \frac{1+p(t-1)}{t^p} C_\delta \int_{\R^N} |\nabla \eta|^p u^{pt} \mathrm{d}x.
\end{aligned}
\end{equation}
Next, using the Young's inequality $ab\leq \epsilon a^\frac{p}{p-1} + C_\epsilon b^p$ (with $C_\epsilon=(\frac{\epsilon p}{p-1})^{1-p} p^{-1}$) on the second term in the right hand side of \eqref{050105}, for any $\epsilon\in(0,1)$, one has
\begin{equation}\label{050107}
\begin{aligned}
      p\int_{\R^N}  \eta^{p-1}  u^{1+p(t-1)} & |\nabla u|^{p-2}  \nabla u   \cdot \nabla\eta \mathrm{d}x \leq  p \int_{\R^N} \eta^{p-1} u^{(p-1)(t-1)} |\nabla u|^{p-1}   |\nabla\eta|u^t \mathrm{d}x \\
      & \leq  p\epsilon \int_{\R^N} \eta^p u^{p(t-1)} |\nabla u|^p \mathrm{d}x + pC_\epsilon  \int_{\R^N} |\nabla\eta|^p u^{pt} \mathrm{d}x.
\end{aligned}
\end{equation}
By taking $p\epsilon=\frac{1+p(t-1)}{2t}$ and hence $pC_\epsilon=\left(2(p-1)t\right)^{p-1}(1+p(t-1))^{1-p}$ in \eqref{050107} and $\delta=2^{-p+1}$ in \eqref{050106}, and noting the fact that $1<t<(1+p(t-1))<pt$, it follows from \eqref{050105}, \eqref{050106} and \eqref{050107} that
\begin{equation}\label{050108}
\begin{aligned}
      & \quad \int_{\R^N} |\nabla u|^{p-2}  \nabla u \cdot \nabla\varphi \mathrm{d}x \\
      &\geq \frac{1+p(t-1)}{2}\int_{\R^N} |\nabla u|^p \eta^p u^{p(t-1)}\mathrm{d}x-\left[\frac{2(p-1)t}{1+p(t-1)}\right]^{p-1}\int_{\R^N} |\nabla\eta|^p u^{pt} \mathrm{d}x \\
      &\geq \frac{1+p(t-1)}{4t^p}\int_{\R^N} \left| \nabla(\eta u^t) \right|^p \mathrm{d}x-C \frac{1+p(t-1)}{2t^p}\int_{\R^N} |\nabla \eta|^p u^{pt} \mathrm{d}x-C\int_{\R^N} |\nabla\eta|^p u^{pt} \mathrm{d}x \\
      &\geq \frac{1}{4t^{p-1}} \| \nabla(\eta u^t) \|^p_{L^p(\R^N)} - C \left\|  |\nabla \eta|u^t \right\|^p_{L^p(\R^N)},
\end{aligned}
\end{equation}
where $C=C(N,p)>0$.

\smallskip

Now, substituting \eqref{050103}, \eqref{050104} and \eqref{050108} into \eqref{050102}, we obtain
\begin{equation}\label{050109+}
\begin{aligned}
    \frac{1}{4t^{p-1}} \| \nabla(\eta u^t) &\|^p_{L^p(\R^N)} \leq  C \left\|  |\nabla \eta|u^t \right\|^p_{L^p(\R^N)} \\
    &+ C_S^p \left\| V(x) \right\|_{L^{\frac{N}{p-s}}(\R^N \cap supp(\eta))} \left\| |x|^{-s} \right\|_{L^{\frac{N}{s}}(\R^N \cap supp(\eta))} \left\| \nabla(\eta u^t) \right\|^p_{L^p(\R^N)}\\
    &+ \mu C_S^{p-l} \left\| \frac{1}{|x|^p} \right\|_{L^{\frac{N}{p-l}}(\R^N \cap supp(\eta))}  \left\| \nabla(\eta u^t) \right\|^{p-l}_{L^p(\R^N)}  \left\| \eta u^t \right\|^l_{L^p(\R^N)}.
\end{aligned}
\end{equation}
Since $B_{3R_1}(y)\subset \R^N \setminus \{0\}$ and $supp(\eta) \subset B_h(y) \subset B_{2R_1}(y)$, we have
$$ |x|\geq |y|-h \geq R_1\geq\frac{h}{2}, \,\,\quad  \forall \,\, x\in B_h(y),$$
which implies that
\begin{equation}\label{050113}
     \left\| |x|^{-p} \right\|_{L^{\frac{N}{p-l}}(\R^N \cap supp(\eta))}\leq  \frac{2^{p}}{h^p}\cdot \left| B_h(y) \right|^\frac{p-l}{N} \leq 2^{p}|B_{1}(0)|^\frac{p-l}{N} h^{-l},
\end{equation}
and
\begin{equation}\label{050113+}
     \left\||x|^{-s} \right\|_{L^{\frac{N}{s}}(\R^N \cap supp(\eta))} \leq  2^{s}\frac{1}{h^s}\cdot \left| B_h(y) \right|^\frac{s}{N} \leq 2^{s}|B_{1}(0)|^\frac{s}{N}.
\end{equation}
Hence
\begin{equation}\label{050109}
\begin{aligned}
    \frac{1}{4t^{p-1}} \| \nabla(\eta u^t) \|^p_{L^p(\R^N)} &\leq  C \left\|  |\nabla \eta|u^t \right\|^p_{L^p(\R^N)} \\
    &+ 2^{s}|B_{1}(0)|^\frac{s}{N}\cdot C_S^p \left\| V(x) \right\|_{L^{\frac{N}{p-s}}(\R^N \cap supp(\eta))} \left\| \nabla(\eta u^t) \right\|^p_{L^p(\R^N)}\\
    &+ 2^{p}\mu |B_{1}(0)|^\frac{p-l}{N} \cdot C_S^{p-l} h^{-l}  \left\| \nabla(\eta u^t) \right\|^{p-l}_{L^p(\R^N)}  \left\| \eta u^t \right\|^l_{L^p(\R^N)}.
\end{aligned}
\end{equation}

\smallskip

Since $supp(\eta)\subset B_{2R_1}(y)$, it follows from \eqref{050101} that, for any $1<t\leq \bar{p}/p$,
\begin{equation}\label{050110}
  \| V \|_{L^\frac{N}{p-s}(\R^N \cap supp(\eta))} \leq \| V \|_{L^\frac{N}{p-s}(B_{2R_1}(y))} \leq \frac{1}{2^{3+s}|B_{1}(0)|^\frac{s}{N} C_S^p (\bar{p}/p)^{p-1}} \leq \frac{1}{2^{3+s}|B_{1}(0)|^\frac{s}{N} C_S^p t^{p-1}},
\end{equation}
which together with \eqref{050109} imply
\begin{equation}\label{050111}
\begin{aligned}
    \frac{1}{8t^{p-1}} \| \nabla(\eta u^t) \|^p_{L^p(\R^N)} &\leq  C \left\| |\nabla \eta|u^t \right\|^p_{L^p(\R^N)} \\
    &+ 2^{p}\mu |B_{1}(0)|^\frac{p-l}{N} C_S^{p-l} h^{-l}  \left\| \nabla(\eta u^t) \right\|^{p-l}_{L^p(\R^N)}  \left\| \eta u^t \right\|^l_{L^p(\R^N)}.
\end{aligned}
\end{equation}
Now, let
$$ z=\frac{\| \nabla(\eta u^t) \|_{L^p(\R^N)}}{\| \eta u^t \|_{L^p(\R^N)}} \,\quad\, \mbox{and} \,\,\quad \xi=\frac{\left\|  |\nabla \eta|u^t \right\|_{L^p(\R^N)}}{\| \eta u^t \|_{L^p(\R^N)}}, $$
then \eqref{050111} implies that
\begin{equation*}
     z^p \leq  C \left[ t^{p-1}\xi^p + t^{p-1}\mu h^{-l} z^{p-l} \right].
\end{equation*}
Applying Lemma \ref{tech-in-1} and using the fact that $l\in(0,p)$, we get
\begin{equation}\label{d-g-m-1}
     z   \leq  C_1 \left[ (t^{p-1}\xi^p)^\frac{1}{p} + \left(t^{p-1}\mu h^{-l} \right)^\frac{1}{l} \right]
     \leq C_2 t^\frac{2p}{l}\left(\xi +  h^{-1} \right),
\end{equation}
i.e.,
\begin{equation}\label{d-g-m-2}
     \| \nabla(\eta u^t) \|_{L^p(\R^N)} \leq  C_2 t^\frac{2p}{l}\left(\left\|  |\nabla \eta|u^t \right\|_{L^p(\R^N)} + h^{-1}\| \eta u^t \|_{L^p(\R^N)}\right),
\end{equation}
where $C_1=C_1(N,p,s)$ and $C_2=C_2(N,p,s,\mu)>0$. Using the Sobolev inequality, we obtain
\begin{equation}\label{050201}
     \| \eta u^t \|_{L^{p^\star}(\R^N)} \leq  C t^\frac{2p}{l}\left(\left\|  |\nabla \eta|u^t \right\|_{L^p(\R^N)} + h^{-1}\| \eta u^t \|_{L^p(\R^N)}\right),
\end{equation}
which yields that
\begin{equation}\label{050202}
   \left( \int_{B_{h^\prime}(y)} u^{p^\star t}(x) \mathrm{d}x \right)^\frac{1}{p^\star t} \leq \left( C \cdot t^\frac{2p}{l} \right)^\frac{1}{t} \left( \frac{1}{h-h^\prime} + \frac{1}{h} \right)^\frac{1}{t} \left( \int_{B_h(y)} u^{pt}(x) \mathrm{d}x \right)^\frac{1}{pt}.
\end{equation}

\smallskip

Let $\chi:=p^\star/p=\frac{N}{N-p}>1$, and for any $0<R\leq R_1$, denote $r_i:=R+\frac{2R}{2^i}$ for $i=1,2,\cdots,i_0$, where $\chi^{i_0}\leq \bar{p}/p \leq \chi^{i_0+1}$ and $i_0=i_0(\bar{p})\geq1$. Then $r_i-r_{i+1}=\frac{R}{2^i}$. By taking $h^\prime=r_{i+1}$, $h=r_i$ and $t=\chi^i$ in \eqref{050202}, for any $i=1,2,\cdots,i_0$, we get
\begin{equation}\label{050203}
   \left( \int_{B_{r_{i+1}}(y)} u^{p\chi^{i+1}}(x) \mathrm{d}x \right)^\frac{1}{p\chi^{i+1}} \leq \left( C \cdot {\chi^i}^\frac{2p}{l} \right)^\frac{1}{\chi^i} \left( \frac{2^i}{R} + \frac{1}{R} \right)^\frac{1}{\chi^i} \left( \int_{B_{r_i}(y)} u^{p\chi^i}(x) \mathrm{d}x \right)^\frac{1}{p\chi^i}.
\end{equation}
By iteration, we obtain that, for any $R\leq R_1$,
\begin{equation}\label{050204}
\begin{aligned}
   \left( \int_{B_{r_{i_0+1}}(y)} u^{p\chi^{i_0+1}}(x) \mathrm{d}x \right)^\frac{1}{p\chi^{i_0+1}}
   & \leq  C^\frac{1}{\chi^{i_0}} \cdot (2\chi^\frac{2p}{l})^\frac{i_0}{\chi^{i_0}} \cdot  \left(\frac{1}{R}\right)^\frac{1}{\chi^{i_0}} \left( \int_{B_{r_{i_0}}(y)} u^{p\chi^{i_0}}(x) \mathrm{d}x \right)^\frac{1}{p\chi^{i_0}} \\
   & \leq C^{\sum_{k=1}^{i_0}\frac{1}{\chi^k}} \cdot (2\chi^\frac{2p}{l})^{\sum_{k=1}^{i_0} \frac{k}{\chi^k}} \cdot  \left(\frac{1}{R}\right)^{\sum_{k=1}^{i_0} \frac{1}{\chi^k}}  \left( \int_{B_{2R}(y)} u^{p^\star}(x) \mathrm{d}x \right)^\frac{1}{p^\star}.
\end{aligned}
\end{equation}
Noting that
$$ \sum_{k=1}^\infty \frac{k}{\chi^k}<\infty,$$
$$ \sum_{k=1}^\infty \frac{1}{\chi^k}=\frac{N-p}{p},$$
and
$$ \sum_{k=1}^{i_0} \frac{1}{\chi^k}=\sum_{k=1}^\infty \frac{1}{\chi^k}-\sum_{k=i_0+1}^\infty \frac{1}{\chi^k}=\frac{N-p}{p}-\frac{N}{p\chi^{i_0+1}},$$
from \eqref{050204} we have
\begin{equation}\label{050205}
   \left( \int_{B_R (y)} u^{p\chi^{i_0+1}}(x) \mathrm{d}x \right)^\frac{1}{p\chi^{i_0+1}}
  \leq \frac{C}{R^{\frac{N-p}{p}-\frac{N}{p\chi^{i_0+1}}}}  \left( \int_{B_{2R}(y)} u^{p^\star}(x) \mathrm{d}x \right)^\frac{1}{p^\star}, \,\quad\,\, \forall \,\, R\leq R_1,
\end{equation}
where $C=C(N,p,s,\mu)>0$.

Combining \eqref{050205} with $\chi^{i_0}\leq \bar{p}/p \leq \chi^{i_0+1}$, we get
\begin{equation}\label{050206}
\begin{aligned}
   \|u\|_{L^{\bar{p}} (B_R(y))} & \leq \left[ \left( \int_{B_R(y)} 1 \mathrm{d}x \right)^{1-\frac{\bar{p}}{p\chi^{i_0+1}}} \cdot \left( \int_{B_R(y)} u^{p\chi^{i_0+1}} \mathrm{d}x \right)^{\frac{\bar{p}}{p\chi^{i_0+1}}} \right]^\frac{1}{\bar{p}} \\
   &\leq C \cdot R^{\frac{N}{\bar{p}} - \frac{N}{p\chi^{i_0+1}}} \cdot \|u\|_{L^{p\chi^{i_0+1}} (B_R(y))}\\
   &\leq C \cdot R^{\frac{N}{\bar{p}} - \frac{N}{p\chi^{i_0+1}}} \cdot \frac{1}{R^{\frac{N-p}{p}-\frac{N}{p\chi^{i_0+1}}}}  \left( \int_{B_{2R}(y)} u^{p^\star}(x) \mathrm{d}x \right)^\frac{1}{p^\star} \\
   &\leq \frac{C}{R^{\frac{N-p}{p}-\frac{N}{\bar{p}}}} \|u\|_{L^{p^\star}(B_{2R}(y))}, \,\,\,\,\,\,\,\, \forall R\leq R_1,
\end{aligned}
\end{equation}
where $C=C(N,p,s,\mu)>0$. This proves \eqref{loc-est-1}. Thus it follows from the arbitrariness of $\bar{p}\in[p^\star,+\infty)$ and finite covering theorem that $u\in L_{loc}^{r}(\R^N\setminus \{0\})$ for any $0<r<+\infty$.

\medskip

Next, we prove: for any $y \in \R^N \setminus \{0\}$, and for any $0<R\leq\frac{|y|}{3}$, if $R^{p-s-\frac{N}{q}}\|V\|_{L^{q}(B_{2R}(y))}\leq \widetilde{C}$ for some $\frac{N}{p-s}<q<+\infty$ and some constant $\widetilde{C}>0$, then there exists  $C=C\left(N,p,s,\mu,q,\widetilde{C}\right)>0$ such that
\begin{equation}\label{050207}
  \max_{B_R(y)} u(x) \leq \frac{C}{R^\frac{N}{\bar{p}}} \| u \|_{L^{\bar{p}}(B_{2R}(y))}, \quad\,\,\, \forall \,\, p^\star\leq\bar{p}<+\infty.
\end{equation}

\smallskip

To this end, for any $0<h^\prime<h<2R$, we take the cut-off function $\eta\in C_0^\infty(B_h(y))$, with $\eta=1$ in $B_{h^\prime}(y)$, $0\leq\eta\leq 1$, and $|\nabla \eta|\leq\frac{1}{h-h^\prime}$. Let $\varphi=\eta^p u^{1+p(t-1)}$ for all $t>1$. Note that $q>\frac{N}{p-s}$, so there exists $l\in(0,p-s)$ such that $q=\frac{N}{p-s-l}$.

\smallskip

Through similar method in deriving \eqref{050109+}, \eqref{050113}, \eqref{050113+} and \eqref{050109}, we get
\begin{equation}\label{050208}
\begin{aligned}
    & \frac{1}{4t^{p-1}} \| \nabla(\eta u^t) \|^p_{L^p(\R^N)} \leq  C \left\|  |\nabla \eta|u^t \right\|^p_{L^p(\R^N)} \\
    &+ \mu C_S^{p-l} \left\| \frac{1}{|x|^p} \right\|_{L^{\frac{N}{p-l}}(\R^N \cap supp(\eta))}  \left\| \nabla(\eta u^t) \right\|^{p-l}_{L^p(\R^N)}  \left\| \eta u^t \right\|^l_{L^p(\R^N)}\\
    &+ C_S^{p-l} \left\||x|^{-s} \right\|_{L^{\frac{N}{s}}(\R^N \cap supp(\eta))}
               \left\| V(x) \right\|_{L^{\frac{N}{p-s-l}}(\R^N \cap supp(\eta))}
               \left\| \nabla(\eta u^t) \right\|^{p-l}_{L^p(\R^N)}  \left\| \eta u^t \right\|^l_{L^p(\R^N)}\\
    &\leq  C \left[ \left\|  |\nabla \eta|u^t \right\|^p_{L^p(\R^N)}
    + \frac{1}{h^l} \left\| \nabla(\eta u^t) \right\|^{p-l}_{L^p(\R^N)}  \left\| \eta u^t \right\|^l_{L^p(\R^N)} \right],
\end{aligned}
\end{equation}
where $C=C\left(N,p,s,\mu,q,\widetilde{C}\right)>0$ and we have used the following
\begin{eqnarray*}\label{050212}
     &&\left\|V(x) \right\|_{L^{\frac{N}{p-s-l}}(\R^N \cap supp(\eta))}\leq
     \frac{1}{h^{l}} \left\|h^{l}V(x)\right\|_{L^{\frac{N}{p-s-l}}(B_h(y))}\leq \frac{1}{h^{l}} \left\|R^{p-s-\frac{N}{q}}V(x)\right\|_{L^{\frac{N}{p-s-l}}(B_{2R}(y))} \\
     &&\qquad\qquad\qquad\qquad\quad\,\,\,\,\quad\leq R^{p-s-\frac{N}{q}}\frac{1}{h^{l}}\left\|V(x)\right\|_{L^{\frac{N}{p-s-l}}(B_{2R}(y))}.
\end{eqnarray*}
Then, by Lemma \ref{tech-in-1} and a similar method in deriving \eqref{d-g-m-2}, for any $t>1$, we get
\begin{equation}\label{050209}
     \| \nabla(\eta u^t) \|_{L^p(\R^N)} \leq  C t^\frac{2p}{l}\left(\left\|  |\nabla \eta|u^t \right\|_{L^p(\R^N)} + h^{-1}\| \eta u^t \|_{L^p(\R^N)}\right),
\end{equation}
i.e.,
\begin{equation}\label{050210}
     \|u^t\|_{L^{p^{\star}}(B_{h^\prime}(y))} \leq  C t^\frac{2p}{l}\left( \frac{2}{h-h^\prime}+ h^{-1}\right) \| u^t \|_{L^p(B_h(y))}.
\end{equation}
Finally, by a similar method in deriving \eqref{050205}, letting $\chi:=p^\star/p=\frac{N}{N-p}>1$, and taking $r_i=R+\frac{2R}{2^i}$, $h^\prime=r_{i+1}$, $h=r_i$ and $t=\chi^i$ ($i=i_{0},i_{0}+1,\cdots$) in \eqref{050210}, we obtain,  for any $i\geq i_{0}$,
\begin{eqnarray}\label{050211}
 \nonumber && \left( \int_{B_R (y)} u^{p\chi^{i+1}}(x) \mathrm{d}x \right)^\frac{1}{p\chi^{i+1}}
  \leq \frac{C}{R^{\frac{N}{p\chi^{i_{0}}}-\frac{N}{p\chi^{i+1}}}}  \left( \int_{B_{2R}(y)} u^{p\chi^{i_{0}}}(x) \mathrm{d}x \right)^\frac{1}{p\chi^{i_{0}}}\\
  && \qquad\qquad\qquad\qquad\qquad\quad\,\,\,\,\leq \frac{C}{R^{\frac{N}{\bar{p}}-\frac{N}{p\chi^{i+1}}}}  \left( \int_{B_{2R}(y)} u^{\bar{p}}(x) \mathrm{d}x \right)^\frac{1}{\bar{p}}, \qquad \forall \,\, p^\star\leq\bar{p}<+\infty,
\end{eqnarray}
where $i_0=i_0(\bar{p})$ satisfies $p\chi^{i_{0}}\leq\bar{p}\leq p\chi^{i_{0}+1}$. Letting $i\to\infty$, one has
\begin{equation*}
  \max_{B_R(y)} u(x) \leq \frac{C}{R^\frac{N}{\bar{p}}} \| u \|_{L^{\bar{p}}(B_{2R}(y))}, \,\,\quad\, \forall \,\, p^\star\leq\bar{p}<+\infty,
\end{equation*}
where $C=C\left(N,p,s,\mu,q,\widetilde{C}\right)>0$. This proves \eqref{050207}. Consequently, if $V\in L^{q}_{loc}(\R^{N}\setminus \{0\})$ for some $\frac{N}{p-s}<q<+\infty$, by using the finite covering theorem, we can obtain $u\in L^{\infty}_{loc}(\R^N \setminus \{0\})$. Then, using the standard $C^{1,\alpha}$-regularity estimates (see e.g. \cite{ED83,KM,PT}), we have $u\in C^{1,\alpha}(\R^N\setminus \{0\})\bigcap L^{\infty}_{loc}(\R^N \setminus\{0\})$ for some $0<\alpha<\min\{1,1/(p-1)\}$.

This completes our proof of Lemma \ref{lm:l-b-r}.
\end{proof}

\begin{cor}\label{lm:l-b-1-cor}
Let $u\in D^{1,p}(\R^N)$ be a nonnegative solution to the generalized equation \eqref{gen-eq} with $0\leq\mu<\bar{\mu}$, $1<p<N$, $0< s < p$ and $0\leq V(x)\in L^\frac{N}{p-s}(\R^N)$. Then $u\in C^{1,\alpha}(\R^N\setminus \{0\})\bigcap L^{\infty}_{loc}(\R^N \setminus \{0\})$ for some $0<\alpha<\min\{1,\frac{1}{p-1}\}$. Moreover, for any $y \in \R^N \setminus \{0\}$ and for any $R>0$ satisfying $B_{3R}(y)\subseteq \R^N\setminus \{0\} $, there holds
\begin{equation}\label{est-1-cor}
  \max_{B_R(y)} u(x) \leq \frac{C}{R^\frac{N}{\bar{p}}} \| u \|_{L^{\bar{p}}(B_{2R}(y))}, \qquad \forall \,\, p^\star\leq\bar{p}<+\infty,
\end{equation}
where $C=C\left(N,p,s,\mu,\|V\|_{L^{\frac{N}{p-s}}(\R^N)}\right)>0$.
\end{cor}
\begin{proof}
In fact, for $0<s<p$, the estimate \eqref{050104} in the proof of Lemma \ref{lm:l-b-r} can be replaced by the following
\begin{equation}\label{050208-cor}
\begin{aligned}
      &\quad \int_{\R^N} V(x)|x|^{-s} u^{p-1} \varphi \mathrm{d}x
      = \int_{\R^N} V(x) |x|^{-s} (u^{t}\eta)^{p-l} (u^{t}\eta)^{l} \mathrm{d}x \\
      &\leq  \left\| V \right\|_{L^{\frac{N}{p-s}}(\R^N)}
             \left\| |x|^{-s} \right\|_{L^{\frac{N}{s-l}}(\R^N \cap supp(\eta))}
             \left\| (u^{t}\eta)^{p-l} \right\|_{L^{\frac{Np}{(N-p)(p-l)}}(\R^N)}
             \left\| (u^{t}\eta)^{l} \right\|_{L^{\frac{p}{l}}(\R^N)} \\
      &\leq  C_S^{p-l} \left\| V \right\|_{L^{\frac{N}{p-s}}(\R^N)}
                  \left\| |x|^{-s} \right\|_{L^{\frac{N}{s-l}}(\R^N \cap supp(\eta))}
                  \left\| \nabla(\eta u^t) \right\|^{p-l}_{L^p(\R^N)(\R^N)}
                  \left\| \eta u^t \right\|^l_{L^p(\R^N)}.
\end{aligned}
\end{equation}
Similar to \eqref{050113}, one has, for any fixed $l\in(0,s)$,
\begin{equation}\label{a3}
     \left\| |x|^{-s} \right\|_{L^{\frac{N}{s-l}}(\R^N \cap supp(\eta))}\leq  \frac{2^{s}}{h^s}\cdot \left| B_h(y) \right|^\frac{s-l}{N} \leq 2^{s}|B_{1}(0)|^\frac{s-l}{N} h^{-l}.
\end{equation}
Then, similar to the method used in deriving \eqref{050109} or \eqref{050208}, we deduce from \eqref{050208-cor}, \eqref{a3} and \eqref{050113+} that, for any $t>1$,
\begin{equation*}
\begin{aligned}
    & \frac{1}{4t^{p-1}} \| \nabla(\eta u^t) \|^p_{L^p(\R^N)} \leq  C \left\|  |\nabla \eta|u^t \right\|^p_{L^p(\R^N)} \\
    &+ \mu C_S^{p-l} \left\| \frac{1}{|x|^p} \right\|_{L^{\frac{N}{p-l}}(\R^N \cap supp(\eta))}  \left\| \nabla(\eta u^t) \right\|^{p-l}_{L^p(\R^N)}  \left\| \eta u^t \right\|^l_{L^p(\R^N)}\\
    &+ C_S^{p-l}  \left\| V \right\|_{L^{\frac{N}{p-s}}(\R^N)}
               \left\||x|^{-s} \right\|_{L^{\frac{N}{s-l}}(\R^N \cap supp(\eta))}
               \left\| \nabla(\eta u^t) \right\|^{p-l}_{L^p(\R^N)}  \left\| \eta u^t \right\|^l_{L^p(\R^N)}\\
    &\leq  C \left[ \left\|  |\nabla \eta|u^t \right\|^p_{L^p(\R^N)}
    + \frac{1}{h^l} \left\| \nabla(\eta u^t) \right\|^{p-l}_{L^p(\R^N)}  \left\| \eta u^t \right\|^l_{L^p(\R^N)} \right],
\end{aligned}
\end{equation*}
where $0<l<s<p$ and $C=C\left(N,p,\mu,s,l,\left\|V\right\|_{L^{\frac{N}{p-s}}(\R^N)}\right)>0$. Then, through a similar method in deriving \eqref{050209} (or \eqref{d-g-m-2}), by Lemma \ref{tech-in-1}, we obtain
\begin{equation}\label{050209-cor}
     \| \nabla(\eta u^t) \|_{L^p(\R^N)} \leq  C t^\frac{2p}{l}\left(\left\|  |\nabla \eta|u^t \right\|_{L^p(\R^N)} + h^{-1}\| \eta u^t \|_{L^p(\R^N)}\right), \,\,\quad\, \forall \,\, t>1.
\end{equation}
Finally, we set $l=\frac{s}{2}$ and complete the proof of Corollary \ref{lm:l-b-1-cor} in similar way as that of Lemma \ref{lm:l-b-r}.
\end{proof}

\begin{lem}\label{lm:l-b-2}
Let $u$ be a nonnegative solution to the generalized equation \eqref{gen-eq} with $0\leq\mu<\bar{\mu}$ $1<p<N$, $0\leq s <p$ and $0\leq V(x)\in L^\frac{N}{p-s}(\R^N)$.
For any $\bar{p}>p^\star$, there exists a constant $\epsilon_1=\epsilon_1(N,\mu,s,p,\bar{p})>0$ such that, \\
$(i)$ for any $\rho_1\in(0,1)$ satisfying $ \|V\|_{L^\frac{N}{p-s} (B_{\rho_1}(0))}\leq \epsilon_1$, it holds
\begin{equation}\label{loc-est-3}
  \|u\|_{L^{\bar{p}}(B_{4R}(0)\setminus B_{R/4}(0))} \leq \frac{C}{R^{\frac{N-p}{p}-\frac{N}{\bar{p}}}} \|u\|_{L^{p^\star}(B_{8R}(0)\setminus B_{R/8}(0))}, \,\,\,\quad \forall \,\, R\leq \frac{\rho_1}{8};
\end{equation}
$(ii)$ for any $\rho_1\in(0,1)$ satisfying $\|V\|_{L^\frac{N}{p-s} (\R^N\setminus B_{1/{\rho_1}}(0))} \leq \epsilon_1$, it holds
\begin{equation}\label{loc-est-4}
  \|u\|_{L^{\bar{p}}(\R^N\setminus B_{2R}(0))} \leq \frac{C}{R^{\frac{N-p}{p}-\frac{N}{\bar{p}}}} \|u\|_{L^{p^\star}(\R^N\setminus B_{R}(0))}, \,\,\,\quad \forall \,\, R\geq \frac{1}{\rho_1},
\end{equation}
where $C=C(N,p,s,\mu)>0$.
\end{lem}

\begin{proof}
Fixing any $\bar{p}>p^\star$, since $V(x)\in L^\frac{N}{p-s}(\R^N)$, then for $\epsilon_1:=\frac{1}{8^{2s+1}|B_{1}(0)|^\frac{s}{N} (\bar{p}/p)^{p-1}C_S^p}$,
\begin{equation}\label{050215}
\|V\|_{L^\frac{N}{p-s} (B_{\rho_1}(0))}+ \|V\|_{L^\frac{N}{p-s} (\R^N\setminus B_{1/{\rho_1}}(0))} \leq \frac{1}{8^{2s+1}|B_{1}(0)|^\frac{s}{N} (\bar{p}/p)^{p-1}C_S^p}
\end{equation}
holds for any $\rho_1\in(0,1)$ small enough, where $|B_{1}(0)|$ is the volume of the unit ball in $\R^N$ and $C_S$ is Sobolev embedding constant.

\smallskip

First, we prove $(i)$. Let $0<R\leq \rho_1/8$. For any $1\leq h^\prime < h \leq 2$, define $D_h=B_{4hR}(0)\setminus B_{\frac{R}{4h}}(0)$, and let $\eta\in C_0^\infty(\R^N)$ be a cut-off function such that $0\leq\eta\leq 1$ in $\R^N$, $\eta=1$ on $D_{h^\prime}$, $\eta=0$ on $\R^N\setminus D_h$ and $|\nabla\eta|<\frac{32}{(h-h^\prime)R}$. We can take the test function $\varphi=\eta^p u^{1+p(t-1)}$ in \eqref{gen-eq}, where $1<t\leq\frac{\bar{p}}{p}$.

\smallskip

For any fixed $l\in(0,p)$, one has
\[\left\| |x|^{-p} \right\|_{L^{\frac{N}{p-l}}(\R^N \cap supp(\eta))}^\frac{1}{l}\leq  \left(\frac{4h}{R}\right)^\frac{p}{l} \cdot \left| B_{4hR}(0)\setminus B_{\frac{R}{4h}}(0)\right|^\frac{p-l}{Nl} \leq |B_{1}(0)|^{\frac{p-l}{Nl}}(4h)^{\frac{2p}{l}-1}R^{-1},\]
and
\[\left\| |x|^{-s} \right\|_{L^{\frac{N}{s}}(\R^N \cap supp(\eta))}\leq  \left(\frac{4h}{R}\right)^{s} \cdot \left| B_{4hR}(0)\setminus B_{\frac{R}{4h}}(0)\right|^\frac{s}{N} \leq |B_{1}(0)|^{\frac{s}{N}}(4h)^{2s}.\]
Next, by a similar method in deriving \eqref{050201}, we can deduce
\begin{equation}\label{050216}
\| \eta u^t \|_{L^{p^\star}(\R^N)} \leq  C t^\frac{2p}{l}\left(\left\|  |\nabla \eta|u^t \right\|_{L^p(\R^N)} + R^{-1}\| \eta u^t \|_{L^p(\R^N)}\right),
\end{equation}
i.e.,
\begin{equation}\label{050217}
\|u^t\|_{L^{p^\star}(D_{h^\prime})} \leq  C t^\frac{2p}{l}\left( \frac{2}{h-h^\prime}+1 \right)\frac{1}{R} \| \eta u^t \|_{L^p(D_h)},
\end{equation}
where $C=C(N,p,s,\mu)>0$. Now, using a similar iterative method in deriving \eqref{050205}, taking $h_i=1+\frac{2}{2^i}$, $h=h_i$, $h^\prime=h_{i+1}$ and  $t=\chi^i$ in \eqref{050217} for $i=1,2,\cdots,i_0$ (where $i_0=i_0(\bar{p})>0$ such that $\chi^{i_0}\leq \bar{p}/p \leq \chi^{i_0+1}$), we can derive
\begin{equation}\label{050218}
   \left( \int_{D_1} u^{p\chi^{i_0+1}}(x) \mathrm{d}x \right)^\frac{1}{p\chi^{i_0+1}}
  \leq \frac{C}{R^{\frac{N-p}{p}-\frac{N}{p\chi^{i_0+1}}}}  \left( \int_{D_2} u^{p^\star}(x) \mathrm{d}x \right)^\frac{1}{p^\star}, \,\,\,\quad  \forall \,\, R\leq \frac{\rho_1}{8}.
\end{equation}
Then, from $\chi^{i_0}\leq \bar{p}/p \leq \chi^{i_0+1}$, we can get
\begin{equation}\label{050226}
\begin{aligned}
    \|u\|_{L^{\bar{p}} (B_{4R}(0)\setminus B_{R/4}(0))} & \leq \|u\|_{L^{\bar{p}} (D_1)} \\
   &\leq C \cdot R^{\frac{N}{\bar{p}} - \frac{N}{p\chi^{i_0+1}}} \cdot \|u\|_{L^{p\chi^{i_0+1}} (D_1)}\\
   &\leq C \cdot R^{\frac{N}{\bar{p}} - \frac{N}{p\chi^{i_0+1}}} \cdot \frac{1}{R^{\frac{N-p}{p}-\frac{N}{p\chi^{i_0+1}}}}  \left( \int_{D_2} u^{p^\star}(x) \mathrm{d}x \right)^\frac{1}{p^\star} \\
   &\leq \frac{C}{R^{\frac{N-p}{p}-\frac{N}{\bar{p}}}} \|u\|_{L^{p^\star}(B_{8R}(0)\setminus B_{R/8}(0))}, \,\,\,\,\,\,\,\,\,\,\, \forall \,\,  R\leq \frac{\rho_1}{8},
\end{aligned}
\end{equation}
where $C=C(N,p,s,\mu)>0$. Thus we have proved $(i)$.

\smallskip

Next, we are to show $(ii)$. Let $R\geq 1/\rho_1$, for any $2R\geq h'>h\geq R$, and take $\eta\in C^2(\R^N)$ such that $\eta=0$ in $B_{h} (0)$, $\eta=1$ in $\R^N\setminus B_{h'}(0)$, $0\leq\eta\leq 1$ and $|\nabla \eta| \leq \frac{2}{h'-h}$. Choose the test function $\varphi=\eta^p u^{1+p(t-1)}$ in \eqref{gen-eq}, where $1<t\leq \bar{p}/p$.

\smallskip

For any fixed $l\in (0,1)$, one has
$$\left\| \frac{1}{|x|^p} \right\|_{L^{\frac{N}{p-l}}(\R^N \cap supp(\eta))}^\frac{1}{l}
\leq \left( \int_{\R^N\setminus B_R(0)} \frac{1}{|x|^\frac{pN}{p-l}} \mathrm{d}x \right)^\frac{p-l}{Nl}
\leq C \left( \int_R^\infty \frac{r^{N-1}}{r^\frac{pN}{p-l}} dr \right)^\frac{p-l}{Nl}
\leq \frac{C}{R},$$
where $C=C(N,p)>0$. Consequently, by using \eqref{050104} if $s=0$ and \eqref{050208-cor} if $0<s<p$, we can prove via a similar way and iterative method in deriving \eqref{050205} that
\begin{equation}\label{050218'}
   \left( \int_{\R^N\setminus B_{2R}(0)} u^{p\chi^{i}}(x) \mathrm{d}x \right)^\frac{1}{p\chi^{i}}
  \leq \frac{C}{R^{\frac{N-p}{p}-\frac{N}{p\chi^{i}}}}  \left( \int_{\R^N\setminus B_{R}(0)} u^{p^\star}(x) \mathrm{d}x \right)^\frac{1}{p^\star}, \,\,\,\quad  \forall \,\, R\geq \frac{1}{\rho_{1}},
\end{equation}
where $i=i_{0}$ or $i_{0}+1$ with $i_0=i_0(\bar{p})>0$ such that $\chi^{i_0}\leq \bar{p}/p \leq \chi^{i_0+1}$. Next, let $\frac{1}{r_2}=\left(\frac{\bar{p}}{p\chi^{i_0}}-1\right)\frac{N-p}{p}\in(0,1)$, $m_{1}=\frac{1}{r_1}p\chi^{i_0}$ and $m_{2}=\frac{1}{r_2}p\chi^{i_0+1}$,
where $\frac{1}{r_1}=1-\frac{1}{r_2}$. Noting that a direct calculation yields $m_{1}+m_{2}=\bar{p}$, by H\"{o}lder's inequality and \eqref{050218'}, we get
\begin{equation}\label{050226}
\begin{aligned}
   &\quad \|u\|_{L^{\bar{p}} (\R^N\setminus B_{2R}(0))} = \left(\int_{\R^N\setminus B_{2R}(0)} u^{m_1} u^{m_2} \mathrm{d}x \right)^\frac{1}{\bar{p}}\\
   &\leq \left(\int_{\R^N\setminus B_{2R}(0)} u^{m_1 r_1}\mathrm{d}x \right)^\frac{1}{r_1 \bar{p}} \cdot \left(\int_{\R^N\setminus B_{2R}(0)} u^{m_2 r_2} \mathrm{d}x \right)^\frac{1}{r_2 \bar{p}}\\
   &=\|u\|^{\frac{m_1}{\bar{p}}}_{L^{p\chi^{i_0}} (\R^N\setminus B_{2R}(0))} \cdot \|u\|^{\frac{m_2}{\bar{p}}}_{L^{p\chi^{i_0+1}} (\R^N\setminus B_{2R}(0))}\\
   &\leq C \left(\frac{1}{R^{\frac{N-p}{p}-\frac{N}{p\chi^{i_0}}}} \right)^{\frac{m_1}{\bar{p}}} \left(\frac{1}{R^{\frac{N-p}{p}-\frac{N}{p\chi^{i_0+1}}}} \right)^{\frac{m_2}{\bar{p}}} \|u\|^{\frac{m_1}{\bar{p}}}_{L^{p^\star} (\R^N\setminus B_{R}(0))} \cdot \|u\|^{\frac{m_2}{\bar{p}}}_{L^{p^\star} (\R^N\setminus B_{R}(0))}\\
   &\leq \frac{C}{R^{\frac{N-p}{p}-\frac{N}{\bar{p}}}} \|u\|_{L^{p^\star} (\R^N\setminus B_{R}(0))},\,\,\,\,\,\,\,\,\,\, \forall \,\, R\geq \frac{1}{\rho_1},
\end{aligned}
\end{equation}
where $C=C(N,p,s,\mu)>0$. This proves $(ii)$ and finishes our proof of Lemma \ref{lm:l-b-2}.
\end{proof}

\begin{lem}\label{lm:l-b-3}
Let $u$ be a nonnegative solution to the generalized equation \eqref{gen-eq} with $0\leq\mu< \bar{\mu}$, $1<p<N$, $0\leq s <p$ and $0\leq V(x)\in L^\frac{N}{p-s}(\R^N)$.
There exists a constant $\epsilon_2=\epsilon_2(N,p,\mu)>0$ such that, \\
$(i)$ for any $\rho_2\in(0,1)$ satisfying
$ \|V\|_{L^\frac{N}{p-s} (B_{\rho_2}(0))}\leq \epsilon_2$,
it holds
\begin{equation}\label{loc-est-5}
  \|u\|_{L^{p^\star}(B_R (0))} \leq C R^{\tau_1}, \,\,\,\quad \forall \,\, R \leq \rho_2;
\end{equation}
$(ii)$ for any $\rho_2\in(0,1)$ satisfying
$\|V\|_{L^\frac{N}{p-s} (\R^N\setminus B_{1/{\rho_2}}(0))} \leq \epsilon_2$, it holds
\begin{equation}\label{loc-est-6}
  \|u\|_{L^{p^\star}(\R^N\setminus B_{R}(0))} \leq C R^{-\tau_2}, \,\,\,\quad \forall \,\, R\geq \frac{1}{\rho_2},
\end{equation}
where $C=C(N,p,s,\mu,\rho_2,\|u\|_{L^{p^\star}(\R^N)})>0$ and $\tau_i =\tau_i(N,p,s,\mu)>0$, $i=1,2$.
\end{lem}

\begin{proof}
We only need to prove $(i)$. The conclusion in $(ii)$ can be proved in similar way (c.f. also \cite{CDPSYS} and \cite[Lemma 3.2]{DLL}), so we omit the details.

To this aim, for $R>0$, let $\eta\in C_0^\infty (\R^N)$ be a cut-off function such that $\eta=1$ in $B_{R/2}(0)$, $\eta=0$ in $\R^N\setminus B_R(0)$, $0\leq\eta\leq1$ and $|\nabla \eta|\leq 4/R $.

\smallskip

By density argument, we take $\varphi=\eta^p u$ as test function in \eqref{gen-eq} and get
\begin{equation*}
  \int_{\R^N} |\nabla u|^{p-2}\nabla u \cdot \nabla \varphi \mathrm{d}x - \mu \int_{\R^N} \frac{1}{|x|^p} u^{p-1} \varphi \mathrm{d}x = \int_{\R^N} V(x)|x|^{-s}u^{p-1} \varphi \mathrm{d}x.
\end{equation*}
By similar methods in deriving \eqref{050106} and \eqref{050107}, for any $\delta\in(0,1)$, one has, there exists a constant $C_\delta$ such that
\begin{equation*}
\begin{aligned}
  \int_{\R^N} |\nabla u|^{p-2}\nabla u \cdot \nabla \varphi \mathrm{d}x
  &=\int_{\R^N} \eta^p |\nabla u|^p \mathrm{d}x + p \int_{\R^N} \eta^{p-1} u |\nabla u|^{p-2}\nabla u \cdot \nabla \eta \mathrm{d}x \\
  &\geq (1-\delta)\int_{\R^N} |\nabla (\eta u)|^p \mathrm{d}x - C_\delta \int_{\R^N} |\nabla\eta|^p u^p \mathrm{d}x,
\end{aligned}
\end{equation*}
From Remark \ref{re-HDSI}, we deduce that
$$ \mu \int_{\R^N} \frac{1}{|x|^p} u^{p-1} \varphi \mathrm{d}x = \mu \int_{\R^N} \frac{1}{|x|^p} \eta^p u^p \mathrm{d}x \leq  \frac{\mu}{\bar{\mu}} \int_{\R^N} |\nabla (\eta u)|^p \mathrm{d}x.$$
Now, choosing $\delta=\delta(N,p,\mu)>0$ sufficiently small such that $1-\delta-\mu/\bar{\mu}>0$, we have
\begin{equation*}
\begin{aligned}
  \int_{\R^N} V(x)|x|^{-s} \eta^p u^p \mathrm{d}x
  &\geq (1-\delta-\mu/\bar{\mu})\int_{\R^N} |\nabla (\eta u)|^p \mathrm{d}x - C_\delta \int_{\R^N} |\nabla\eta|^p u^p \mathrm{d}x \\
  &\geq C_1 \|\nabla(\eta u)\|^p_{L^{p}(\R^N)} - C_\delta \cdot (4/R)^p \cdot \|u\|^p_{L^p(B_R(0)\setminus B_{R/2}(0))}\\
  &\geq C_1 \|\nabla(\eta u)\|^p_{L^{p}(\R^N)} - C_2 \|u\|^p_{L^{p^\star}(B_R(0)\setminus B_{R/2}(0))},
\end{aligned}
\end{equation*}
where $C_i=C_i(N,p,\mu)>0$, $i=1,2$. On the other hand, since $V(x)\in L^\frac{N}{p-s}(\R^N)$, then there exists a sufficiently small constant $\rho_2\in(0,1)$ such that
$$\|V\|_{L^\frac{N}{p-s} (B_{\rho_2}(0))}+ \|V\|_{L^\frac{N}{p-s} (\R^N\setminus B_{1/{\rho_2}}(0))} \leq \frac{C_1}{2C^p_H},$$
which together with H\"{o}lder's inequality and Remark \ref{re-HDSI} (with $\tau=Np/(N-p+s)$, $\gamma=-s/p$ and $p=p$) imply that
\begin{equation*}
\begin{aligned}
\int_{\R^N} V(x) |x|^{-s} \eta^p u^p \mathrm{d}x
&\leq  \|V\|_{L^\frac{N}{p-s} (\R^N \cap supp(\eta))} \cdot \||x|^{-s}\eta^p u^p\|_{L^\frac{N}{N-p+s} (\R^N)} \\
&\leq   C^p_H \|V\|_{L^\frac{N}{p-s} (\R^N \cap supp(\eta))} \cdot \|\nabla(\eta u)\|^p_{L^{p}(\R^N)}\\
&\leq \frac{C_1}{2} \|\nabla(\eta u)\|^p_{L^{p}(\R^N)}, \,\,\,\,\,\,\,\,\,\,\,\,\,\,\,\,\,\,\,\,\,\, \forall R \leq \rho_2 < 1.
\end{aligned}
\end{equation*}
As a consequence, we obtain that
$$ \|u\|^p_{L^{p^\star}(B_{R/2}(0))} \leq \|\eta u\|^p_{L^{p^\star}(\R^N)}\leq C_S^p \|\nabla(\eta u)\|^p_{L^{p}(\R^N)} \leq C \|u\|^p_{L^{p^\star}(B_R(0)\setminus B_{R/2}(0))},\,\,\,\,\,\,\, \forall \,\, R \leq \rho_2 < 1,$$
i.e.,
$$ \|u\|_{L^{p^\star}(B_{R/2}(0))} \leq C \|u\|_{L^{p^\star}(B_R(0)\setminus B_{R/2}(0))}, \,\,\,\,\,\,\, \forall \,\, R \leq \rho_2 < 1,$$
where $C=C(N,p,s,\mu)>0$.

\smallskip

Define $\Psi(R):=\| u \|_{L^{p^{\star}}(B_R(0))}$ for $0<R \leq \rho_2 < 1$, then it follows that
\begin{equation}\label{050306}
\Psi(R/2) \leq \theta \Psi(R), \,\,\,\ \forall\ R \leq \rho_2 < 1,
\end{equation}
where $\theta:=\left(\frac{C^{p^{\star}}}{C^{p^\star}+1}\right)^{\frac{1}{p^{\star}}}\in(0,1)$ depending only on $N$, $p$, $s$ and $\mu$. For any $0<|x|\leq \rho_2<1$, there is a $k\in \mathbb{N}^+$ such that
$$ 2^{-k}\rho_2 \leq |x| \leq 2^{-k+1}\rho_2,$$
which combined with the definition of $\Psi(R)$ and \eqref{050306} yields that
$$ \Psi(|x|) \leq \Psi(2^{-k+1}\rho_2)\leq \theta^{k-1} \Psi(\rho_2) \leq \theta^{\log_2\rho_2 - \log_2|x|} \theta^{-1} \Psi(\rho_2),$$
in which we have also used the fact that $\log_2 |x| \geq \log_2 (2^{-k}\rho_2)=-k+\log_2 \rho_2$ and $\theta\in(0,1)$.
Since
$$ \theta^{-\log_2|x|}=2^{\log_2\theta \cdot (-\log_2|x|)}= |x|^{-\log_2\theta},$$
and
$$ \tau_1:= -\log_2\theta>0,$$
we obtain
$$ \Psi(|x|)\leq C|x|^{\tau_1}, \,\,\,\,\,\,\,\, \forall \,\, |x|\leq\rho_2<1.$$
Thus, we have shown that there exists a $\tau_1=\tau_1(N,p,s,\mu)>0$ such that
$$\|u\|_{L^{p^\star}(B_R (0))} \leq C R^{\tau_1}, \,\,\,\quad \forall \,\, R \leq \rho_2,$$
where $C=C(N,p,s,\mu,\rho_2,\|u\|_{L^{p^\star}(\R^N)})>0$. This finishes our proof of Lemma \ref{lm:l-b-3}.
\end{proof}

\subsection{Proofs of Theorems \ref{th2.1.1} and \ref{th2.1++}}\label{sc3.2}

\begin{proof}[The proof of Theorem \ref{th2.1.1}]  Theorem \ref{th2.1.1} can be deduced directly from Lemmas \ref{lm:l-b-r}, \ref{lm:l-b-2} and \ref{lm:l-b-3}, and Corollary \ref{lm:l-b-1-cor}.
\end{proof}

Next, we will provide better preliminary decay estimates near the origin and the infinity for solutions to the generalized equation \eqref{gen-eq}, i.e., the proof of Theorem \ref{th2.1++}. Although the better preliminary decay estimates are not sharp, but they will play a key role in verifying the conditions on $u$ and $V(x)$ in Theorem \ref{th2.1}.

\begin{proof}[The proof of Theorem \ref{th2.1++}]
For any $|y|\leq \frac{\rho}{16}$ or $|y|\geq \frac{16}{\rho}$, let $r:=\min\left\{1,\frac{|y|}{4}\right\}$ and $R:=\frac{|y|}{4}$. In the case $s=0$, from the assumption \eqref{est-bdd-1+}, we know that, there exists some $l\in(0,p)$ such that $q=\frac{N}{p-l}\in (\frac{N}{p}, +\infty)$ and
$$\|r^{p-\frac{N}{q}} V(x)\|_{L^{q}(B_{2r}(y))}=r^l \|V(x)\|_{L^{\frac{N}{p-l}}(B_{2r}(y))}\leq \tilde{C}, \,\,\,\quad \forall \,\, |y|\leq \frac{\rho}{16} \,\, \mbox{or} \,\, |y|\geq \frac{16}{\rho},$$
where $\tilde{C}$ is independent of $y$. If $s=0$ and $|y|\geq4$, then $r=1$, Lemma \ref{lm:l-b-r} implies that
\begin{equation}\label{est-6.6}
  u(y) \leq \max_{B_1(y)} u(x) \leq \| u \|_{L^{\bar{p}}(B_{2}(y))},\,\,\,\quad \forall \,\, |y|\geq \frac{16}{\rho}, \quad \forall \,\, p^\star\leq \bar{p}<+\infty,
\end{equation}
where $C=C(N, p, s, \mu, q, \tilde{C})>0$ is independent of $y$; if $s=0$ and $|y|\leq4$, then $r=R$, it follows from Lemma \ref{lm:l-b-r} with $\bar{p}=p^\star$ that
\begin{equation}\label{est-6.6'}
  u(y) \leq \max_{B_R(y)} u(x) \leq \frac{C}{R^\frac{N-p}{p}} \| u \|_{L^{p^{\star}}(B_{2R}(y))},\,\,\,\quad \forall \,\, |y|\leq \frac{\rho}{16},
\end{equation}
where $C=C(N, p, s, \mu, q, \tilde{C})>0$ is independent of $y$; if $0<s<p$, from Corollary \ref{lm:l-b-1-cor}, we deduce that
\begin{equation}\label{est-6.6''}
  u(y) \leq \max_{B_R(y)} u(x) \leq \frac{C}{R^\frac{N-p}{p}} \| u \|_{L^{p^\star}(B_{2R}(y))},\,\,\,\quad \forall \,\, |y|\leq \frac{\rho}{16} \,\, \mbox{or} \,\, |y|\geq \frac{16}{\rho},
\end{equation}
where $C=C(N, p, s, \mu,\|V\|_{L^{\frac{N}{p-s}}}(\mathbb{R}^{N}))>0$ is independent of $y$.

\smallskip

On one hand, from Lemma \ref{lm:l-b-2}, we derive that
\begin{equation}\label{loc-est-4'}
  \|u\|_{L^{\bar{p}}(\R^N\setminus B_{2R}(0))} \leq \frac{C}{R^{\frac{N-p}{p}-\frac{N}{\bar{p}}}} \|u\|_{L^{p^\star}(\R^N\setminus B_{R}(0))}, \,\,\,\quad \forall \,\, R\geq \frac{1}{\rho},  \quad \forall \,\, p^\star\leq \bar{p}<+\infty,
\end{equation}
where $C=C(N,p,s,\mu)>0$. On the other hand, by Lemma \ref{lm:l-b-3}, we deduce that, there exists a positive constant $C=C(N,p,s,\mu,\rho,\|u\|_{L^{p^\star}(\R^N)})>0$ such that
\begin{equation}\label{loc-est6.6}
  \|u\|_{L^{p^\star}(B_{6R} (0))} \leq C R^{\tau_1}, \,\,\,\quad \forall \,\, 6R \leq \rho,
\end{equation}
and
\begin{equation}\label{loc-est6.7}
  \|u\|_{L^{p^\star}(\R^N\setminus B_{2R}(0))} \leq C R^{-\tau_2}, \,\,\,\quad \forall \,\, 2R\geq \frac{1}{\rho},
\end{equation}
where $\tau_i =\tau_i(N,p,s,\mu)>0$, $i=1,2$. Finally, recalling that $R=|y|/4$, by substituting \eqref{loc-est-4'} and \eqref{loc-est6.7} into \eqref{est-6.6} and taking $\bar{p}$ large enough such that $\frac{N}{\bar{p}}\leq\frac{\tau_{2}}{2}$ if $s=0$ and $|y|\geq4$, and by substituting \eqref{loc-est6.7} into \eqref{est-6.6''} if $0<s<p$, we get
$$ u(y) \leq C |y|^{-\frac{N-p}{p}-\frac{\tau_2}{2}}, \,\,\,\,\,\,\,\, \forall \ |y|\geq \frac{16}{\rho};$$
by substituting \eqref{loc-est6.6} into \eqref{est-6.6'} if $s=0$ and $|y|\leq4$, and by substituting \eqref{loc-est6.6} into \eqref{est-6.6''} if $0<s<p$, we get
$$ u(y) \leq C |y|^{-\frac{N-p}{p}+\tau_1}, \,\,\,\,\,\,\,\, \forall \ |y|\leq \frac{\rho}{16}.$$
This concludes our proof of Theorem \ref{th2.1++}.
\end{proof}

\subsection{The proof of Theorem \ref{th2.1}}\label{sc3.3}

Now, we are ready to prove the sharp asymptotic estimates for any positive weak solution $u$ to the generalized equation \eqref{gen-eq} and $|\nabla u|$ in Theorem \ref{th2.1}.

\begin{proof}[The proof of Theorem \ref{th2.1}]

The proof of the sharp asymptotic estimates in Theorem \ref{th2.1} will be divided into the following three steps.

\medskip

\noindent{\bf Step 1.} We will show that there exist some constants $c_0, C_0, R_0, R_1 > 0$ such that, for any positive $D^{1,p}(\R^{N})$-weak solution $u$ to \eqref{gen-eq} there hold
\begin{equation}\label{06191}
   c_0 |x|^{-\gamma_1} \leq u(x) \leq C_0 |x|^{-\gamma_1} \qquad \mbox{in}\,\,\,  |x| < R_0,
\end{equation}
and
\begin{equation}\label{06192}
 c_0 |x|^{-\gamma_2} \leq u(x) \leq C_0 |x|^{-\gamma_2} \qquad \mbox{in}\,\,\,  |x| > R_1,
\end{equation}
where $0<R_0<1<R_1$ are constants depending on $N,\mu,p,s,V$ and the solution $u$.

\medskip

We will show \eqref{06191} and \eqref{06192} by applying the following four key Lemmas \ref{X-estimate1}-\ref{X-estimate4}.
\begin{lem}[Theorem 1.3 in \cite{XCL}]\label{X-estimate1}
Let $\Omega$ be a bounded domain containing the origin and function $f(x)\in L^\frac{N}{p}(\Omega)$ satisfying $f(x)\leq C_{f,1}|x|^{-\beta_1}$ in $\Omega$ for some positive constants $C_{f,1}$ and $\beta_1<p$. If $u\in D^{1,p}(\Omega)$ is a weak sub-solution to
$$-\Delta_p u - \mu\frac{1}{|x|^p}u^{p-1}=f(x) u^{p-1}  \,\,\,\,\,\,\, \mbox{in}\,\, \Omega$$
with $1<p<N$ and $0\leq\mu< \bar{\mu}$, then there exists a positive constant $C=C(N,p,\mu,\beta_1,C_{f,1})>0$ such that
$$ u(x)\leq C M |x|^{-\gamma_1} \,\,\,\,\,\,\,\,\mbox{in}\,\,\, B_{R_0}(0)\subset \Omega,$$
where $M = \sup\limits_{\partial B_{R_0}(0)} u^+ $ and $R_0 > 0$ is a constant depending on $N, p, \mu, C_{f,1}, \beta_1$.
\end{lem}

\begin{lem}[Theorem 1.4 in \cite{XCL}]\label{X-estimate2}
Let $\Omega$ be a bounded domain containing the origin and $0\leq f(x)\in L^\frac{N}{p}(\Omega)$. If $u\in D^{1,p}(\Omega)$ is a nonnegative super-solution to
$$-\Delta_p u -\mu\frac{1}{|x|^p}u^{p-1} =f(x) u^{p-1}  \,\,\,\,\,\,\, \mbox{in}\,\, \Omega$$
with $1<p<N$ and $0\leq\mu< \bar{\mu}$, then
$$ u(x)\geq C m |x|^{-\gamma_1},\,\,\,\,\,\,\ \forall \,\ x\in\Omega$$
where $m=\inf\limits_{\Omega} u$ and $C=\inf\limits_{\partial\Omega} |x|^{\gamma_1}$.
\end{lem}

\begin{lem}[Theorem 1.5 in \cite{XCL}]\label{X-estimate3}
Let $\Omega$ be an exterior domain in $\R^N$ such that $\Omega^c=\R^N\setminus \Omega$ is bounded and
function $f(x) \in L^\frac{N}{p}(\Omega)$
satisfying $f(x) \leq C_{f,2} |x|^{-\beta_2}$ in $\Omega$ for some positive constants $C_{f,2}$ and $\beta_2>p$. If $u\in D^{1,p}(\Omega)$ is a nonnegative sub-solution to
$$
-\Delta_p u -\mu\frac{1}{|x|^p}u^{p-1} =f(x) u^{p-1}  \,\,\,\,\,\,\, \mbox{in}\,\, \Omega
$$
with $1<p<N$ and $0\leq\mu< \bar{\mu}$, then there exists a positive constant $C=C(N,p,\mu,\beta_2,C_{f,2})>0$ such that
$$ u(x)\leq C M |x|^{-\gamma_2}\,\,\,\,\,\,\,\,\mbox{in}\,\,\, \R^N\setminus B_{R_1}(0),$$
where $M = \sup\limits_{\partial B_{R_1}(0)} u^+ $ and $R_1 > 1$ is a constant depending on $N, p, \mu, C_{f,2}, \beta_2$.
\end{lem}

\begin{lem}[Theorem 1.6 in \cite{XCL}]\label{X-estimate4}
Let $\Omega$ be an exterior domain in $\R^N$ such that $\Omega^c=\R^N\setminus \Omega$ is bounded and
nonnegative function $f(x)\in L^\frac{N}{p}(\Omega)$. If $u\in D^{1,p}(\Omega)$ is a nonnegative weak super-solution to
$$-\Delta_p u -\mu\frac{1}{|x|^p}u^{p-1} =f(x) u^{p-1}  \,\,\,\,\,\,\, \mbox{in}\,\, \Omega$$
with $1<p<N$ and $0\leq\mu< \bar{\mu}$, then
$$ u(x)\geq C m |x|^{-\gamma_2},\,\,\,\,\,\,\ \forall \,\ x\in\Omega$$
where $m=\inf\limits_{\partial\Omega} u$ and $C=\inf\limits_{\partial\Omega} |x|^{\gamma_2}$.
\end{lem}

By the better preliminary estimates in Theorem \ref{th2.1++} and assumptions on $V(x)$ in Theorem \ref{th2.1}, we have $V(x)|x|^{-s} \leq C_{V,1} |x|^{-\beta_1}$ for some $C_{V,1}>0$, $\beta_1<p$ and all $|x|$ small, and $V(x)|x|^{-s} \leq C_{V,2} |x|^{-\beta_2}$ for some $C_{V,2}>0$, $\beta_2>p$ and all $|x|$ large. Combining these with $V(x)\in L^\frac{N}{p-s}(\R^N)$ and using H\"{o}lder's inequality, we can deduce that $V(x)|x|^{-s}\in L^\frac{N}{p}(\R^N)$. Thus we get the desired sharp upper bound in \eqref{06191} from Lemma \ref{X-estimate1} and the sharp lower bound in \eqref{06191} from Lemma \ref{X-estimate2}, respectively. In the same way, from Lemmas \ref{X-estimate3} and \ref{X-estimate4}, we obtain \eqref{06192}. This proves the sharp asymptotic estimates \eqref{eq0806} and \eqref{eq0806++} for $u$ in Theorem \ref{th2.1}.

\medskip

Next, in Steps 2 and 3, we will prove the sharp asymptotic estimates \eqref{eq0806+++} and \eqref{eq0806+} for $|\nabla u|$ in Theorem \ref{th2.1} by using the arguments in \cite{DLL,EMSV,BS16,VJ16}.

\medskip

\noindent{\bf Step 2.}
We are going to show that, for some constant $C_0>0$,
\begin{equation}\label{06191+}
|\nabla u(x)| \leq C_0 |x|^{-\gamma_1-1} \qquad \mbox{for}\,\,\, |x| < R_0 < 1,
\end{equation}
and
\begin{equation}\label{06192+}
|\nabla u(x)| \leq C_0 |x|^{-\gamma_2-1} \qquad \mbox{for}\,\,\, |x| > R_1>1.
\end{equation}

\smallskip

First, we prove \eqref{06192+}. For any $R>0$ and $y\in \R^N$, we define
$$ u_R (y): =R^{\gamma_2}u(Ry).$$
From the generalized equation \eqref{gen-eq}, we deduce that
\begin{align*}
-\Delta_p u_R(y) = - R^{(\gamma_2+1)(p-1)+1} \Delta_p u(x)
                 =\mu \frac{1}{|y|^p} u^{p-1}_R(y) + R^p V(Ry)|Ry|^{-s}u^{p-1}_R(y) \,\,\,\,\,\,\,\mbox{in}\,\,\R^N,
\end{align*}
where $x=Ry$. From the fact that $V(x) \leq C_{V,2} |x|^{-\beta_2+s}$ with $\beta_2>p$ when $|x|\geq R_{1}$ with the same $R_{1}>1$ large enough as in \eqref{06192}, it follows that
\begin{equation}\label{eq2902}
\begin{aligned}
-\Delta_p u_R(y)
&=\mu \frac{1}{|y|^p} u^{p-1}_R(y) + R^p V(Ry)|Ry|^{-s} u^{p-1}_R(y) \\ & \leq \mu \frac{1}{|y|^p} u^{p-1}_R(y)+\frac{C_{V,2}}{R^{\beta_2-p} |y|^{\beta_2}} u^{p-1}_R(y) \, \,\,\,\,\,\,\,\mbox{for}\,\,R>R_1\,\,\,\mbox{and}\,\,|y|>1,
\end{aligned}
\end{equation}
which implies that
\begin{align}\label{eq2903}
-\Delta_p u_R(y)
\leq (C_{V,2}+\mu) u^{p-1}_R(y) \,\,\,\,\,\,\,\mbox{for}\,\,R>R_1\,\,\,\mbox{and}\,\, y\in\R^N\setminus \overline{B_1(0)}.
\end{align}

\smallskip

On the other hand, according to \eqref{eq0806++}, we have, for any $R>R_1$,
\begin{align}\label{eq2905}
u_R (y) =R^{\gamma_2} u(Ry)\leq C_0 R^{\gamma_2} |Ry|^{-\gamma_2}= C_0 |y|^{-\gamma_2} \leq C_0 \,\,\,\,\quad \mbox{in}\,\,\R^N\setminus \overline{B_1(0)}.
\end{align}
By \eqref{eq2903}, \eqref{eq2905} and the regularity estimates in DiBenedetto \cite{ED83} and Tolksdorf \cite{PT}, we have
\begin{align}\label{eq2906}
\|\nabla u_R\|_{L^\infty (B_{4}(0) \setminus B_{2}(0))}\leq C_1, \,\,\quad\, \forall \,\, R>R_1
\end{align}
for some constant $C_1>0$ depending on $C_{0}$. Finally, for any $x\in\R^N\setminus B_{3R}(0)$, by applying \eqref{eq2906} with $R=\frac{|x|}{3}$ and $y\in B_{4}(0) \setminus B_{2}(0)$, we obtain
\begin{align}\label{eq2907}
|\nabla_x u(x)|= R^{-\gamma_2-1} |\nabla_y u_R(y) | \leq 3^{\gamma_{2}+1}C_1 |x|^{-\gamma_2-1}\,\,\quad\,\,\,\mbox{for}\,\,|x|>3 R_1.
\end{align}
By replacing $R_{1}$ with $3R_{1}$ in \eqref{06192} and still denoting $3R_{1}$ by $R_{1}$, we have arrived at \eqref{06192+}.

\medskip

Letting $u_R (y):=R^{\gamma_1}u(Ry)$, by similar way in deriving \eqref{eq2903} and \eqref{eq2905}, we can derive
\begin{equation*}\label{eq2902+}
\begin{aligned}
-\Delta_p u_R(y) \leq (C_{V,1}+\mu) u^{p-1}_R(y) \,\quad\,\mbox{in}\,\, B_6(0)\setminus B_1(0), \,\,\,\,\,\,\,\forall\,R<R_0<1,
\end{aligned}
\end{equation*}
and
\begin{align*}
u_R (y) =R^{\gamma_1} u(Ry)\leq C_0 R^{\gamma_1} |Ry|^{-\gamma_1}= C_0 |y|^{-\gamma_1} \leq C_0, \, \,\quad\,\,\,\forall\,\,6R<R_0\,\,\,\,\,\mbox{and}\,\,\,\,\, 6>|y|>1.
\end{align*}
Then, similar to \eqref{eq2906} and \eqref{eq2907}, by the regularity estimates in DiBenedetto \cite{ED83} and Tolksdorf \cite{PT}, we can deduce that \eqref{06191+} holds.

\medskip

\noindent{\bf Step 3.}
We will show that, for some constant $c_0>0$,
\begin{equation}\label{a1}
  |\nabla u(x)| \geq \frac{c_0}{|x|^{\gamma_1+1}} \quad \,\,\mbox{for}\,\,|x|<R_0,
\end{equation}
and
\begin{equation}\label{a2}
   |\nabla u(x)| \geq \frac{c_0}{|x|^{\gamma_2+1}} \quad \,\,\mbox{for}\,\,|x|>R_1.
\end{equation}

We will only prove \eqref{a2} by contradiction argument. The lower bound in \eqref{a1} can be proved in entirely similar way, so we omit the details. Suppose that \eqref{a2} does not hold, we assume that there exist sequences of radii $R_n$ and points $x_n\in\R^N$ with $R_n\to\infty$ as $n\to\infty$ and $|x_n|=R_n$, such that
\begin{align}\label{eq2908}
|\nabla u(x_n)|\leq \frac{\theta_n}{R_n^{\gamma_2+1}}
\end{align}
with $\theta_n\to 0$ as $n\to\infty$. For $0<a<A$ fixed, we set
$$ u_{R_n} (x): =R_n^{\gamma_2}   u(R_n  x), \qquad \forall \,\, x\in\overline{B_A \setminus B_a}.$$
It follows from \eqref{06192} and \eqref{06192+} that, for any $n$ large enough such that $|R_n x|\geq R_{n}a>R_1$, we have
\begin{equation}\label{eq2911+}
\begin{aligned}
&\frac{c_0}{|x|^{\gamma_2}}   \leq    u_{R_n} (x)=R_n^{\gamma_2} u(R_n  x) \leq \frac{C_0}{|x|^{\gamma_2}},\\
&|\nabla u_{R_n} (x)|=R_n^{\gamma_2+1}|\nabla u(R_n  x)| \leq \frac{C_0}{|x|^{\gamma_2+1}},
\end{aligned}
\end{equation}
which imply that
\begin{equation}\label{eq2911}
\begin{aligned}
\frac{c_0}{A^{\gamma_2}}   \leq    u_{R_n} (x) \leq \frac{C_0}{a^{\gamma_2}}\,\quad\,\,\,\,\,\,&\mbox{in}\,\,\overline{B_A \setminus B_a},\\
|\nabla u_{R_n} (x)| \leq \frac{C_0}{a^{\gamma_2+1}}\,\quad\,\,\,\,\,\,&\mbox{in}\,\,\overline{B_A \setminus B_a}.
\end{aligned}
\end{equation}
\smallskip

On the other hand, similar to \eqref{eq2902}, we can obtain
\begin{align}\label{eq2913}
0\leq -\Delta_p u_{R_n} (x) - \mu \frac{1}{|x|^p} u^{p-1}_{R_n}(x)\leq\frac{C_{V,2}}{{R_n}^{\beta_2-p} |x|^{\beta_2}} u^{p-1}_{R_n}(x) \leq C u^{p-1}_{R_n}(x) \,\,\,\,\,\,\,\mbox{in}\,\,\overline{B_A \setminus B_a}
\end{align}
for any $n$ large enough such that $|R_n x|\geq R_{n}a>R_1$ and $a^{\beta_2} {R_n}^{\beta_2-p}\geq1$. Therefore, from \eqref{eq2911} and \eqref{eq2913}, by the regularity results in \cite{PT}, we get that, there exists some $0<\alpha<1$ such that
$$ \sup\limits_{n}\|u_{R_n}\|_{C^{1,\alpha} (K)}<+\infty$$
for any compact set $K\subset B_A \setminus B_a$. Without loss of generalities, we may redefine $a$ and $A$ and assume that the above $C^{1,\alpha}$ estimates uniformly hold (w.r.t. $n$) in $\overline{B_A \setminus B_a}$. Hence, up to subsequences, we have
\begin{align}\label{eq2914}
u_{R_n} (x) \to  u_{a,A}(x),\,\quad \mbox{as}\,\,n\to\infty \,\,\,\,\,\quad\,\, \mbox{in}\,\,C^{1,\alpha^\prime}(B_A \setminus B_a)
\end{align}
for $0<\alpha^\prime<\alpha$. From \eqref{eq2913}, we deduce that
\begin{align}\label{eq2915}
-\Delta_p u_{a,A}(x)-\mu \frac{1}{|x|^p} u^{p-1}_{a,A}(x)=0\,\,\,\,\,\,\,\mbox{in}\,\,B_A \setminus \overline{B_a}.
\end{align}
Now, for every $j\in \mathbb{N}^+$, letting $a_j=\frac{1}{j}$ and $A_j=j$, we constructs $u_{a_j,A_j}$ by reasoning as above. Then, as $j\to\infty$, a diagonal argument implies that there exists a limiting profile $u_\infty$ such that
$$ u_\infty \equiv u_{a_j,A_j} \,\,\,\,\,\,\,\mbox{in}\,\, B_{A_j} \setminus B_{a_j}, \quad \forall \,\,j\in \mathbb{N}^+.$$
In particular,
\begin{align}\label{eq2916}
-\Delta_p u_\infty (x)-\mu \frac{1}{|x|^p} u^{p-1}_\infty(x)=0\,\,\,\,\,\,\,\,\mbox{in}\,\,\R^N \setminus \{ 0 \}.
\end{align}
From \eqref{eq2911+}, it follows that
\begin{equation}\label{eq2911++}
\frac{c_0}{|x|^{\gamma_2}}   \leq   u_\infty (x) \leq \frac{C_0}{|x|^{\gamma_2}} \quad\,\mbox{and} \quad\,
|\nabla u_\infty (x)| \leq \frac{C_0}{|x|^{\gamma_2+1}},\,\,\mbox{for any}\,\,x\in \R^N \setminus \{ 0 \}.
\end{equation}
Thus it follows from \cite[Theorems 1.3 and 4.1]{EMSV}, \eqref{eq2916} and \eqref{eq2911++} that
$$ u_\infty (x)= \frac{\hat{C}}{|x|^{\gamma_2}},$$
where $\hat{C}:=\limsup\limits_{|x|\to0^+}|x|^{\gamma_2}u_\infty (x)$. Let $\{x_n\}$ be a sequence the same as in \eqref{eq2908} and set $y_n=\frac{x_n}{R_n}$. Then, by \eqref{eq2908}, it follows that
$$|\nabla u_{R_n}(y_n)|=R_n^{\gamma_2+1} |\nabla u(x_n)|  \leq \theta_n \to 0,\quad\,\,\,\,\mbox{as}\,\,n\to\infty.$$
Since $|y_n|=1$, up to subsequences, we have $y_n\to\overline{y} \in \partial B_1$. Consequently, by the uniform convergence of the gradient indicated by \eqref{eq2914}, one has
$$ \nabla u_\infty(\overline{y})=0,$$
which is absurd since the fundamental solution $u_{\infty}$ has no critical points. This establishes the sharp asymptotic estimates for $|\nabla u|$ and hence concludes our proof of Theorem \ref{th2.1}.
\end{proof}

\subsection{Proofs of the sharp asymptotic estimates in Theorem \ref{th2}}\label{sc3.4}

\begin{proof}[Proof of the sharp asymptotic estimates in Theorem \ref{th2}]
The regularity and sharp asymptotic estimates in Theorem \ref{th2} are immediate consequences of Theorems \ref{th2.1.1} and \ref{th2.1}. We only need to verify that $V_1(x)=|x|^{-2p}\ast u^{p}$ satisfies all the assumptions on $V$ and $u$ in Theorems \ref{th2.1.1}, \ref{th2.1++} and \ref{th2.1}.

\smallskip

To this end, we will first prove the following key estimates on $V_{1}$.
\begin{lem}\label{V-1-norm}
Let $u\in D^{1,p}(\R^N)$ be a nonnegative solution to \eqref{eq1.1} with $1<p<N/2$ and $0\leq\mu<\bar{\mu}$. Then $V_{1}\in L^{\frac{N}{p}}(\mathbb{R}^{N})$. Let $\rho>0$ be small enough such that $ \|V_{1}\|_{L^\frac{N}{p} (B_{\rho}(0))}\leq \epsilon_1$ and $\|V_{1}\|_{L^\frac{N}{p} (\R^N\setminus B_{1/{\rho}}(0))} \leq \epsilon_1$, where $\epsilon_1>0$ is given by Lemma \ref{lm:l-b-2}. Then, for any $l\in(0,p)$,
\begin{equation}\label{est-bdd-1-loc}
V_1(x)=\left(|x|^{-2p}\ast u^{p}\right) \in L^{\frac{N}{p-l}}_{loc}(\mathbb{R}^{N}\setminus \{0\}),
\end{equation}
and
\begin{equation}\label{est-bdd-1}
\|V_1\|_{L^{\frac{N}{p-l}} (B_{2R}(y))} \leq \tilde{C} R^{-l}, \,\,\,\,\,\quad \forall \,\, |y|\leq \frac{\rho}{16} \quad \mbox{or} \quad |y|\geq \frac{16}{\rho},
\end{equation}
where $R=|y|/4$ and the positive constant $\tilde{C}$ is independent of $y$.
\end{lem}
\begin{proof}
By Hardy-Littlewood-Sobolev inequality in Theorem \ref{HLSI}, one has
\begin{equation*}
\begin{aligned}
     \|V_1(x)\|_{L^{\frac{N}{p}}(\R^N)}
     = \left\| \int_{\R^N} \frac{u^p(y)}{|x-y|^{2p}} \mathrm{d}y \right\|_{L^{\frac{N}{p}}(\R^N)}
     \leq \bar{C} \|u^p\|_{L^{\frac{N}{N-p}}(\R^N)}= \bar{C} \|u\|^p_{L^{p^\star}(\R^N)}.
\end{aligned}
\end{equation*}

\smallskip

Now we prove \eqref{est-bdd-1}. First, we note that, for any $l\in(0,p)$, there exists $q\in (\frac{N}{p},\infty)$ such that $q=\frac{N}{p-l}$. For any $|y|\leq \frac{\rho}{16}$ or $|y|\geq\frac{16}{\rho}$, let $R:=|y|/4$. Hence, we get
\begin{equation}\label{060102}
\begin{aligned}
     &\quad \|V_1(x)\|_{L^{q}(B_{2R}(y))} \leq \|V_1(x)\|_{L^{\frac{N}{p-l}}(B_{6R}(0)\setminus B_{2R}(0))} \\
     &= \left\| \left\{\int_{B_{R/8}(0)} + \int_{B_{8R}(0)\setminus B_{R/8}(0)} +\int_{B^c_{8R}(0)} \right\} \frac{u^p(z)}{|x-z|^{2p}} \mathrm{d}z \right\|_{L^{\frac{N}{p-l}}(B_{6R}(0)\setminus B_{2R}(0))} \\
     &\leq \left\| \int_{B_{\frac{R}{8}}(0)} \frac{u^p(z)}{|x-z|^{2p}} \mathrm{d}z \right\|_{L^{\frac{N}{p-l}}(B_{6R}\setminus B_{2R})}
          + \left\| \int_{B^c_{8R}(0)} \frac{u^p(z)}{|x-z|^{2p}} \mathrm{d}z \right\|_{L^{\frac{N}{p-l}}(B_{6R}(0)\setminus B_{2R}(0))} \\
          &+ \left\| \int_{B_{8R}(0)\setminus B_{\frac{R}{8}}(0)} \frac{u^p(z)}{|x-z|^{2p}} \mathrm{d}z \right\|_{L^{\frac{N}{p-l}}(B_{6R}(0)\setminus B_{2R}(0))}
     =:I_1 + I_2 +I_3.
\end{aligned}
\end{equation}
Since $|x-z|\geq 2R-R/8 >R$ for $x\in B_{6R}(0)\setminus B_{2R}(0)$ and $z\in B_{R/8}(0)$, by the Minkowski inequality and H\"{o}lder's inequality, we have
\begin{equation}\label{060103}
\begin{aligned}
     I_1 &=\left\| \int_{B_{R/8}(0)} \frac{u^p(z)}{|x-z|^{2p}} \mathrm{d}z \right\|_{L^{\frac{N}{p-l}}(B_{6R}(0)\setminus B_{2R}(0))} \\
         &\leq  \int_{B_{R/8}(0)} \left( \int_{B_{6R}(0)\setminus B_{2R}(0)} \left( \frac{u^{p}(z)}{|x-z|^{2p}} \right)^{\frac{N}{p-l}} \mathrm{d}x \right)^{\frac{p-l}{N}} \mathrm{d}z \\
         &\leq \frac{C}{R^{2p}} |B_{6R}(0)\setminus B_{2R}(0)|^{\frac{p-l}{N}}  \int_{B_{R/8}(0)} u^p(z) \mathrm{d}z\\
         &\leq C R^{-2p+N\cdot \frac{p-l}{N}} |B_{R/8}(0)|^{\frac{p}{N}} \left( \int_{B_{R/8}(0)} u^{p^\star}(z) \mathrm{d}z \right)^{\frac{p}{p^\star}} \leq C R^{-l} \|u\|^p_{L^{p^\star}(\R^N)},
\end{aligned}
\end{equation}
where $C>0$ is independent of $y$. Due to $|x-z|>|z|-6R\geq|z|-\frac{3}{4}|z|=\frac{|z|}{4}$ for any $x\in B_{6R}(0)\setminus B_{2R}(0)$ and $z\in B^c_{8R}(0)$, we have
\begin{equation}\label{060104}
\begin{aligned}
     I_2= & \left\| \int_{B^c_{8R}(0)} \frac{u^p(z)}{|x-z|^{2p}} \mathrm{d}z \right\|_{L^{\frac{N}{p-l}}(B_{6R}(0)\setminus B_{2R}(0))}\\
         &\leq  \int_{B^c_{8R}(0)} \left( \int_{B_{6R}(0)\setminus B_{2R}(0)} \left( \frac{u^{p}(z)}{|x-z|^{2p}} \right)^{\frac{N}{p-l}} \mathrm{d}x \right)^{\frac{p-l}{N}} \mathrm{d}z \\
         &\leq C|B_{6R}(0)\setminus B_{2R}(0)|^{\frac{p-l}{N}}  \int_{B^c_{8R}(0)} \frac{u^p(z)}{|z|^{2p}}\mathrm{d}z\\
         &\leq C R^{p-l} \left( \int_{B^c_{8R}(0)} \frac{\mathrm{d}z}{|z|^{2N}} \right)^{\frac{p}{N}} \left( \int_{B^c_{8R}(0)} u^{p^\star}(z) \mathrm{d}z \right)^{\frac{p}{p^\star}} \leq C R^{-l}  \|u\|^p_{L^{p^\star}(\R^N)},
\end{aligned}
\end{equation}
where $C>0$ is independent of $y$.

\smallskip

Thanks to Hardy-Littlewood-Sobolev inequality in Theorem \ref{HLSI} and Lemma \ref{lm:l-b-2} with $\bar{p}=Np/(N-p-l)$ and $s=0$, recalling that $R=|y|/4$, we deduce that, for any $|y|\leq \rho/16$ or $|y|\geq 16/\rho$,
\begin{equation}\label{060105}
\begin{aligned}
     I_3 &= \left\| \int_{B_{8R}(0)\setminus B_{R/8}(0)} \frac{u^p(z)}{|x-z|^{2p}} \mathrm{d}z \right\|_{L^{\frac{N}{p-l}}(B_{6R}(0)\setminus B_{2R}(0))} \\
          &\leq \left\| \int_{B_{8R}(0)\setminus B_{R/8}(0)} \frac{u^p(z)}{|x-z|^{2p}} \mathrm{d}z \right\|_{L^{\frac{N}{p-l}}(\R^N)}\\
          &\leq C \left\| u^p \right\|_{L^{\frac{N}{N-p-l}}(B_{8R}(0)\setminus B_{R/8}(0))} = C \left\| u \right\|^p_{L^{\frac{Np}{N-p-l}}(B_{8R}(0)\setminus B_{R/8}(0))} \\
          &\mathop{\leq}_{\text{Lemma \,3.3}} C \left( R^{N \cdot \frac{N-p-l}{Np} - \frac{N-p}{p}} \right)^p \|u\|^p_{L^{p^\star}(B_{16R}(0)\setminus B_{R/16}(0))} \leq C R^{-l} \|u\|^p_{L^{p^\star}(\R^N)},
\end{aligned}
\end{equation}
where $C>0$ is independent of $y$. By substituting \eqref{060103}, \eqref{060104} and \eqref{060105} into \eqref{060102}, we derive \eqref{est-bdd-1}.

\smallskip

Next, we show \eqref{est-bdd-1-loc}. Since we have proved $V_{1}(x) \in L^{\frac{N}{p}}(\R^N)$, it follows from Lemma \ref{lm:l-b-r} that $u\in L^{r}_{loc}(\mathbb{R}^{N}\setminus\{0\})$ for any $0<r<+\infty$. Then, for any $y\in\mathbb{R}^{N}\setminus\{0\}$ and $R:=\frac{|y|}{4}$, from \eqref{060102}, \eqref{060103}, \eqref{060104} and the estimate $I_2\leq C \left\| u \right\|^p_{L^{\frac{Np}{N-p-l}}(B_{8R}(0)\setminus B_{R/8}(0))}$ indicated by \eqref{060105}, we get
\begin{equation*}
  \|V_1(x)\|_{L^{\frac{N}{p-l}}(B_{2R}(y))}\leq C R^{-l} \|u\|^p_{L^{p^\star}(\R^N)}+C \left\| u \right\|^p_{L^{\frac{Np}{N-p-l}}(B_{8R}(0)\setminus B_{R/8}(0))}<+\infty,
\end{equation*}
which combined with the Heine-Borel finite covering theorem yields that $V_{1}\in L^{\frac{N}{p-l}}_{loc}(\mathbb{R}^{N}\setminus \{0\})$, i.e., \eqref{est-bdd-1-loc} holds. This completes the proof of Lemma \ref{V-1-norm}.
\end{proof}

Now we continue the proof of Theorem \ref{th2}. It follows from Lemma \ref{V-1-norm} that $V_1(x)\in L^{\frac{N}{p}}(\R^N)$. Now, let $\rho>0$ be small enough such that $ \|V_1(x)\|_{L^\frac{N}{p} (B_{\rho}(0))}\leq \min\{\epsilon_1,\epsilon_2\}$ and $\|V_1(x)\|_{L^\frac{N}{p} (\R^N\setminus B_{1/{\rho}}(0))} \leq \min\{\epsilon_1,\epsilon_2\}$, where $\epsilon_1>0$ and $\epsilon_2>0$ are given by Lemmas \ref{lm:l-b-2} and \ref{lm:l-b-3} respectively.
\smallskip

From Lemma \ref{V-1-norm}, we have $V_1(x)\in L^{q}_{loc}(\mathbb{R}^{N}\setminus \{0\})$ for some $\frac{N}{p}<q<+\infty$, and $V_1(x)$ satisfies condition \eqref{est-bdd-1+} in Theorem \ref{th2.1++}. Then, it follows from Theorems \ref{th2.1.1} and \ref{th2.1++} that, $u\in C^{1,\alpha}(\R^N\setminus \{0\})\cap L^{\infty}_{loc}(\mathbb{R}^{N}\setminus \{0\})$ for some $0<\alpha<\min\{1,\frac{1}{p-1}\}$, and there exists a positive constant $C=C(N,p,\mu,\rho,u,\tilde{C})>0$ such that
\[u(y) \leq C |y|^{-\frac{N-p}{p}+\tau_1}, \,\,\,\,\quad\,\, \forall \,\, |y|\leq \frac{\rho}{16},\]
and
\[u(y) \leq C |y|^{-\frac{N-p}{p}-\frac{\tau_2}{2}}, \,\,\,\quad\,\,\, \forall \,\, |y|\geq \frac{16}{\rho},\]
where $\tau_1, \tau_2 > 0$ are given by Lemma \ref{lm:l-b-3} and $\tilde{C}$ be given in \eqref{est-bdd-1} in Lemma \ref{V-1-norm}.
\smallskip

On the other hand, we can observe that $N > N-p-p\tau > N-2p$ and $N > N-p+p\tau > N-2p$, where $\tau=\frac{1}{2}\min\{1,\tau_1,\tau_2\}$. Then, using (ii) of Lemma \ref{lm.6} with $\nu=2p$, $g(x)=u^p(x)$, $\beta_1=N-p-p\tau$ and $\beta_2=N-p+p\tau$, we have, for $|x|<\rho/16$,
\[V_1(x)=|x|^{-2p}\ast u^{p}\leq C_{V,1}|x|^{-\bar{\beta}_1} \quad \text{with} \quad \bar{\beta}_1:=\beta_1+2p-N=p-\tau p<p,\]
and for $|x|>16/\rho$,
\[V_1(x)=|x|^{-2p}\ast u^{p}\leq C_{V,2}|x|^{-\bar{\beta}_2} \quad \text{with} \quad \bar{\beta}_2:=\beta_2+2p-N=p+\tau p>p.\]

Thus Theorems \ref{th2.1.1} and \ref{th2.1} can be applied to equation \eqref{eq1.1}. So the positive solution $u$ to \eqref{eq1.1} satisfies the regularity result in Theorem \ref{th2.1.1} and the sharp asymptotic estimates \eqref{eq0806}-\eqref{eq0806+} in Theorem \ref{th2.1}. This finishes our proof of the regularity and sharp asymptotic estimates in Theorem \ref{th2}.
\end{proof}

\section{Radial symmetry and monotonicity of solutions to equation \eqref{eq1.1}}\label{sc4}

Since the sharp asymptotic estimates \eqref{eq0806}-\eqref{eq0806+} in Theorem \ref{th2.1} hold for equation \eqref{eq1.1}, we can apply the method of moving planes to derive the radial symmetry and monotonicity of weak solutions to \eqref{eq1.1}, and hence complete the proof of Theorem \ref{th2}. Before carrying out the moving planes technique, we need some standard notations.

\smallskip

Let $\nu$ be any direction in $\R^N$, i.e. $\nu\in \R^N$ and $|\nu|=1$. Due the rotation invariance of the problems, we choose $\nu=e_1$ and let
$$ \Sigma_\lambda: =\{ x\in\R^N \mid x_1 < \lambda \}, $$
$$ T_\lambda: =\{ x\in\R^N \mid x_1 = \lambda \},$$
$$ u_\lambda (x) := u(x_\lambda), \quad\,\,\,\,\forall \,\, x\in\R^N,$$
and define the reflection $R_\lambda$ w.r.t. the hyperplane $T_\lambda$ by
$$ x_\lambda: =R_\lambda(x)=(2\lambda-x_1,x^\prime)\in\R \times \R^{N-1}.$$
We denote the critical set of $u$ by $Z_u:=\{ x\in\R^N \mid |\nabla u (x)|=0 \}$ and define
\begin{equation}\label{Z}
  Z^{u}_\lambda:=\{ x\in\R^N \mid |\nabla u (x)|=|\nabla u_\lambda (x)|=0 \} \subseteq  Z_u.
\end{equation}

\medskip

First, for any $q,\sigma>0$, we will show the following basic property of the convolution $\bar{V}(x)=|x|^{-\sigma}\ast u^q(x)$.
\begin{lem}\label{lm.1}
For the convolution $\bar{V}(x)=|x|^{-\sigma}\ast u^q(x)$ with $q,\sigma>0$, one has
\begin{align}\label{eq0805}
\bar{V}(x)-\bar{V}(x_\lambda)&=\int_{\R^N} \frac{u^q(y)}{|x-y|^\sigma}\mathrm{d}y-\int_{\R^N} \frac{u^q(y)}{|x_\lambda-y|^\sigma}\mathrm{d}y \\
&=\int_{\Sigma_\lambda}\left(\frac{1}{|x-y|^\sigma}- \frac{1}{|x_\lambda-y|^\sigma}\right) \left[u^q(y)-(u_\lambda)^q(y)\right]\mathrm{d}y. \nonumber
\end{align}
\end{lem}
\begin{proof}
Since $|x-y_\lambda|=|x_\lambda-y|$ and $|x-y|=|x_\lambda-y_\lambda|$, we have
\begin{align*}
\bar{V}(x)=\int_{\R^N} \frac{u^q(y)}{|x-y|^\sigma}\mathrm{d}y
&=\int_{\Sigma_\lambda} \frac{u^q(y)}{|x-y|^\sigma}\mathrm{d}y+\int_{\R^N \setminus \Sigma_\lambda} \frac{u^q(y)}{|x-y|^\sigma}\mathrm{d}y \nonumber\\
&=\int_{\Sigma_\lambda} \frac{u^q(y)}{|x-y|^\sigma}\mathrm{d}y+\int_{\Sigma_\lambda} \frac{\left(u_\lambda\right)^q(y)}{|x_\lambda-y|^\sigma}\mathrm{d}y,
\end{align*}
and
\begin{align*}
\bar{V}(x_\lambda)=\int_{\R^N} \frac{u^q(y)}{|x_\lambda-y|^\sigma}\mathrm{d}y
&=\int_{\Sigma_\lambda} \frac{u^q(y)}{|x_\lambda-y|^\sigma}\mathrm{d}y+\int_{\R^N \setminus \Sigma_\lambda} \frac{u^q(y)}{|x_\lambda-y|^\sigma}\mathrm{d}y \nonumber\\
&=\int_{\Sigma_\lambda} \frac{u^q(y)}{|x_\lambda-y|^\sigma}\mathrm{d}y+\int_{\Sigma_\lambda} \frac{\left(u_\lambda\right)^q(y)}{|x-y|^\sigma}\mathrm{d}y.
\end{align*}
Subtracting the above two formulae yields \eqref{eq0805}. This finishes the proof of Lemma \ref{lm.1}.
\end{proof}

\eqref{eq1.1} can be written as follows:
\begin{align}\label{eq-m-p-1}
\left\{ \begin{array}{ll} \displaystyle
-\Delta_p u - \frac{\mu}{|x|^p} u^{p-1}=V_1(x) u^{p-1}  \quad\,\,\,\,&\mbox{in}\,\, \R^N, \\ \\
u \in D^{1,p}(\R^N),\quad\,\,\,\,\,\,u\geq0 \qquad\,\,\,\,&  \mbox{in}\,\, \R^N,
\end{array}
\right.\hspace{1cm}
\end{align}
where $0\leq\mu< \bar{\mu}$, $1<p<\frac{N}{2}$ and $0\leq V_1(x)= \left(|x|^{-2p}\ast u^{p}\right)(x) \in L^\frac{N}{p}(\R^N)$.
\smallskip

It is easy to see that $u_\lambda$ solves
\begin{equation}\label{eq-m-p-2}
-\Delta_p u_\lambda - \frac{\mu}{|x_\lambda|^p} u_\lambda^{p-1}=V_1(x_\lambda) u_\lambda^{p-1}  \quad\,\,\,\, \mbox{in}\,\, \R^N.
\end{equation}
From Lemma \ref{lm.1} with $\sigma=2p$ and $q=p$, we obtain that
\begin{align}\label{eq-v1-1}
V_1(x)-V_1 (x_\lambda)
        =\int_{\Sigma_\lambda}\left(\frac{1}{|x-y|^{2p}}- \frac{1}{|x_\lambda-y|^{2p}}\right) \left[u^p(y)-(u_\lambda)^p(y)\right]\mathrm{d}y.
\end{align}

\medskip

\begin{proof}[Completion of the proof of Theorem \ref{th2}]
We will prove the radial symmetry and monotonicity result by the method of moving planes, and treating the singular elliptic case $1<p\leq2$ and the degenerate elliptic case $p\geq2$ separately. We divide our proof into three steps. In what follows, let
$$\Lambda^- := \{ \lambda<0 \mid  u\leq u_\mu \,\,\, \mbox{in} \,\ \Sigma_\mu, \,\  \forall \mu \leq \lambda  \}, \,\,\,\,\,
\Lambda^+ := \{ \lambda>0 \mid  u\leq u_\mu \,\,\, \mbox{in} \,\ \R^N\setminus\overline{\Sigma_\mu}, \,\ \forall \mu \geq \lambda  \}. $$

\medskip

{\bf Step 1.} We will prove that $\Lambda^- \neq \emptyset $ and $\Lambda^+ \neq \emptyset$.

We only prove $\Lambda^- \neq \emptyset$, namely, there exists $\lambda<0$ sufficiently negative such that $u\leq u_\mu$ in $\Sigma_\mu$ for every $\mu\leq\lambda$. Indeed, the proof for $\Lambda^+ \neq \emptyset $ is similar to that for $\Lambda^- \neq \emptyset$, so we omit it here.

\smallskip

Throughout the proof, we denote by $R_0$ and $R_1$ the radii given by \eqref{eq0806}--\eqref{eq0806+} in Theorem \ref{th2.1} and the asymptotic estimates in Theorem \ref{th2}. By \eqref{eq0806} and \eqref{eq0806++}, we can see that, for any $\lambda<0$ with $|\lambda|> R_1$, there exists $\tilde{R}_0:=\tilde{R}_0(R_{1})>0$ such that $\tilde{R}_0\leq R_0$, $B_{\tilde{R}_0}(0_{\lambda}) \subset \Sigma_{\lambda}$ and
\begin{equation}\label{050333}
\sup_{B_{\tilde{R}_0}(0_\lambda)} u(x) < \inf_{B_{\tilde{R}_0}(0_\lambda)} u_\lambda (x), \qquad \forall \,\, \lambda<-R_{1},
\end{equation}
which implies that $u<u_\lambda$ in $B_{\tilde{R}_0}(0_\lambda) \subset \Sigma_\lambda$ (with $\tilde{R}_0$ independent of $\lambda$) for every $\lambda<-R_{1}$. Consequently, in what follows, we define $\Sigma^\prime_\lambda:=\Sigma_\lambda \setminus B_{\tilde{R}_0}(0_\lambda)$ and $\widetilde{B}_r:=B_r(0)\cap \Sigma^\prime_\lambda$ for $r>0$.
\smallskip

For $\lambda<-R_{1}<0$ and $\tau > \max\{2,p\}$, we construct
$$ \varphi_{1,\lambda}:=\eta^\tau u^{1-p} (u^p-u_\lambda^p)_+ \chi_{\Sigma_\lambda}  \,\,\quad \mbox{and} \,\,\quad  \varphi_{2,\lambda}:=\eta^\tau u_\lambda^{1-p} (u^p-u_\lambda^p)_+ \chi_{\Sigma_\lambda},$$
where $\eta \in C_0^\infty (B_{2R}(0))$ is a cut-off function such that $0\leq\eta\leq 1$, $\eta=1$ in $B_R(0)$ and $|\nabla \eta|\leq 2/R$, and $f_{+}:=\frac{|f|+f}{2}\geq0$. Due to \eqref{050333}, one has $supp(\varphi_{i,\lambda})\subset \widetilde{B}_{2R}$ for $i=1,2$. By taking $\varphi_{1,\lambda}$ as the test function in \eqref{eq-m-p-1} and $\varphi_{2,\lambda}$ in \eqref{eq-m-p-2} respectively, and subtracting, we obtain
\begin{equation}\label{050401}
\begin{aligned}
&\quad\int_{\widetilde{B}_{2R}}\left( |\nabla u|^{p-2}\nabla u \cdot \nabla\varphi_{1,\lambda} - |\nabla u_\lambda|^{p-2}\nabla u_\lambda \cdot \nabla\varphi_{2,\lambda} \right) \mathrm{d}x \\
&= \mu \int_{\widetilde{B}_{2R}}\left(\frac{\eta^\tau}{|x|^p}-\frac{\eta^\tau}{|x_\lambda|^p}\right)(u^p-u^p_\lambda)_+ \mathrm{d}x+\int_{\widetilde{B}_{2R}}\left( V_1(x) - V_1(x_\lambda) \right)\eta^\tau (u^p-u^p_\lambda)_+ \mathrm{d}x.
\end{aligned}
\end{equation}
It follows from $|x|\geq |x_\lambda|$ in $\Sigma_\lambda$ that
$$ \mu \int_{\widetilde{B}_{2R}}\left(\frac{1}{|x|^p}-\frac{1}{|x_\lambda|^p} \right)\eta^\tau (u^p-u^p_\lambda)_+ \mathrm{d}x \leq 0,$$
which together with \eqref{050401} lead to
\begin{equation}\label{050402}
\begin{aligned}
I_1:=&\int_{\widetilde{B}_{2R}} \eta^\tau \left[ |\nabla u|^{p-2}\nabla u \cdot \nabla\left(u^{1-p} (u^p-u_\lambda^p)_+ \right) - |\nabla u_\lambda|^{p-2}\nabla u_\lambda \cdot \nabla\left(u_\lambda^{1-p} (u^p-u_\lambda^p)_+\right) \right] \mathrm{d}x \\
     &\leq -\tau \int_{\widetilde{B}_{2R}} \eta^{\tau-1} u^{1-p} (u^p-u_\lambda^p)_+  |\nabla u|^{p-2}\nabla u \cdot \nabla\eta \mathrm{d}x  \\
     & + \tau \int_{\widetilde{B}_{2R}} \eta^{\tau-1} u_\lambda^{1-p} (u^p-u_\lambda^p)_+ |\nabla u_\lambda|^{p-2}\nabla u_\lambda \cdot \nabla\eta \mathrm{d}x \\
     & + \int_{\widetilde{B}_{2R}}\left( V_1(x) - V_1(x_\lambda) \right)\eta^\tau (u^p-u^p_\lambda)_+ \mathrm{d}x =: I_2 + I_3 + I_4.
\end{aligned}
\end{equation}

First, we estimate $I_1$. Lemma \ref{lm.w-g-i} yields that
\begin{equation*}
\begin{aligned}
I_1:=&\int_{\widetilde{B}_{2R}\cap\{u \geq u_\lambda\}} \eta^\tau \left[ |\nabla u|^{p-2}\nabla u \cdot \nabla\left(u^{1-p} (u^p-u_\lambda^p) \right) - |\nabla u_\lambda|^{p-2}\nabla u_\lambda \cdot \nabla\left(u_\lambda^{1-p} (u^p-u_\lambda^p)\right) \right] \mathrm{d}x \\
     &\geq C_p \int_{\widetilde{B}_{2R}\cap\{u \geq u_\lambda\}} \eta^\tau u_\lambda^p (|\nabla \log u|+|\nabla \log u_\lambda|)^{p-2} |\nabla \log u- \nabla \log u_\lambda|^2 \mathrm{d}x,
\end{aligned}
\end{equation*}
which implies that
\begin{equation}\label{05040301}
I_1\geq C_p \int_{\widetilde{B}_{2R}\cap\{u \geq u_\lambda\}} \eta^\tau u_\lambda^p u^{2-p} (|\nabla u|+|\nabla u_\lambda|)^{p-2} |\nabla \log u- \nabla \log u_\lambda|^2 \mathrm{d}x, \,\,\,\,\, \text{if} \,\, p\geq 2,
\end{equation}
and
\begin{equation}\label{05040301+}
\begin{aligned}
I_1 \geq  C_p \int_{\widetilde{B}_{2R}\cap\{u \geq u_\lambda\}} \eta^\tau u_\lambda^2 (|\nabla u|+|\nabla u_\lambda|)^{p-2} |\nabla \log u- \nabla \log u_\lambda|^2 \mathrm{d}x, \,\,\,\,\, \text{if} \,\, 1<p\leq 2.
\end{aligned}
\end{equation}

\smallskip

Now, we are to show the following useful Lemma.
\begin{lem}\label{a}
Assume $\kappa>0$. Then, for any $\lambda\leq-\kappa$,
\begin{equation}\label{050403}
\frac{u_\lambda}{u}\geq \tilde{c} \,\,\quad\,\, \mbox{in} \,\,\Sigma_\lambda,
\end{equation}
where $\tilde{c}=\tilde{c}\left(c_0, C_0, R_{1}, \inf\limits_{B_{R_1}(0)} u(x), \sup\limits_{|x|>\kappa}u(x)\right)>0$ and $c_0, C_0$ are given by \eqref{eq0806} and \eqref{eq0806++}.
\end{lem}
\begin{proof}
Indeed, if $x\in\Sigma_\lambda\setminus B_{R_1}(0_\lambda)$, then \eqref{eq0806} and \eqref{eq0806++} imply that (recall that $|x|\geq |x_\lambda|$ in $\Sigma_\lambda$)
$$ \frac{u_\lambda(x)}{u(x)} \geq  \frac{c_0}{|x_\lambda|^{\gamma_2}} \cdot \frac{|x|^{\gamma_2}}{C_0} \geq c_0/C_0.$$
Otherwise, if $x\in\Sigma_\lambda \cap B_{R_1}(0_\lambda)$, then
$$ \frac{u_\lambda(x)}{u(x)} \geq \frac{\inf\limits_{B_{R_1}(0_\lambda)} u_\lambda(x)}{\sup\limits_{|x|\geq\kappa}u(x)}\geq \frac{\inf\limits_{B_{R_1}(0)} u(x)}{\sup\limits_{|x|>\kappa}u(x)}\geq c_1>0,$$
where $\sup\limits_{|x|>\kappa}u(x)\leq \frac{C_{0}}{R_{1}^{\gamma_{2}}}+\sup\limits_{\kappa<|x|\leq R_{1}}u(x)$. Thus we have reached \eqref{050403} and hence finished the proof of Lemma \ref{a}.
\end{proof}

It follows from \eqref{05040301} and \eqref{050403} that
\begin{equation}\label{050404}
I_1 \geq C_p \tilde{c}^p \int_{\widetilde{B}_{2R}\cap\{u \geq u_\lambda\}} \eta^\tau u^2 (|\nabla u|+|\nabla u_\lambda|)^{p-2} |\nabla \log u- \nabla \log u_\lambda|^2 \mathrm{d}x, \,\,\,\,\, \text{if} \,\, p\geq2.
\end{equation}

\medskip

Next, we estimate $I_2$. By using \eqref{eq0806++} and \eqref{eq0806+}, we deduce that
\begin{equation}\label{050405}
\begin{aligned}
     I_2 &= -\tau \int_{\widetilde{B}_{2R}} \eta^{\tau-1} u^{1-p} (u^p-u_\lambda^p)_+  |\nabla u|^{p-2}\nabla u \cdot \nabla\eta \mathrm{d}x  \\
     &\leq \tau \int_{\widetilde{B}_{2R}\cap\{u \geq u_\lambda\}} \eta^{\tau-1} u^{1-p} (u^p-u_\lambda^p)  |\nabla u|^{p-1} |\nabla\eta| \mathrm{d}x \\
     &\leq \frac{2\tau}{R} \int_{\widetilde{B}_{2R}\setminus \widetilde{B}_R} u |\nabla u|^{p-1} \mathrm{d}x
      \leq  \frac{2\tau C^p_0}{R} \int_{\widetilde{B}_{2R}\setminus \widetilde{B}_R} \frac{1}{|x|^{(\gamma_2+1)(p-1)+\gamma_2}} \mathrm{d}x
      \leq \frac{C}{R^{\beta_1}},
\end{aligned}
\end{equation}
where $\beta_1:=p\gamma_2+p-N>0$ since $\gamma_2>\frac{N-p}{p}$.

\medskip

Now, we estimate $I_3$. By \eqref{eq0806++} and \eqref{050403}, we have
\begin{equation}\label{050406}
\begin{aligned}
     I_3 & = \tau \int_{\widetilde{B}_{2R}} \eta^{\tau-1} u_\lambda^{1-p} (u^p-u_\lambda^p)_+ |\nabla u_\lambda|^{p-2}\nabla u_\lambda \cdot \nabla\eta \mathrm{d}x \\
     &\leq \tau \int_{\widetilde{B}_{2R}} \eta^{\tau-1} u_\lambda^{1-p} (u^p-u_\lambda^p)_+ |\nabla u_\lambda|^{p-1} |\nabla\eta| \mathrm{d}x \\
     & \leq \frac{2\tau}{R} \int_{(\widetilde{B}_{2R}\setminus \widetilde{B}_R)\cap\{u \geq u_\lambda\}} u_\lambda \left(\left(\frac{u}{u_\lambda}\right)^p-1\right) |\nabla u_\lambda|^{p-1} \mathrm{d}x \\
     & 
     \leq \frac{2\tau}{\tilde{c}^p R} \int_{\widetilde{B}_{2R}\setminus \widetilde{B}_R} u |\nabla u_\lambda|^{p-1} \mathrm{d}x
         \leq \frac{2\tau}{\tilde{c}^p R} \left( \int_{\R^N} |\nabla u_\lambda|^p \mathrm{d}x \right)^\frac{p-1}{p} \left( \int_{\widetilde{B}_{2R}\setminus \widetilde{B}_R} u^p \mathrm{d}x \right)^\frac{1}{p}\\
     &\leq \frac{2\tau\cdot C_0}{\tilde{c}^p R} \|\nabla u\|^{p-1}_{L^p (\R^N)} \left( \int_{\widetilde{B}_{2R}\setminus \widetilde{B}_R} \frac{1}{|x|^{p\gamma_2}} \mathrm{d}x  \right)^\frac{1}{p}
     \leq  \frac{C}{R^\frac{\beta_1}{p}}.
\end{aligned}
\end{equation}

\medskip

Finally, we give the estimate of $I_4$, which is one of the key ingredients in our proof. Noting that $|x-y|\leq |x_\lambda-y|$ for every $x,y\in\Sigma_\lambda$ and $u(x)<u_\lambda(x)$ in $B_{\tilde{R}_0}(0_\lambda)$, it follows from Lemma \ref{lm.1}, the definition of $\Sigma_\lambda^\prime$ and the fact that $u^p$ is convex in $u$ that
\begin{equation}\label{050408}
\begin{aligned}
 V_1(x) - V_1(x_\lambda)
       &=\int_{\Sigma_\lambda}\left(\frac{1}{|x-y|^{2p}}- \frac{1}{|x_\lambda-y|^{2p}}\right) \left[u^p(y)-(u_\lambda)^p(y)\right] \mathrm{d}y \\
       &\leq \int_{\Sigma_\lambda}\frac{1}{|x-y|^{2p}}  \left[u^p(y)-(u_\lambda)^p(y)\right]_+ \mathrm{d}y \\
       &\leq p\int_{\Sigma_\lambda^\prime}\frac{u^{p-1}(y) \cdot (u-u_\lambda)_+ (y)}{|x-y|^{2p}} \mathrm{d}y, \qquad \ \forall \,\, x \in\Sigma_\lambda,
\end{aligned}
\end{equation}
which together with the definition of $I_4$ lead to
\begin{equation}\label{050409}
\begin{aligned}
   I_4 &=\int_{\widetilde{B}_{2R}}\left( V_1(x) - V_1(x_\lambda) \right)\eta^\tau (u^p-u^p_\lambda)_+ \mathrm{d}x \\
       &\leq p^{2}\int_{\Sigma_\lambda^\prime}\int_{\Sigma_\lambda^\prime}\frac{u^{p-1}(y) \cdot (u-u_\lambda)_+ (y) \cdot u^{p-1}(x) \cdot (u-u_\lambda)_+ (x)}{|x-y|^{2p}}  \mathrm{d}x\mathrm{d}y \\
       &\leq C \left\| u^{p-1}\cdot (u-u_\lambda)_+ \right\|^2_{L^\frac{N}{N-p}(\Sigma_\lambda^\prime)},
\end{aligned}
\end{equation}
where the last inequality is derived by using the Hardy-Littlewood-Sobolev inequality in Theorem \ref{HLSI}. Now, let us consider $f(t)=\log \left( a+t(b-a) \right)$, where $t\in(0,1)$ and $b\geq a>0$. The Newton-Leibniz formula gives
$$ \log b = \log a + (b-a)\int_0^1 \frac{1}{a+t(b-a)} dt.$$
Since for $t\in(0,1)$, $a\leq a+t(b-a) \leq b$, therefore we get
\begin{equation}\label{log-ineq}
b-a \leq b(\log b -\log a).
\end{equation}
Using \eqref{log-ineq}, H\"{o}lder's inequality and \eqref{eq0806++} in \eqref{050409}, we get
\begin{equation}\label{050410}
\begin{aligned}
   I_4
       &\leq C \left\| u^{p-1}\cdot (u-u_\lambda) \right\|^2_{L^\frac{N}{N-p}\left(\Sigma_\lambda^\prime \cap \{u\geq u_\lambda\}\right)}\\
       &\leq C \left\| u^{p-1} \cdot u \cdot (\log u- \log u_\lambda) \right\|^2_{L^\frac{N}{N-p}\left(\Sigma_\lambda^\prime \cap \{u\geq u_\lambda\}\right)} \\
       &= C \left\| u^\frac{p(N-2p)}{2(N-p)} \cdot u^\frac{Np}{2(N-p)} \cdot (\log u- \log u_\lambda) \right\|^2_{L^\frac{N}{N-p}\left(\Sigma_\lambda^\prime \cap \{u\geq u_\lambda\}\right)} \\
       &\leq C \left( \int_{ \Sigma_\lambda^\prime \cap \{u\geq u_\lambda\} } u^{p^\star} \mathrm{d}x \right)^\frac{N-2p}{N} \int_{ \Sigma_\lambda^\prime \cap \{u\geq u_\lambda\} } u^{p^\star} ( \log u - \log u_\lambda )^2 \mathrm{d}x \\
       &\leq C \left\| u \right\|_{L^{p^\star} (\R^N)}^\frac{p(N-2p)}{N-p} \cdot \int_{ \Sigma_\lambda^\prime \cap \{u\geq u_\lambda\} } \frac{1}{|x|^{\gamma_2 p^\star}} ( \log u - \log u_\lambda )^2 \mathrm{d}x \\
       &= C \left\| u \right\|_{L^{p^\star} (\R^N)}^\frac{p(N-2p)}{N-p} \cdot \int_{ \Sigma_\lambda^\prime \cap \{u\geq u_\lambda\} } \frac{1}{|x|^{\beta_2 - 2\theta+2}} ( \log u - \log u_\lambda )^2 \mathrm{d}x \\
       &\leq C \left\| u \right\|_{L^{p^\star} (\R^N)}^\frac{p(N-2p)}{N-p} \cdot \frac{1}{|\lambda|^{\beta_2}} \cdot \int_{ \Sigma_\lambda^\prime \cap \{u\geq u_\lambda\} }  \frac{1}{|x|^{-2\theta+2}} ( \log u - \log u_\lambda )^2 \mathrm{d}x,
\end{aligned}
\end{equation}
where $2\theta:=-\left[ (\gamma_2+1)(p-2) + 2 \gamma_2 \right]$ and $\beta_2:=\gamma_2 (p^\star-p)-p$. Here we have used the facts that $\gamma_{2}p^\star=\beta_2 - 2\theta+2$, and
\[\beta_2 =\gamma_2 (p^\star-p)-p > \frac{N-p}{p} \cdot \frac{p^2}{N-p}-p=0,\]
since $\gamma_2>(N-p)/p$.

\medskip

In order to estimate the right hand side of \eqref{050410}, we can apply the Caffarelli-Kohn-Nirenberg inequality in $(ii)$ of Theorem \ref{HDSI} with $\tau=q=2$, $a=1$, $\gamma=\theta-1$, $\alpha=\theta$ and
$$ \frac{1}{2}+\frac{\theta-1}{N} =\frac{N+2\theta-2}{2N} =\frac{N-p\gamma_2-p}{2N} <0,$$
and hence, we get (recall that $supp(( \log u - \log u_\lambda )_+\chi_{\Sigma_\lambda^\prime}) \subset\R^N\setminus\{0\}$)
\begin{equation}\label{050411}
\begin{aligned}
   I_4
       &\leq C \left\| u \right\|_{L^{p^\star} (\R^N)}^\frac{p(N-2p)}{N-p} \cdot \frac{1}{|\lambda|^{\beta_2}} \cdot \int_{ \Sigma_\lambda^\prime \cap \{u\geq u_\lambda\} }  |x|^{2\theta-2} ( \log u - \log u_\lambda )^2 \mathrm{d}x\\
       &\leq C \left\| u \right\|_{L^{p^\star} (\R^N)}^\frac{p(N-2p)}{N-p} \cdot \frac{1}{|\lambda|^{\beta_2}} \cdot \int_{ \Sigma_\lambda^\prime \cap \{u\geq u_\lambda\} }  |x|^{2\theta} |\nabla ( \log u - \log u_\lambda )|^2 \mathrm{d}x.
\end{aligned}
\end{equation}

\medskip

Next, in order to continue estimating the right hand side of \eqref{050411}, we will discuss the case $p\geq2$ and the case $1<p<2$ separately.

\smallskip

For $p\geq2$, from \eqref{eq0806++} and \eqref{eq0806+}, we get
\begin{equation}\label{050412}
\begin{aligned}
   I_4
       &\leq C \left\| u \right\|_{L^{p^\star} (\R^N)}^\frac{p(N-2p)}{N-p} \cdot \frac{1}{|\lambda|^{\beta_2}} \cdot \int_{ \Sigma_\lambda^\prime \cap \{u\geq u_\lambda\} }  |x|^{2\theta} |\nabla ( \log u - \log u_\lambda )|^2 \mathrm{d}x \\
       &\leq C \left\| u \right\|_{L^{p^\star} (\R^N)}^\frac{p(N-2p)}{N-p} \cdot \frac{1}{|\lambda|^{\beta_2}} \cdot \int_{ \Sigma_\lambda^\prime \cap \{u\geq u_\lambda\} }  \frac{1}{|x|^{(\gamma_2+1)(p-2) + 2 \gamma_2}} |\nabla ( \log u - \log u_\lambda )|^2 \mathrm{d}x \\
       &\leq C c_0^{-2} \left\| u \right\|_{L^{p^\star} (\R^N)}^\frac{p(N-2p)}{N-p} \cdot \frac{1}{|\lambda|^{\beta_2}}  \int_{ \Sigma_\lambda^\prime \cap \{u\geq u_\lambda\} }  \frac{1}{|x|^{(\gamma_2+1)(p-2)}} u^2 |\nabla ( \log u - \log u_\lambda )|^2 \mathrm{d}x \\
       &\leq C c_0^{-p} \left\| u \right\|_{L^{p^\star} (\R^N)}^\frac{p(N-2p)}{N-p} \cdot \frac{1}{|\lambda|^{\beta_2}}  \int_{ \Sigma_\lambda^\prime \cap \{u\geq u_\lambda\} }  u^2 |\nabla u|^{p-2} |\nabla ( \log u - \log u_\lambda )|^2 \mathrm{d}x \\
       &\leq \frac{C}{|\lambda|^{\beta_2}}  \int_{ \Sigma_\lambda^\prime \cap \{u\geq u_\lambda\} }  u^2 ( |\nabla u|+ |\nabla u_\lambda|)^{p-2} |\nabla ( \log u - \log u_\lambda )|^2 \mathrm{d}x.
\end{aligned}
\end{equation}

\smallskip

As to the case $1<p<2$, noting that $2\theta<0$ and $|x|\geq |x_\lambda|$ for $x\in\Sigma_\lambda$, we obtain $|x|^{2\theta} \leq |x_\lambda|^{2\theta}$ in $\Sigma_\lambda$, which together with \eqref{050411} yield
\begin{equation}\label{0504122}
\begin{aligned}
   I_4
       &\leq C \left\| u \right\|_{L^{p^\star} (\R^N)}^\frac{p(N-2p)}{N-p} \cdot \frac{1}{|\lambda|^{\beta_2}} \cdot \int_{ \Sigma_\lambda^\prime \cap \{u\geq u_\lambda\} }  |x|^{2\theta} |\nabla ( \log u - \log u_\lambda )|^2 \mathrm{d}x \\
       &\leq C \left\| u \right\|_{L^{p^\star} (\R^N)}^\frac{p(N-2p)}{N-p} \cdot \frac{1}{|\lambda|^{\beta_2}} \cdot \int_{ \Sigma_\lambda^\prime \cap \{u\geq u_\lambda\} }  |x_\lambda|^{2\theta} |\nabla ( \log u - \log u_\lambda )|^2 \mathrm{d}x.
\end{aligned}
\end{equation}
Based on the sharp decay estimates \eqref{eq0806++} and \eqref{eq0806+} for $|x|>R_1$ and the definition of $\Sigma^\prime_\lambda=\Sigma_\lambda \setminus B_{\tilde{R}_0}(0_\lambda)$, we can take $A_{R_1,\tilde{R}_0}:= \overline{B_{R_1}(0_\lambda) \setminus B_{\tilde{R}_0}(0_\lambda)} $ and $E_\lambda:=\Sigma_\lambda \setminus B_{R_1}(0_\lambda)$ and derive
\begin{equation}\label{050413}
\begin{aligned}
     &\quad \int_{ \Sigma_\lambda^\prime \cap \{u\geq u_\lambda\} }  |x_\lambda|^{2\theta} |\nabla ( \log u - \log u_\lambda )|^2 \mathrm{d}x \\
     &\leq \left\{ \int_{ E_\lambda \cap \{u\geq u_\lambda\}} +  \int_{ A_{R_1,\tilde{R}_0} \cap \{u\geq u_\lambda\}} \right\}   |x_\lambda|^{2\theta} |\nabla ( \log u - \log u_\lambda )|^2 \mathrm{d}x.
\end{aligned}
\end{equation}
Since the sharp decay estimates \eqref{eq0806++} and \eqref{eq0806+} hold on $E_\lambda$ (which imply $ |\nabla u| + |\nabla u_\lambda| \leq C_0 (|x|^{-(\gamma_2+1)} + |x_\lambda|^{-(\gamma_2+1)})$ and $|u_\lambda(x)| > c_0 |x_\lambda|^{-\gamma_2}$ in $E_\lambda$), and $|x| \geq |x_\lambda|$ in $\Sigma_\lambda$, we have
\begin{equation}\label{050414}
\begin{aligned}
    & \quad \int_{ E_\lambda \cap \{u\geq u_\lambda\}} |x_\lambda|^{2\theta} |\nabla ( \log u - \log u_\lambda )|^2 \mathrm{d}x \\
    & = \int_{ E_\lambda \cap \{u\geq u_\lambda\}} |x_\lambda|^{(\gamma_2+1)(2-p)} |x_\lambda|^{-2 \gamma_2} |\nabla ( \log u - \log u_\lambda )|^2 \mathrm{d}x \\
    & \leq  (2C_0)^{p-2} c_0^{-2} \int_{ E_\lambda \cap \{u\geq u_\lambda\}} u_\lambda^2 \frac{|\nabla ( \log u - \log u_\lambda )|^2}{(|\nabla u| + |\nabla u_\lambda|)^{2-p}} \mathrm{d}x.
\end{aligned}
\end{equation}
In $A_{R_1,\tilde{R}_0}$, it holds that $\tilde{R}_0 \leq |x_\lambda| \leq R_1$ and $|x|\geq |\lambda|>R_{1}$, and since $u\in C_{loc}^{1,\alpha}(\R^N\setminus \{0\})$ for some $0<\alpha<\min\{1,\frac{1}{p-1}\}$, we get that $|\nabla u|+|\nabla u_\lambda|\leq C$ for some $C>0$ on $A_{R_1,\tilde{R}_0}$. Let
$L:=\inf\limits_{A_{R_1,\tilde{R}_0}} u$. Then, by using \eqref{050403} and the fact that $2-p>0$, we get (recall that $2\theta<0$)
\begin{equation}\label{050415}
\begin{aligned}
     & \quad \int_{ A_{R_1,\tilde{R}_0} \cap \{u\geq u_\lambda\}} |x_\lambda|^{2\theta} |\nabla ( \log u - \log u_\lambda )|^2 \mathrm{d}x \\
     & \leq \tilde{R}_0^{2\theta} \int_{ A_{R_1,\tilde{R}_0} \cap \{u\geq u_\lambda\}} \frac{u^2}{L^2} \frac{u_\lambda^2}{u_\lambda^2}  \frac{C}{(|\nabla u| + |\nabla u_\lambda|)^{2-p}} |\nabla ( \log u - \log u_\lambda )|^2 \mathrm{d}x \\
     & \leq \frac{C \tilde{R}_0^{2\theta}}{ \tilde{c}^2 L^2}  \int_{ A_{R_1,\tilde{R}_0} \cap \{u\geq u_\lambda\}} u_\lambda^2  \frac{|\nabla (\log u - \log u_\lambda)|^2}{(|\nabla u| + |\nabla u_\lambda|)^{2-p}}  \mathrm{d}x.
\end{aligned}
\end{equation}
Hence, for $1<p<2$, we have
\begin{equation}\label{050415}
\begin{aligned}
   I_4 &\leq \frac{C}{|\lambda|^{\beta_2}} \int_{ \Sigma_\lambda^\prime \cap \{u\geq u_\lambda\} }  u_\lambda^2  \frac{|\nabla (\log u - \log u_\lambda)|^2}{(|\nabla u| + |\nabla u_\lambda|)^{2-p}}\mathrm{d}x.
\end{aligned}
\end{equation}
This concludes our estimates on $I_{4}$.

\medskip

By substituting the estimates \eqref{05040301+}, \eqref{050404}, \eqref{050405}, \eqref{050406}, \eqref{050412} and \eqref{050415} in \eqref{050402}, we have
\begin{equation}\label{050416}
\begin{aligned}
  &C_p \tilde{c}^p \int_{\widetilde{B}_{2R}\cap\{u \geq u_\lambda\}} \eta^\tau u^2 (|\nabla u|+|\nabla u_\lambda|)^{p-2} |\nabla \log u- \nabla \log u_\lambda|^2 \mathrm{d}x \\
  &\leq \frac{C\tau}{R^{\beta_1}} + \frac{C}{R^\frac{\beta_1}{p}} + \frac{C}{|\lambda|^{\beta_2}}  \int_{ \Sigma_\lambda^\prime \cap \{u\geq u_\lambda\} }  u^2 ( |\nabla u|+ |\nabla u_\lambda|)^{p-2} |\nabla ( \log u - \log u_\lambda )|^2 \mathrm{d}x
  , \,\,\, &\text{if} \,\, p\geq2,
\end{aligned}
\end{equation}
and
\begin{equation}\label{050417}
\begin{aligned}
   & C_p \int_{\widetilde{B}_{2R}\cap\{u \geq u_\lambda\}} \eta^\tau u_\lambda^2 (|\nabla u|+|\nabla u_\lambda|)^{p-2} |\nabla \log u- \nabla \log u_\lambda|^2 \mathrm{d}x \\
   & \leq \frac{C\tau}{R^{\beta_1}} + \frac{C}{R^\frac{\beta_1}{p}} + \frac{C}{|\lambda|^{\beta_2}}\int_{ \Sigma_\lambda^\prime \cap \{u\geq u_\lambda\} }  u_\lambda^2  \frac{|\nabla (\log u - \log u_\lambda)|^2}{(|\nabla u| + |\nabla u_\lambda|)^{2-p}} \mathrm{d}x
   , \,\,\,\,\,\, & \text{if} \,\, 1<p<2,
\end{aligned}
\end{equation}
where $R>0$ can be arbitrarily large and the positive constants $C_p$, $\tilde{c}$ and $C$ are independent of $R$ and $\lambda$. Take $|\lambda|$ sufficiently large such that $\frac{C}{|\lambda|^{\beta_2}} \leq \frac{1}{2}\min\{C_p \tilde{c}^p, C_p\}$. Then, as $R$ goes to $+\infty$, there hold (recall that $\beta_1, \beta_2>0$)
\begin{equation}\label{050418}
\begin{aligned}
  &\quad C_p \tilde{c}^p \int_{ \Sigma_\lambda^\prime \cap \{u\geq u_\lambda\} }  u^2 ( |\nabla u|+ |\nabla u_\lambda|)^{(p-2)} |\nabla ( \log u - \log u_\lambda )|^2 \mathrm{d}x \\
  &= \lim_{R\to\infty} C_p \tilde{c}^p \int_{\widetilde{B}_{R}\cap\{u \geq u_\lambda\}}  u^2 (|\nabla u|+|\nabla u_\lambda|)^{p-2} |\nabla \log u- \nabla \log u_\lambda|^2 \mathrm{d}x \\
  & \leq \frac{C_p \tilde{c}^p}{2}  \int_{ \Sigma_\lambda^\prime \cap \{u\geq u_\lambda\} }  u^2 ( |\nabla u|+ |\nabla u_\lambda|)^{p-2} |\nabla ( \log u - \log u_\lambda )|^2 \mathrm{d}x
  , \,\,\, & \text{if} \,\, p\geq2,
\end{aligned}
\end{equation}
and
\begin{equation}\label{050419}
\begin{aligned}
   &\quad C_p \int_{ \Sigma_\lambda^\prime \cap \{u\geq u_\lambda\} }  u_\lambda^2  \frac{|\nabla (\log u - \log u_\lambda)|^2}{(|\nabla u| + |\nabla u_\lambda|)^{2-p}} \mathrm{d}x \\
   &= \lim_{R\to\infty} C_p \int_{\widetilde{B}_{R}\cap\{u \geq u_\lambda\}}  u_\lambda^2 (|\nabla u|+|\nabla u_\lambda|)^{p-2} |\nabla \log u- \nabla \log u_\lambda|^2 \mathrm{d}x \\
   &\leq \frac{C_p}{2} \int_{ \Sigma_\lambda^\prime \cap \{u\geq u_\lambda\} }  u_\lambda^2  \frac{|\nabla (\log u - \log u_\lambda)|^2}{(|\nabla u| + |\nabla u_\lambda|)^{2-p}} \mathrm{d}x
   , \,\,\,\,\,\, & \text{if} \,\, 1<p<2.
\end{aligned}
\end{equation}
Consequently, in both cases $1<p<2$ and $p\geq2$, it follows from \eqref{050418} and \eqref{050419} that $|\nabla (\log u - \log u_\lambda)|=0$ in $\Sigma_\lambda^\prime \cap \{u\geq u_\lambda\}$. Since $u=u_\lambda$ on $T_\lambda$, then one has $\log u - \log u_\lambda=0$ in $\Sigma_\lambda^\prime \cap \{u\geq u_\lambda\}$. Thus we get $u \leq u_\lambda$ in $\Sigma_\lambda$ for $\lambda<-R_1<0$ with sufficiently large $|\lambda|$, i.e., $\Lambda^- \neq \emptyset$.

\medskip

In almost the same way, we can also prove $\Lambda^+ \neq \emptyset$. From the preceding discussion, letting $\lambda_0^-:=\sup \Lambda^-$ and $\lambda_0^+:=\inf \Lambda^+$, we have already derived $-\infty<\lambda_0^- \leq 0$ and $0\leq\lambda_0^+<+\infty$.

\medskip

{\bf Step 2.} We will prove that $\lambda_0^+=\lambda_0^-=0$.

\smallskip

Assume by contradiction that $\lambda_0^-<0$. Then, arguing as in Step 1, we will obtain a contradiction by proving that $ u\leq u_\lambda $ in $\Sigma_\lambda$ for all $\lambda_0^- \leq \lambda \leq \lambda_0^- + \varepsilon_0<0$ for some $\varepsilon_0>0$ small.

\medskip

Indeed, by continuity, we obtain
$$ u \leq u_{\lambda_0^-} \,\,\, \mbox{in} \,\ \Sigma_{\lambda_0^-},$$
which together with \eqref{eq-v1-1} implies that
$$V_1(x)\leq V_1(x_{\lambda_0^-}), \,\,\,\,\   \forall \,\, x\in \Sigma_{\lambda_0^-}.$$
Since $u, u_{\lambda_0^-} \in C_{loc}^{1,\alpha} (\Sigma_{\lambda_0^-}\setminus\{0_{\lambda_0^-}\})$ for some $0< \alpha < \min\{1,\frac{1}{p-1}\}$ and $|x|\geq |x_{\lambda_0^-}|$ for all $x\in \Sigma_{\lambda_0^-}$, by the strong comparison principle in Lemmas \ref{th3.1} and \ref{th3.2} and the following fact
$$ -\Delta_p u = \mu \frac{1}{|x|^p}u^{p-1}+V_1(x)u^{p-1} \leq \mu \frac{1}{|x_{\lambda_0^-}|^p} u_{\lambda_0^-}^{p-1} + V_1(x_{\lambda_0^-}) u_{\lambda_0^-}^{p-1} = -\Delta_p u_{\lambda_0^-} \,\,\quad \mbox{in} \,\, \Sigma_{\lambda_0^-},$$
we deduce that, either $u \equiv u_{\lambda_0^-}$ or $u < u_{\lambda_0^-}$ in any connected component $\mathcal{C}$ of $\Sigma_{\lambda_0^-}\setminus (Z^{u}_{\lambda_0^-}\cup\{0_{\lambda_0^-}\})$, where $Z^{u}_{\lambda_0^-}=\{ x\in\R^N \mid |\nabla u|=|\nabla u_{\lambda_0^-}|=0 \}$. Moreover, by the sharp upper and lower bounds estimates for solution $u$ near $0$ in \eqref{eq0806}, and choosing $\tilde{R}_0>0$ smaller if necessary (such that at least $\tilde{R}_0<|\lambda^{-}_{0}|$), similar to \eqref{050333}, we can derive that
\begin{equation}\label{05041911}
u<u_{\lambda_0^-}\,\,\,\,\,\mbox{in}\,\,\,B_{\tilde{R}_0} (0_{\lambda_0^-}).
\end{equation}
Thus, we infer that, either $u \equiv u_{\lambda_0^-}$ or $u < u_{\lambda_0^-}$ in any $\mathcal{C}$, where $\mathcal{C}$ is a connected component of $\Sigma_{\lambda_0^-}\setminus Z^{u}_{\lambda_0^-}$.

\smallskip

The sharp lower bounds of $|\nabla u|$ in \eqref{eq0806+++} and \eqref{eq0806+} indicate that the critical set $Z_u=\{ x\in\R^N \mid |\nabla u|=0 \}$ of the solution $u$ is a closed set belonging to $B_{R_1}(0)\setminus B_{R_0}(0)$, i.e,
$$ Z^{u}_{\lambda_0^-} \subset Z_u \subset B_{R_1}(0)\setminus B_{R_0}(0).$$
Meanwhile, it follows from Corollary \ref{re2333} that $|Z_{u}|=0$ (see Remark \ref{re2334}), and from Lemma \ref{lm.7} that $\Omega\setminus Z_u$ is connected for any smooth bounded domain $\Omega\subset\mathbb{R}^{N}$ with connected boundary such that $B_{R_1}(0)\setminus B_{R_0}(0)\subseteq\Omega$. As a consequence, one has that $\mathbb{R}^{N}\setminus Z_u$ is connected. Since $Z_{u}$ is closed and $|Z_{u}|=0$, we can deduce further that $\mathbb{R}^{N}\setminus Z^{u}_{\lambda_{0}^-}$ is connected.

\smallskip

Since $u<u_{\lambda_0^-}$ in $B_{\tilde{R}_0} (0_{\lambda_0^-})$ by \eqref{05041911}, we infer that there exists a least one connected component $\mathcal{C}$ of $\Sigma_{\lambda_0^-}\setminus Z^{u}_{\lambda_0^-}$ such that $u<u_{\lambda_0^-}$ in $\mathcal{C}$. If $\Sigma_{\lambda_0^-}\setminus Z^{u}_{\lambda_0^-}$ has only one connected component, it follows that $u<u_{\lambda_0^-}$ in $\Sigma_{\lambda_0^-}\setminus Z^{u}_{\lambda_0^-}$. If $\Sigma_{\lambda_0^-}\setminus Z^{u}_{\lambda_0^-}$ has at least two different connected components $\mathcal{C}_{1}$ and $\mathcal{C}_{2}$ such that $u=u_{\lambda_0^-}$ in $\mathcal{C}_{1}$ and $u<u_{\lambda_0^-}$ in $\mathcal{C}_{2}$. Then, by symmetry, $\mathcal{C}_{1}\cup R_{\lambda_{0}^-}(\mathcal{C}_{1})\subsetneq \mathbb{R}^{N}\setminus Z^{u}_{\lambda_0^-}$ contains at least one connected component of $\mathbb{R}^{N}\setminus Z^{u}_{\lambda_{0}^-}$, which is absurd. Thus we have proved that
$$ u < u_{\lambda_0^-} \,\,\,\ \mbox{in}  \,\,\ \Sigma_{\lambda_0^-}\setminus Z^{u}_{\lambda_{0}^-},$$
which implies that
\begin{equation}\label{K}
  u < u_{\lambda_0^-} \,\,\,\ \mbox{in}  \,\,\ \Sigma_{\lambda_0^-}\setminus Z_u.
\end{equation}

\smallskip

Recalling that $Z_u\subset B_{R_1}(0)\setminus B_{R_0}(0)$ and $|Z_u|=0$, for any $\sigma>0$ small, we let $Z_u^\sigma$ be an open set containing $Z_u$ such that $|Z_u^\sigma|<\sigma$. Then, for $\delta, \sigma >0$, $\bar{R}>R_1$ and $\lambda\in(\lambda_0^-,0)$, we define
$$ E_{\lambda,\bar{R}} := \Sigma_\lambda \setminus B_{\bar{R}}(0),\,\,\,\
   S_\delta^\lambda := \left((\Sigma_\lambda\setminus\Sigma_{\lambda_0^- -\delta}) \cap B_{\bar{R}}(0) \right) \cup (Z_u^\sigma \cap \Sigma_{\lambda_0^- -\delta}) $$
and
$$ K_\delta:= \overline{\left( B_{\bar{R}}(0) \cap \Sigma_{\lambda_0^- -\delta} \right)} \cap (Z_u^\sigma)^c,$$
where $\delta\leq\bar{\delta}$ with $\bar{\delta}:=\bar{\delta}(\sigma)$ small enough such that $K_\delta\neq\emptyset$.
It is clear that
$$\Sigma_\lambda=E_{\lambda,\bar{R}} \cup S_\delta^\lambda \cup K_\delta, \,\,\,\quad \lambda \in (\lambda_0^-,0).$$
For $\delta\in(0,\bar{\delta})$, \eqref{K} implies that $u<u_{\lambda_0^-}$ in $K_\delta$. Since $K_\delta$ is compact, by the uniform continuity of $u$ and $u_\lambda$, there exists a small $\varepsilon_0 \in (0,|\lambda_0^-|/2)$ such that
\begin{equation}\label{K1}
  u < u_\lambda \,\,\,\,\ \mbox{in} \,\,\ K_\delta
\end{equation}
for any $\lambda\in [\lambda_0^-,\lambda_0^-+\varepsilon_0]$. Moreover, by choosing $\tilde{R}_0$ in \eqref{05041911} smaller if necessary (such that at least $\tilde{R}_0<|\lambda_0^-+\varepsilon_0|$), similar to \eqref{050333} and \eqref{05041911}, we can derive that, for any $\lambda\in [\lambda_0^-,\lambda_0^-+\varepsilon_0]$,
\begin{equation}\label{eq100227}
u < u_\lambda \,\,\,\,\ \mbox{in} \,\,\ B_{\tilde{R}_0}(0_\lambda).
\end{equation}

Let $\Sigma^\prime_\lambda:=\Sigma_\lambda \setminus B_{\tilde{R}_0}(0_\lambda)$ and $\widetilde{B}_r:=B_r(0)\cap \Sigma^\prime_\lambda$ for $r>0$.
For any $\lambda\in [\lambda_0^-,\lambda_0^-+\varepsilon_0]$, as in Step 1, we can choose the following functions
$$ \varphi_{1,\lambda}:=\eta^\tau u^{1-p} (u^p-u_\lambda^p)_+ \chi_{\Sigma_\lambda}  \,\,\ \mbox{and} \,\,\  \varphi_{2,\lambda}:=\eta^\tau u_\lambda^{1-p} (u^p-u_\lambda^p)_+ \chi_{\Sigma_\lambda},$$
where $\eta\in C^\infty_0(B_{2R}(0))$ is a cut-off function with $0\leq \eta\leq 1$, $\eta\equiv1$ on $B_R(0)$ and $|\nabla \eta|\leq 2/R$ for any $R>\bar{R}$. Furthermore, for any $\lambda\in [\lambda_0^-,\lambda_0^-+\varepsilon_0]$, since $\lambda\leq\lambda_0^-+\varepsilon_0<\frac{\lambda_0^-}{2}$, so the lower bound estimate \eqref{050403} in Lemma \ref{a} with $\kappa=-\frac{\lambda_0^-}{2}$ holds. In the same way as Step 1, taking $\varphi_{1,\lambda}$ and $\varphi_{2,\lambda}$ as the test function in \eqref{eq-m-p-1} in \eqref{eq-m-p-2} respectively, and then similar to \eqref{050402}, \eqref{05040301+}, \eqref{050404}, \eqref{050405}, \eqref{050406} and \eqref{050410}, we get that, for any $\lambda\in [\lambda_0^-,\lambda_0^-+\varepsilon_0]$,
\begin{equation}\label{050420}
\begin{aligned}
  &\quad C_p \tilde{c}^p \int_{\widetilde{B}_{2R}\cap\{u \geq u_\lambda\}} \eta^\tau u^2 (|\nabla u|+|\nabla u_\lambda|)^{p-2} |\nabla \log u- \nabla \log u_\lambda|^2 \mathrm{d}x \\
  &\leq \frac{C\tau}{R^{\beta_1}} + \frac{C}{R^\frac{\beta_1}{p}} + C \left\| u \right\|_{L^{p^\star} (\R^N)}^\frac{p(N-2p)}{N-p} \cdot \int_{ \Sigma_\lambda^\prime \cap \{u\geq u_\lambda\} } u^{p^\star} ( \log u - \log u_\lambda )^2 \mathrm{d}x, \,\,\quad & \text{if} \,\, p\geq2,
\end{aligned}
\end{equation}
and
\begin{equation}\label{050421}
\begin{aligned}
   &\quad C_p \int_{\widetilde{B}_{2R}\cap\{u \geq u_\lambda\}} \eta^\tau u_\lambda^2 (|\nabla u|+|\nabla u_\lambda|)^{p-2} |\nabla \log u- \nabla \log u_\lambda|^2 \mathrm{d}x \\
   &\leq \frac{C\tau}{R^{\beta_1}} + \frac{C}{R^\frac{\beta_1}{p}} + C \left\| u \right\|_{L^{p^\star} (\R^N)}^\frac{p(N-2p)}{N-p} \cdot \int_{ \Sigma_\lambda^\prime \cap \{u\geq u_\lambda\} } u^{p^\star} ( \log u - \log u_\lambda )^2 \mathrm{d}x ,\,\,\quad & \text{if} \,\, 1<p\leq2,
\end{aligned}
\end{equation}
where $R>0$ can be arbitrarily large and the positive constants $C_p$, $\tilde{c}$ and $C$ are independent of $R$ and $\lambda$.

\smallskip

Arguing in similar way as the estimates of $I_4$ in Step 1 and replacing $|\lambda|$ with $\bar{R}$ therein, we can derive that, if $p\geq2$,
\begin{equation}\label{050422'}
\int_{E_{\lambda,\bar{R}} \cap \{u\geq u_\lambda\}}u^{p^\star} ( \log u - \log u_\lambda )^2 \mathrm{d}x\leq \frac{C}{\bar{R}^{\beta_2}}  \int_{ E_{\lambda,\bar{R}} \cap \{u\geq u_\lambda\} }  u^2 ( |\nabla u|+ |\nabla u_\lambda|)^{p-2} |\nabla ( \log u - \log u_\lambda )|^2 \mathrm{d}x,
\end{equation}
and if $1<p\leq2$,
\begin{equation}\label{050423'}
\int_{E_{\lambda,\bar{R}} \cap \{u\geq u_\lambda\}} u^{p^\star} ( \log u - \log u_\lambda )^2 \mathrm{d}x\leq \frac{C}{\bar{R}^{\beta_2}} \cdot \int_{ E_{\lambda,\bar{R}} \cap \{u\geq u_\lambda\} }  u_\lambda^2  \frac{|\nabla (\log u - \log u_\lambda)|^2}{(|\nabla u| + |\nabla u_\lambda|)^{2-p}} \mathrm{d}x,
\end{equation}
where $\beta_2>0$ and $C>0$ is independent of $R$ and $\lambda$. Then, due to $\Sigma_\lambda \cap \{u\geq u_\lambda\} = \Sigma^\prime_\lambda \cap \{u\geq u_\lambda\}$, from \eqref{050422'}, \eqref{050423'}, the definitions of $E_{\lambda,\bar{R}}$, $S_\delta^\lambda$ and $K_\delta$,  we get
\begin{equation}\label{050422}
\begin{aligned}
  & \int_{ \Sigma_\lambda^\prime \cap \{u\geq u_\lambda\} } u^{p^\star} ( \log u - \log u_\lambda )^2 \mathrm{d}x\\
  &= \left\{ \int_{E_{\lambda,\bar{R}} \cap \{u\geq u_\lambda\}}+ \int_{ S_\delta^\lambda \cap \{u\geq u_\lambda\}} \right\} u^{p^\star} ( \log u - \log u_\lambda )^2 \mathrm{d}x\\
  &\leq \frac{C}{\bar{R}^{\beta_2}}  \int_{ E_{\lambda,\bar{R}} \cap \{u\geq u_\lambda\} }  u^2 ( |\nabla u|+ |\nabla u_\lambda|)^{p-2} |\nabla ( \log u - \log u_\lambda )|^2 \mathrm{d}x\\
  &+  C \int_{ S_\delta^\lambda \cap \{u\geq u_\lambda\}} u^{p^\star} ( \log u - \log u_\lambda )^2 \mathrm{d}x, \,\,\, & \text{if} \,\, p\geq2,
\end{aligned}
\end{equation}
and
\begin{equation}\label{050423}
\begin{aligned}
    &\int_{ \Sigma_\lambda^\prime \cap \{u\geq u_\lambda\} } \frac{1}{|x|^{\beta^\star-2\alpha+2}} ( \log u - \log u_\lambda )^2 \mathrm{d}x \\
    &=\left \{ \int_{E_{\lambda,\bar{R}} \cap \{u\geq u_\lambda\}}+ \int_{ S_\delta^\lambda \cap \{u\geq u_\lambda\}} \right\} u^{p^\star} ( \log u - \log u_\lambda )^2 \mathrm{d}x\\
    &\leq \frac{C}{\bar{R}^{\beta_2}} \cdot \int_{ E_{\lambda,\bar{R}} \cap \{u\geq u_\lambda\} }  u_\lambda^2  \frac{|\nabla (\log u - \log u_\lambda)|^2}{(|\nabla u| + |\nabla u_\lambda|)^{2-p}} \mathrm{d}x \\
    &+  C \int_{ S_\delta^\lambda \cap \{u\geq u_\lambda\}} u^{p^\star} ( \log u - \log u_\lambda )^2 \mathrm{d}x,\,\,\, & \text{if} \,\, 1<p\leq2
\end{aligned}
\end{equation}
for any $\lambda\in [\lambda_0^-,\lambda_0^-+\varepsilon_0]$ and $\delta\in(0,\bar{\delta})$, where $C>0$ is independent of $R$ and $\lambda$.

\medskip

Next, from \eqref{050422} and \eqref{050423}, we can see that it is necessary for us to estimate
\begin{equation}\label{a0}
  \int_{ S_\delta^\lambda \cap \{u\geq u_\lambda\}} u^{p^\star} ( \log u - \log u_\lambda )^2 \mathrm{d}x.
\end{equation}
Since $u$ is bounded in $S_{\bar{\delta}}^{\lambda_0^- + \varepsilon_0}$, we have
$$ \int_{ S_\delta^\lambda \cap \{u\geq u_\lambda\}} u^{p^\star} ( \log u - \log u_\lambda )^2 \mathrm{d}x \leq C \int_{ S_\delta^\lambda \cap \{u\geq u_\lambda\}} ( \log u - \log u_\lambda )^2 \mathrm{d}x, $$
where $C:=\sup\limits_{S_{\bar{\delta}}^{\lambda_0^- + \varepsilon_0}} u^{p^\star} $.

Next, we will discuss the two different cases $1<p<2$ and $p\geq2$ separately.

\medskip

If $p\geq2$, in order to estimate the integral in $S_\delta^\lambda \cap \{u\geq u_\lambda\}$, we will use the weighted Poincar\'{e} type inequality in Lemma \ref{lm.3}. For $x\in S_\delta^\lambda$, let $v(x):=(\log u-\log u_{\lambda})_+$. By \eqref{K1} and the definition of $v(x)$, we obtain that $v(x)=0$ on $\partial\Sigma_{\lambda} \cup (\Sigma_{\lambda_0^{-}-\delta}\cap B_{\bar{R}})$. Thus we can define $v(x+te_{1}):=0$ for all $x\in \partial\Sigma_{\lambda}\cap B_{\bar{R}}$ and $t>0$. For all $x\in S_\delta^\lambda$, we can write
$$ v(x)=-\int_{0}^{+\infty}\frac{\partial v}{\partial x_1}(x+te_{1}) \mathrm{d}t=-\int_{0}^{\lambda-x_{1}}\frac{\partial v}{\partial x_1}(x+te_{1}) \mathrm{d}t,$$
which in conjunction with integrating on $x_{2},\cdots,x_{N}$ and finite covering theorem imply that, for some $C>0$,
$$ |v(x)|\leq C\int_{S_\delta^\lambda} \frac{|\nabla v(y)|}{|x-y|^{N-1}}\mathrm{d}y, \quad\quad \,\,\,\,\forall \,\, x\in S_\delta^\lambda.$$
Thus $v$ satisfies the condition \eqref{eq10.25.4} with $\Omega=S_\delta^\lambda\subset\R^N\setminus\{0\}$. From Lemma \ref{lm.4}, Corollary \ref{re2333} and Remark \ref{rem0}, one has that $\rho=|\nabla u|^{p-2}$ satisfies the condition \eqref{eq10.25.2}. Thus, by the weighted Poincar\'{e} type inequality in Lemma \ref{lm.3}, we obtain
\begin{align*}
\int_{S_\delta^\lambda} [(\log u - \log u_{\lambda})_+]^2 \mathrm{d}x \leq C (S_\delta^\lambda) \int_{S_\delta^\lambda} (|\nabla u|+|\nabla u_{\lambda}|)^{p-2} |\nabla(\log u - \log u_{\lambda})_+|^2 \mathrm{d}x,
\end{align*}
where $C(S_\delta^\lambda)\to 0$ as $|S_\delta^\lambda|\to 0$. It follows from $u>0$ on $\overline{S_{\bar{\delta}}^{\lambda_0^- + \varepsilon_0}}$ that
\begin{align*}
\int_{S_\delta^\lambda} [(\log u - \log u_{\lambda})_+]^2 \mathrm{d}x \leq \frac{C (S_\delta^\lambda)}{\inf\limits_{S_{\bar{\delta}}^{\lambda_0^- + \varepsilon_0}} u^2} \int_{S_\delta^\lambda} u^2 (|\nabla u|+|\nabla u_{\lambda}|)^{p-2} |\nabla(\log u - \log u_{\lambda})_+|^2 \mathrm{d}x.
\end{align*}
Hence, for $p\geq2$, one has
\begin{equation}\label{eq1002}
\begin{aligned}
&\quad \int_{S_\delta^\lambda \cap \{u\geq u_\lambda\}} u^{p^\star} ( \log u - \log u_\lambda )^2 \mathrm{d}x \\
&\leq C_1(S_\delta^\lambda)  \int_{S_\delta^\lambda \cap \{u\geq u_\lambda\}} u^2 (|\nabla u|+|\nabla u_{\lambda}|)^{p-2} |\nabla(\log u - \log u_{\lambda})|^2 \mathrm{d}x,\quad\,\, \forall \,\,\lambda\in [\lambda_0^-,\lambda_0^-+\varepsilon_0],
\end{aligned}
\end{equation}
where $C_1(S_\delta^\lambda)\to 0$ as $|S_\delta^\lambda|\to 0$.

\medskip

If $1<p<2$, by \eqref{eq100227} and $u\in C_{loc}^{1,\alpha}(\R^N\setminus \{0\})$ for some $0<\alpha<\min\{1,\frac{1}{p-1}\}$, we get
$$(|\nabla u|+|\nabla u_\lambda|)^{2-p}<C \,\,\ \mbox{in} \,\,\ S_\delta^\lambda \setminus B_{\tilde{R}_0}(0_\lambda) \ \left(\supset S_\delta^\lambda \cap \{u\geq u_\lambda\}\right),$$
then it follows from the Poincar\'{e} type inequality that
\begin{equation}\label{eq1002+}
\begin{aligned}
&\int_{S_\delta^\lambda \cap \{u\geq u_\lambda\}} u^{p^\star} ( \log u - \log u_\lambda )^2 \mathrm{d}x
  \leq C\int_{ S_\delta^\lambda \cap \{u\geq u_\lambda\}} ( \log u - \log u_\lambda )^2 \mathrm{d}x \\
&\leq C \cdot C(S_\delta^\lambda) \int_{ S_\delta^\lambda \cap \{u\geq u_\lambda\}} |\nabla(\log u - \log u_{\lambda})|^2 \mathrm{d}x \\
&\leq C\cdot C(S_\delta^\lambda) \int_{S_\delta^\lambda \cap \{u\geq u_\lambda\}} \frac{1}{\inf_{S_{\bar{\delta}}^{\lambda_0^- + \varepsilon_0}} u^2} \cdot \frac{u^2}{u_\lambda^2} \cdot u_\lambda^2 \cdot \frac{|\nabla(\log u - \log u_{\lambda})|^2}{(|\nabla u|+|\nabla u_{\lambda}|)^{2-p}} \mathrm{d}x \\
&\leq C_2(S_\delta^\lambda) \int_{S_\delta^\lambda \cap \{u\geq u_\lambda\}} u_\lambda^2  \frac{|\nabla(\log u - \log u_{\lambda})|^2}{(|\nabla u|+|\nabla u_{\lambda}|)^{2-p}} \mathrm{d}x, \qquad \forall \,\, \lambda\in [\lambda_0^-,\lambda_0^-+\varepsilon_0],
\end{aligned}
\end{equation}
where $C_2(S_\delta^\lambda)\to 0$ as $|S_\delta^\lambda|\to 0$. Thus we get the desired estimates for the integral in \eqref{a0}.

\medskip

Finally, by substituting the estimates \eqref{050422}, \eqref{050423}, \eqref{eq1002} and \eqref{eq1002+} into \eqref{050420} and \eqref{050421}, we get, for any $\lambda\in [\lambda_0^-,\lambda_0^-+\varepsilon_0]$,
\begin{equation*}
\begin{aligned}
  &\quad C_p \tilde{c}^p \int_{\widetilde{B}_{2R}\cap\{u \geq u_\lambda\}} \eta^\tau u^2 (|\nabla u|+|\nabla u_\lambda|)^{p-2} |\nabla \log u- \nabla \log u_\lambda|^2 \mathrm{d}x \\
  &\leq \frac{C\tau}{R^{\beta_1}} + \frac{C}{R^\frac{\beta_1}{p}} + C \left\| u \right\|_{L^{p^\star} (\R^N)}^\frac{p(N-2p)}{N-p} \cdot \int_{ \Sigma_\lambda^\prime \cap \{u\geq u_\lambda\} } u^{p^\star} ( \log u - \log u_\lambda )^2 \mathrm{d}x\\
  &\leq \frac{C\tau}{R^{\beta_1}} + \frac{C}{R^\frac{\beta_1}{p}} + \frac{C}{\bar{R}^{\beta_2}}  \int_{ E_{\lambda,\bar{R}} \cap \{u\geq u_\lambda\} }  u^2 ( |\nabla u|+ |\nabla u_\lambda|)^{p-2} |\nabla ( \log u - \log u_\lambda )|^2 \mathrm{d}x\\
  &+ C_1(S_\delta^\lambda)  \int_{S_\delta^\lambda \cap \{u\geq u_\lambda\}} u^2 (|\nabla u|+|\nabla u_{\lambda}|)^{p-2} |\nabla(\log u - \log u_{\lambda})|^2 \mathrm{d}x, \,\,\, & \text{if} \,\, p\geq2,
\end{aligned}
\end{equation*}
and
\begin{equation*}
\begin{aligned}
   &\quad C_p \int_{\widetilde{B}_{2R}\cap\{u \geq u_\lambda\}} \eta^\tau u_\lambda^2 (|\nabla u|+|\nabla u_\lambda|)^{p-2} |\nabla \log u- \nabla \log u_\lambda|^2 \mathrm{d}x \\
   &\leq \frac{C\tau}{R^{\beta_1}} + \frac{C}{R^\frac{\beta_1}{p}} + C \left\| u \right\|_{L^{p^\star} (\R^N)}^\frac{p(N-2p)}{N-p} \cdot \int_{ \Sigma_\lambda^\prime \cap \{u\geq u_\lambda\} } u^{p^\star} ( \log u - \log u_\lambda )^2 \mathrm{d}x  \\
   &\leq \frac{C\tau}{R^{\beta_1}} + \frac{C}{R^\frac{\beta_1}{p}} +\frac{C}{\bar{R}^{\beta_2}} \cdot \int_{ E_{\lambda,\bar{R}} \cap \{u\geq u_\lambda\} }  u_\lambda^2  \frac{|\nabla (\log u - \log u_\lambda)|^2}{(|\nabla u| + |\nabla u_\lambda|)^{2-p}} \mathrm{d}x \\
   &+ C_2(S_\delta^\lambda) \int_{S_\delta^\lambda \cap \{u\geq u_\lambda\}} u_\lambda^2  \frac{|\nabla(\log u - \log u_{\lambda})|^2}{(|\nabla u|+|\nabla u_{\lambda}|)^{2-p}} \mathrm{d}x,\,\,\ & \text{if} \,\, 1<p<2,
\end{aligned}
\end{equation*}
where $C_i(S_\delta^\lambda)\to 0 \,\ (i=1,2)$ if $|S_\delta^\lambda|\to 0$, and $\tilde{c}$, $C_p$ and $C$ are independent of $R$ and $\lambda$. As a consequence, by choosing $\bar{R}$ sufficiently large and $\varepsilon_0, \bar{\delta}$ small enough such that $$\frac{C}{\bar{R}^{\beta_2}}+C_1(S_\delta^\lambda)<\frac{C_p \tilde{c}^p}{2} \,\,\,\ \mbox{and} \,\,\,\ \frac{C}{\bar{R}^{\beta_2}}+C_2(S_\delta^\lambda)<\frac{C_p}{2}, \,\,\quad \forall \,\, \lambda\in [\lambda_0^-,\lambda_0^-+\varepsilon_0],\,\,\delta \in (0,\bar{\delta}),$$
we obtain
\begin{equation}\label{050428}
\begin{aligned}
  &\quad C_p \tilde{c}^p \int_{ \Sigma_\lambda^\prime \cap \{u\geq u_\lambda\} }  u^2 ( |\nabla u|+ |\nabla u_\lambda|)^{p-2} |\nabla ( \log u - \log u_\lambda )|^2 \mathrm{d}x \\
  &=\lim_{R\to\infty} C_p \tilde{c}^p \int_{\widetilde{B}_{R}\cap\{u \geq u_\lambda\}}  u^2 (|\nabla u|+|\nabla u_\lambda|)^{p-2} |\nabla \log u- \nabla \log u_\lambda|^2 \mathrm{d}x \\
  &\leq \frac{C_p \tilde{c}^p}{2} \int_{ \Sigma_\lambda^\prime \cap \{u\geq u_\lambda\} }  u^2 ( |\nabla u|+ |\nabla u_\lambda|)^{p-2} |\nabla ( \log u - \log u_\lambda )|^2 \mathrm{d}x, \,\,\, & \text{if} \,\, p\geq2,
\end{aligned}
\end{equation}
and
\begin{equation}\label{050429}
\begin{aligned}
   &\quad C_p \int_{ \Sigma_\lambda^\prime \cap \{u\geq u_\lambda\} }  u_\lambda^2  \frac{|\nabla (\log u - \log u_\lambda)|^2}{(|\nabla u| + |\nabla u_\lambda|)^{2-p}} \mathrm{d}x \\
   &=\lim_{R\to\infty} C_p \int_{\widetilde{B}_{R}\cap\{u \geq u_\lambda\}} u_\lambda^2 (|\nabla u|+|\nabla u_\lambda|)^{p-2} |\nabla \log u- \nabla \log u_\lambda|^2 \mathrm{d}x \\
   &\leq \frac{C_p}{2} \int_{ \Sigma_\lambda^\prime \cap \{u\geq u_\lambda\} }  u_\lambda^2  \frac{|\nabla (\log u - \log u_\lambda)|^2}{(|\nabla u| + |\nabla u_\lambda|)^{2-p}} \mathrm{d}x,\,\,\ & \text{if} \,\, 1<p<2
\end{aligned}
\end{equation}
for all $\lambda\in [\lambda_0^-,\lambda_0^-+\varepsilon_0]$. Therefore, $u\leq u_\lambda$ on $\Sigma_\lambda$ for any $\lambda\in [\lambda_0^-,\lambda_0^-+\varepsilon_0]$, which contradicts the definition of $\lambda_0^-$. This proves that $\lambda_0^-=0$, and a similar approach gives $\lambda_0^+=0$.

\medskip

{\bf Step 3.} Conclusion.

From Step 2, we derive the symmetry of $u$ with respect to the hyperplane $\{x\in\mathbb{R}^{N}\mid \,x_{1}=0\}$.

\smallskip

By the rotation invariance of the problem, we can choose any direction $\nu\in\partial B_{1}(0)$ as the $e_{1}$-direction and exploit the moving plane procedure along $\nu$, and hence deduce that $u$ is radially symmetric and radially decreasing about the origin $0$. This shows that $u(x)=u(|x|)$ and $u^\prime(r)\leq0$ for all $r>0$.

\medskip

Next, we will show that
\begin{equation}\label{final}
  u^\prime (r) < 0, \quad \,\,\,\forall \,\, r>0.
\end{equation}
We will prove \eqref{final} by contradiction arguments. We assume on the contrary that there exists $r_0>0$ such that $u^\prime (r_0) = 0$. Since Lemma \ref{re2333} and Remark \ref{re2334} implies that $|Z_u|=0$, one has $u'(r)<0$ for $r\in(0,r_{0})\cup(r_{0},+\infty)$, so $u(r)>u(r_0)$ in $[0,r_0)$. Applying the strong maximum principle and H\"{o}pf's Lemma in Lemma \ref{hopf} to the positive radial solution $W(x):=u(x)-u(r_0)=u(|x|)-u(r_{0})$ of
\begin{align*}
\left\{ \begin{array}{ll} \displaystyle
-\Delta_p W  = \mu\frac{1}{|x|^p}\cdot u^{p-1}(x)+ \int_{\R^N} \frac{u^p(|y|)}{|x-y|^{2p}}\mathrm{d}y \cdot u^{p-1}(x) > 0  \,\,\,\,&\mbox{in}\, \,B_{r_0}, \\
W>0   &\mbox{in}\,\,  B_{r_0},\\
W=0   &\mbox{on}\,\, \partial B_{r_0},
\end{array}
\right.\hspace{1cm}
\end{align*}
we obtain $W^\prime (r_0)=u'(r_{0})<0$, which is absurd. Hence \eqref{final} holds, i.e., $u'(r)<0$ for any $r>0$. This completes our proof of Theorem \ref{th2}.
\end{proof}

\section{Proof of Theorem \ref{gth2}}\label{sc5}
\begin{proof}[Proof of Theorem \ref{gth2}]
For any nonnegative $D^{1,p}(\mathbb{R}^{N})$-weak solution $u$ to \eqref{eq1.3}, we will first show that the regularity results and the sharp asymptotic estimates of $u$ and $|\nabla u|$ in Theorem \ref{gth2}. In fact, we only need to verify that $V_3(x)=\left(|x|^{-\sigma}\ast u^{p_{s,\sigma}}\right)(x)u^{p_{s,\sigma}-p}(x)$ satisfies all the assumptions on $V$ and $u$ in Theorems \ref{th2.1.1}, \ref{th2.1++} and \ref{th2.1}, where $p_{s,\sigma}-p=\frac{(2p-\sigma-s)p}{2(N-p)}$.
\medskip

To this end, we will prove the following key estimates on $V_3$.
\begin{lem} \label{V-3-norm}
Let $u\in D^{1,p}(\R^N)$ be a nonnegative solution to \eqref{eq1.3} with $1<p<N$, $0\leq s<p$, $0\leq\mu<\bar{\mu}$, $p_{s,\sigma}=\frac{(2N-\sigma-s)p}{2(N-p)}$, $s<\sigma$, $0<\sigma+s\leq 2p$, and $V_3(x) \in L^{\frac{N}{p-s}}(\R^N)$. Let $\rho>0$ be small enough such that $ \|V_3(x)\|_{L^\frac{N}{p-s} (B_{\rho}(0))}\leq \epsilon_1$ and $\|V_3(x)\|_{L^\frac{N}{p-s} (\R^N\setminus B_{1/{\rho}}(0))} \leq \epsilon_1$, where $\epsilon_1>0$ is given by Lemma \ref{lm:l-b-2}. Then, for any $l\in\left(0,\min\{p-s,(\sigma-s)/2\}\right)$,
\begin{equation}\label{est-bdd-3-loc}
V_3(x)=\left(|x|^{-\sigma}\ast u^{p_{s,\sigma}}\right)u^{p_{s,\sigma}-p} \in L^{\frac{N}{p-l-s}}_{loc}(\mathbb{R}^{N}\setminus \{0\}),
\end{equation}
and
\begin{equation}\label{est-bdd-3}
\|V_3(x)\|_{L^{\frac{N}{p-l-s}} (B_{2R}(y))} \leq \tilde{C} R^{-l}, \,\,\,\,\quad\,\, \forall \,\, |y|\leq \frac{\rho}{16} \,\,\, \mbox{or} \,\,\, |y|\geq \frac{16}{\rho},
\end{equation}
where $R=|y|/4$ and the positive constant $\tilde{C}$ is independent of $y$.
\end{lem}
\begin{proof}
It follows from \eqref{060108-gen} that $V_3(x) \in L^{\frac{N}{p-s}}(\R^N)$. Next, we prove \eqref{est-bdd-3}. First, we note that, for any $l\in\left(0,\min\{p-s,(\sigma-s)/2\}\right)$, there exists $q\in(\frac{N}{p-s},\infty)$ such that $q=\frac{N}{p-s-l}$. For any $|y|\leq \frac{\rho}{16}$ or $|y|\geq \frac{16}{\rho}$, let $R=|y|/4$. Now, by H\"{o}lder's inequality, we have
\begin{equation}\label{060109}
\begin{aligned}
     &\quad \|V_3(x)\|_{L^{q}(B_{2R}(y))} \leq \|V_3(x)\|_{L^{\frac{N}{p-s-l}}(B_{6R}(0)\setminus B_{2R}(0))}
         =\|K(x) u^{p_{s,\sigma}-p}(x)\|_{L^{\frac{N}{p-s-l}}(B_{6R}(0)\setminus B_{2R}(0))}\\
     &\leq \left( \int_{B_{6R}(0)\setminus B_{2R}(0)} |K(x)|^{\frac{1}{{\frac{p-s-l}{N}-\frac{p_{s,\sigma}-p}{p^\star}}}} \mathrm{d}x \right)^{\frac{p-s-l}{N}-\frac{p_{s,\sigma}-p}{p^\star}} \left(\int_{B_{6R}(0)\setminus B_{2R}(0)} |u(x)|^{p^\star} \mathrm{d}x \right)^{\frac{p_{s,\sigma}-p}{p^\star}}\\
     &\leq \left( \int_{B_{6R}(0)\setminus B_{2R}(0)} |K(x)|^{\frac{2N}{\sigma-s-2l}} \mathrm{d}x \right)^{\frac{\sigma-s-2l}{2N}} \left\|u \right\|^{p_{s,\sigma}-p}_{L^{p^\star}(\R^N)},
\end{aligned}
\end{equation}
where $K(x)=\left(|x|^{-\sigma} \ast u^{p_{s,\sigma}}\right)(x)$ and $\frac{p_{s,\sigma}-p}{p^\star}=\frac{2p-\sigma-s}{2N}$. For $\|K(x)\|_{L^{\frac{2N}{\sigma-s-2l}}(B_{6R}\setminus B_{2R})}$, we get
\begin{equation}\label{0601092}
\begin{aligned}
      &\quad \|K(x)\|_{L^{\frac{2N}{\sigma-s-2l}}(B_{6R}(0)\setminus B_{2R}(0))}
        = \left\| \int_{\R^N} \frac{u^{p_{s,\sigma}}(z)}{|x-z|^{\sigma}} \mathrm{d}z \right\|_{L^{\frac{2N}{\sigma-s-2l}}(B_{6R}(0)\setminus B_{2R}(0))} \\
      &\leq \left\| \int_{B_{R/8}(0)} \frac{u^{p_{s,\sigma}}(z)}{|x-z|^{\sigma}}\mathrm{d}z \right\|_{L^{\frac{2N}{\sigma-s-2l}}(B_{6R}(0)\setminus B_{2R}(0))}
            + \left\| \int_{B^c_{8R}(0)} \frac{u^{p_{s,\sigma}}(z)}{|x-z|^{\sigma}} \mathrm{d}z \right\|_{L^{\frac{2N}{\sigma-s-2l}}(B_{6R}(0)\setminus B_{2R}(0))} \\
      &+ \left\| \int_{B_{8R}(0)\setminus B_{R/8}(0)} \frac{u^{p_{s,\sigma}}(z)}{|x-z|^{\sigma}} \mathrm{d}z \right\|_{L^{\frac{2N}{\sigma-s-2l}}(B_{6R}(0)\setminus B_{2R}(0))}\\
     &=:J_1 + J_2 +J_3.
\end{aligned}
\end{equation}
Since $|x-z|\geq 2R-R/8 >R$ for $x\in B_{6R}(0)\setminus B_{2R}(0)$ and $z\in B_{R/8}(0)$, by the Minkowski inequality and H\"{o}lder's inequality, we obtain that
\begin{equation}\label{0601093}
\begin{aligned}
      J_1 &= \left\| \int_{B_{R/8}(0)} \frac{u^{p_{s,\sigma}}(z)}{|x-z|^{\sigma}}\mathrm{d}z \right\|_{L^{\frac{2N}{\sigma-s-2l}}(B_{6R}(0)\setminus B_{2R}(0))} \\
      &\leq \int_{B_{R/8}(0)} \left( \int_{B_{6R}(0)\setminus B_{2R}(0)} \left( \frac{u^{p_{s,\sigma}}(z)}{|x-z|^{\sigma}} \right)^{\frac{2N}{\sigma-s-2l}} \mathrm{d}x \right)^{\frac{\sigma-s-2l}{2N}} \mathrm{d}z \\
      &\leq C R^{-\sigma}  |B_{6R}(0)\setminus B_{2R}(0)|^{\frac{\sigma-s-2l}{2N}}  \int_{B_{R/8}(0)} u^{p_{s,\sigma}}(z) \mathrm{d}z \\
      &\leq C R^{-\sigma+\frac{\sigma-s-2l}{2}} |B_{R/8}(0)|^{1-\frac{p_{s,\sigma}}{p^\star}}  \left(\int_{B_{R/8}(0)} u^{p^\star} \mathrm{d}z\right)^{\frac{p_{s,\sigma}}{p^\star}} \leq C R^{-l} \|u\|^{p_{s,\sigma}}_{L^{p^\star}(\R^N)},
\end{aligned}
\end{equation}
where $1- \frac{p_{s,\sigma}}{p^\star}=\frac{\sigma+s}{2N}$ and $C>0$ is independent of $y$. Due to $|x-z|>|z|-6R\geq|z|-\frac{3}{4}|z|=\frac{|z|}{4}$ for any $x\in B_{6R}(0)\setminus B_{2R}(0)$ and $z\in\left(B_{8R}(0)\right)^c$, we get
\begin{equation}\label{0601094}
\begin{aligned}
J_2 &= \left\| \int_{B^c_{8R}(0)} \frac{u^{p_{s,\sigma}}(z)}{|x-z|^{\sigma}} \mathrm{d}z \right\|_{L^{\frac{2N}{\sigma-s-2l}}(B_{6R}(0)\setminus B_{2R}(0))} \\
      &\leq \int_{B^c_{8R}(0)} \left( \int_{B_{6R}(0)\setminus B_{2R}(0)} \left( \frac{u^{p_{s,\sigma}}(z)}{|x-z|^{\sigma}} \right)^{\frac{2N}{\sigma-s-2l}} \mathrm{d}x \right)^{\frac{\sigma-s-2l}{2N}} \mathrm{d}z \\
      &\leq C |B_{6R}(0)\setminus B_{2R}(0)|^{\frac{\sigma-s-2l}{2N}} \int_{B^c_{8R}} \frac{u^{p_{s,\sigma}}(z)}{|z|^\sigma} \mathrm{d}z \\
      &\leq C R^{\frac{\sigma-s-2l}{2}}  \left( \int_{B^c_{8R}}|z|^{-\sigma \cdot \frac{p^\star}{p^\star-p_{s,\sigma}} } \mathrm{d}z \right)^{1-\frac{p_{s,\sigma}}{p^\star}}  \left(\int_{B^c_{8R}} u^{p^\star} \mathrm{d}z\right)^{\frac{p_{s,\sigma}}{p^\star}} \leq C R^{-l} \|u\|^{p_{s,\sigma}}_{L^{p^\star}(\R^N)},
\end{aligned}
\end{equation}
where $C>0$ independent of $y$.
\medskip

By Hardy-Littlewood-Sobolev inequality in Theorem \ref{HLSI} and Lemma \ref{lm:l-b-2} with $\bar{p}=\frac{(2N-\sigma-s) p^\star}{2N-\sigma-s-2l}$, recalling that $R=|y|/4$, we deduce that, for any $|y|\leq \rho/16$ or $|y|\geq 16/\rho$,
\begin{equation}\label{0601095}
\begin{aligned}
     J_3  & \leq \left\| \int_{B_{8R}(0)\setminus B_{R/8}(0)} \frac{u^{p_{s,\sigma}}(z)}{|x-z|^{\sigma}} \mathrm{d}z \right\|_{L^{\frac{2N}{\sigma-s-2l}}(\R^N)} \\
          &\leq C \left\| u^{p_{s,\sigma}} \right\|_{L^{\frac{2N}{2N-\sigma-s-2l}}(B_{8R}(0)\setminus B_{R/8}(0))} = C \left\| u \right\|^{p_{s,\sigma}}_{L^{\frac{2N}{2N-\sigma-s-2l} \cdot \frac{(2N-\sigma-s)p}{2(N-p)}}(B_{8R}(0)\setminus B_{R/8}(0))} \\
          &\mathop{\leq}_{\text{Lemma \,3.3}} C \left( R^{N \cdot \frac{2N-\sigma-s-2l}{2N-\sigma-s} \cdot \frac{N-p}{Np} - \frac{N-p}{p}} \right)^{\frac{(2N-\sigma-s)p}{2(N-p)}} \|u\|^{p_{s,\sigma}}_{L^{p^\star}(B_{16R}\setminus B_{R/16})} \leq C R^{-l} \|u\|^{p_{s,\sigma}}_{L^{p^\star}(\R^N)},
\end{aligned}
\end{equation}
where $C>0$ is independent of $y$. Then, by substituting \eqref{0601093}, \eqref{0601094} and \eqref{0601095} into \eqref{0601092}, we derive \eqref{est-bdd-3}.

\smallskip

Next, we show \eqref{est-bdd-3-loc}. Since $V_{3}(x) \in L^{\frac{N}{p-s}}(\R^N)$, it follows from Lemma \ref{lm:l-b-r} that $u\in L^{r}_{loc}(\mathbb{R}^{N}\setminus\{0\})$ for any $0<r<+\infty$. Then, for any $y\in\mathbb{R}^{N}\setminus\{0\}$ and $R=\frac{|y|}{4}$, from \eqref{0601092}, \eqref{0601093}, \eqref{0601094} and the estimate $J_3 \leq  C \left\| u \right\|^{p_{s,\sigma}}_{L^{\frac{2N}{2N-\sigma-s-2l} \cdot \frac{(2N-\sigma-s)p}{2(N-p)}}(B_{8R}(0)\setminus B_{R/8}(0))}$ indicated by \eqref{0601095}, we get
\begin{equation*}
  \|V_3\|_{L^{\frac{N}{p-s-l}}(B_{2R}(y))}\leq C R^{-l} \|u\|^{p_{s,\sigma}}_{L^{p^\star}(\R^N)}+C \left\| u \right\|^{p_{s,\sigma}}_{L^{\frac{(2N-\sigma-s) p^\star}{2N-\sigma-s-2l}}(B_{8R}(0)\setminus B_{R/8}(0))}<+\infty,
\end{equation*}
which combined with the Heine-Borel finite covering theorem yield that $V_{3}\in L^{\frac{N}{p-s-l}}_{loc}(\mathbb{R}^{N}\setminus \{0\})$, i.e., \eqref{est-bdd-3-loc} holds. This completes the proof of Lemma \ref{V-3-norm}.
\end{proof}

Now we continue the proof of Theorem \ref{gth2}.

\smallskip

From \eqref{060108-gen}, we have $0\leq V_3(x) \in L^{\frac{N}{p-s}}(\mathbb{R}^{N})$. Let $\rho>0$ be small enough such that $ \|V_3\|_{L^\frac{N}{p-s} (B_{\rho}(0))}\leq \min\{\epsilon_1,\epsilon_2\}$ and $\|V_3\|_{L^\frac{N}{p-s} (\R^N\setminus B_{1/{\rho}}(0))} \leq \min\{\epsilon_1,\epsilon_2\}$, where $\epsilon_1>0$ and $\epsilon_2>0$ are given by Lemmas \ref{lm:l-b-2} and \ref{lm:l-b-3} respectively. Then, from Theorems \ref{th2.1.1} and \ref{th2.1++} and Lemma \ref{V-3-norm}, we get that $u\in C^{1,\alpha}(\R^N\setminus \{0\})\cap L^{\infty}_{loc}(\mathbb{R}^{N}\setminus \{0\})$ for some $0<\alpha<\min\{1,\frac{1}{p-1}\}$, and there exists positive constant $C=C(N,p,s,\mu,\rho,u,\tilde{C})>0$  such that
$$ u(y) \leq C |y|^{-\frac{N-p}{p}+\tau_1} \leq C |y|^{-\frac{N-p}{p}+\tau}, \,\,\,\,\,\,\,\quad \forall \,\, |y|\leq \frac{\rho}{16},$$
and
$$ u(y) \leq C |y|^{-\frac{N-p}{p}-\frac{\tau_2}{2}} \leq C |y|^{-\frac{N-p}{p}-\tau}, \,\,\,\,\,\,\,\quad \forall \,\, |y|\geq \frac{16}{\rho},$$
where $\tau=\frac{1}{3}\min\left\{\frac{\sigma-s}{p_{s,\sigma}},\tau_1,\tau_2\right\}$, $\tau_1, \tau_2>0$ are given by Lemma \ref{lm:l-b-3} and $\tilde{C}$ is given by \eqref{est-bdd-3} in Lemma \ref{V-3-norm}.

\medskip

One can verify that $N>p_{s,\sigma}\left( \frac{N-p}{p}-\tau \right)>N-\sigma$ and $N>p_{s,\sigma}\left( \frac{N-p}{p}+\tau \right)>N-\sigma$. Then, using (ii) of Lemma \ref{lm.6} with $\nu=\sigma$, $g(x)=u^{p_{s,\sigma}}(x)$, $\beta_1=p_{s,\sigma}\left( \frac{N-p}{p}-\tau \right)$ and $\beta_2=p_{s,\sigma}\left( \frac{N-p}{p}+\tau \right)$, we have, for $|x|<\rho/16$,
\begin{equation}\label{V-3-in-1}
\begin{aligned}
V_3(x)=\left(|x|^{-\sigma}\ast u^{p_{s,\sigma}}\right)(x)u^{p_{s,\sigma}-p}(x)\leq C|x|^{-\beta_1-\sigma+N}|x|^{-(p_{s,\sigma}-p)\left(\frac{N-p}{p}-\tau \right)}=:C|x|^{-\tilde{\beta}_1},
\end{aligned}
\end{equation}
and for $|x|>16/\rho$,
\begin{equation}\label{V-3-ex-2}
\begin{aligned}
V_3(x)=\left(|x|^{-\sigma}\ast u^{p_{s,\sigma}}\right)(x)u^{p_{s,\sigma}-p}(x)\leq C|x|^{-\beta_2-\sigma+N}|x|^{-(p_{s,\sigma}-p)\left(\frac{N-p}{p}+\tau \right)}=:C|x|^{-\tilde{\beta}_2},
\end{aligned}
\end{equation}
where
\[\tilde{\beta}_1=\beta_1 +\sigma-N + (p_{s,\sigma}-p)\left(\frac{N-p}{p}-\tau \right) <\frac{\sigma-s}{2}-p_{s,\sigma}\tau+\frac{2p-\sigma-s}{2} <p-s,\]
and
\[\tilde{\beta}_2=\beta_2 +\sigma-N + (p_{s,\sigma}-p)\left(\frac{N-p}{p}+\tau \right) >\frac{\sigma-s}{2}+p_{s,\sigma}\tau+\frac{2p-\sigma-s}{2} >p-s.\]

\smallskip

Thus Theorems \ref{th2.1.1} and \ref{th2.1} can be applied to the general nonlocal quasi-linear equation \eqref{eq1.3}. Consequently, the positive solution $u$ to \eqref{eq1.3} satisfies the regularity result in Theorem \ref{th2.1.1} and the sharp asymptotic estimates \eqref{eq0806}-\eqref{eq0806+} in Theorem \ref{th2.1}.

\medskip

Finally, we will prove the results of radial symmetry and monotonicity in Theorem \ref{gth2}. In fact, in the proof of Theorem \ref{th2} in Section \ref{sc4}, noting that $u < u_\lambda$ in $B_{\tilde{R}_0}(0_\lambda)$ and hence $\Sigma_\lambda\cap \{u\geq u_\lambda\}=\Sigma_\lambda'\cap \{u\geq u_\lambda\}$ and $\Sigma_\lambda\cap B_{2R}\cap \{u\geq u_\lambda\}=\Sigma_\lambda^\prime\cap B_{2R}\cap \{u\geq u_\lambda\}=\widetilde{B}_{2R}\cap \{u\geq u_\lambda\}$, we can see clearly from \eqref{050410}, \eqref{050420} and \eqref{050421} that, the crucial inequality in the estimate of the key integral term $I_{4}$ is the following: for any $\lambda<-\kappa<0$,
\begin{equation}\label{eqmv-I4-1}
\begin{aligned}
   I_4 &=\int_{\widetilde{B}_{2R}}\left( V_1(x) - V_1(x_\lambda) \right)\eta^\tau (u^p-u^p_\lambda)_+ \mathrm{d}x \\
       &\leq C\int_{ \Sigma_\lambda^\prime \cap \{u\geq u_\lambda\} } u^{p^\star} ( \log u - \log u_\lambda )^2 \mathrm{d}x
       =C \int_{\Sigma_\lambda} u^{p^\star} \left[(\log u - \log u_\lambda)_+ \right]^2 \mathrm{d}x, \ \ \quad \ \forall \,\, R>0,
\end{aligned}
\end{equation}
where $C>0$ may depend on $\kappa>0$ but is independent of $R$ and $\lambda$. Moreover, in \eqref{050410}, we deduced further from \eqref{eqmv-I4-1} that, for $\lambda<-R_1$,
\begin{equation}\label{eqmv-I4-2}
\begin{aligned}
   I_4 \leq C \frac{1}{|\lambda|^{\beta_2}} \int_{ \Sigma_\lambda^\prime \cap \{u\geq u_\lambda\} }  \frac{1}{|x|^{-2\theta+2}} ( \log u - \log u_\lambda )^2 \mathrm{d}x,
\end{aligned}
\end{equation}
where $\beta_2=\gamma_2 (p^\star-p)-p>0$ and $2\theta=-\left[ (\gamma_2+1)(p-2)+2\gamma_2 \right]$, $R_1$ is the radius given by \eqref{eq0806++} and \eqref{eq0806+}, and $C>0$ is independent of $R$ and $\lambda$.

\medskip

Keeping these key ingredients in the proof of Theorem \ref{th2} in Section \ref{sc4} in mind, in order to complete the proof of Theorem \ref{gth2}, we define
\begin{equation}\label{eqmv3-Ic-1}
\begin{aligned}
   \hat{I}_4:&=\int_{\widetilde{B}_{2R}}\left( \bar{V}_3(x) - \bar{V}_3(x_\lambda) \right)\eta^{\tilde{\tau}}(u^p-u^p_\lambda)_+ \mathrm{d}x \\
     &=\int_{\Sigma_\lambda\cap B_{2R}} \left(\bar{V}_3(x)-\bar{V}_3(x_\lambda)\right) \eta^{\tilde{\tau}}(u^p-u^p_\lambda)_+ \mathrm{d}x,
\end{aligned}
\end{equation}
where $\bar{V}_3(x):=|x|^{-s}\left(|x|^{-\sigma}\ast u^{p_{s,\sigma}}\right)(x)u^{p_{s,\sigma}-p}(x)$, $\tilde{\tau} > \max\{2,p\}$ and $\eta \in C_0^\infty (B_{2R}(0))$ is a cut-off function such that $0\leq\eta\leq 1$, $\eta=1$ in $B_R(0)$ and $|\nabla \eta|\leq 2/R$. One can see clearly from \eqref{eqmv-I4-1} that the main key point for us is to prove, for any $\lambda<-\kappa<0$,
\begin{equation}\label{eqmv3-15-1}
\begin{aligned}
   \hat{I}_4&=\int_{\Sigma_\lambda\cap B_{2R}} \left(\bar{V}_3(x)-\bar{V}_3(x_\lambda)\right)\eta^{\tilde{\tau}}(u^p-u^p_\lambda)_+ \mathrm{d}x \\
   &\leq C \int_{\Sigma_\lambda} u^{p^\star} \left[(\log u - \log u_\lambda)_+ \right]^2 \mathrm{d}x + C \int_{\Sigma_\lambda} K(x) \frac{ u^{p_{s,\sigma}}}{|x|^s} \left[(\log u - \log u_\lambda)_+\right]^2 \mathrm{d}x, \,\,\,\,\, \forall \,\, R>0,
\end{aligned}
\end{equation}
where $K(x)=|x|^{-\sigma}*u^{p_{s,\sigma}}$, and $C>0$ may depend on $\kappa>0$ but is independent of $R$ and $\lambda$. In order to prove \eqref{eqmv3-15-1}, we need the following Lemma.
\begin{lem}\label{v3-int-esta}
Let $u\in D^{1,p}(\R^N)$ be a nonnegative solution to \eqref{eq1.3} with $1<p<N$, $0\leq s<p$, $0\leq\mu<\bar{\mu}$, $p_{s,\sigma}=\frac{(2N-\sigma-s)p}{2(N-p)}$, $s<\sigma$ and $0<\sigma+s\leq 2p$. Assume $\kappa>0$. Then, for any $R>0$,
\begin{equation}\label{eqmv-16}
\begin{aligned}
   &\quad \int_{\Sigma_\lambda\cap B_{2R}}  \left(\bar{V}_3(x)-\bar{V}_3(x_\lambda)\right) (u^p-u^p_\lambda)_+ \mathrm{d}x \\
   &\leq C \int_{\Sigma_\lambda} u^{p^\star} \left[(\log u - \log u_\lambda)_+ \right]^2 \mathrm{d}x \\
   &\  + C \int_{\Sigma_\lambda} K(x) \frac{ u^{p_{s,\sigma}}}{|x|^s} \left[(\log u - \log u_\lambda)_+\right]^2 \mathrm{d}x, \,\,\,\,\quad \forall \,\, \lambda<-\kappa<0,
\end{aligned}
\end{equation}
where $\bar{V}_3(x)=|x|^{-s} \left(|x|^{-\sigma}\ast u^{p_{s,\sigma}}\right)u^{p_{s,\sigma}-p}$ and $C>0$ is independent of $R$ and $\lambda$.
\end{lem}
\begin{proof}
From the fact that $|x|>|x_\lambda|$ in $\Sigma_\lambda$ and Lagrange's mean value Theorem, we have
\begin{equation}\label{eqmv-21}
\begin{aligned}
       &\quad \int_{\Sigma_\lambda\cap B_{2R}}\left[ \bar{V}_3(x) - \bar{V}_3(x_\lambda) \right] (u^p-u^p_\lambda)_+ \mathrm{d}x \\
       &\leq C \int_{\Sigma_\lambda\cap B_{2R}}\left[K(x) \frac{u^{p_{s,\sigma}-p}}{|x|^s} - K(x_\lambda) \frac{u_\lambda^{p_{s,\sigma}-p}}{|x_\lambda|^s} \right] u^{p-1} (u - u_\lambda)_+ \mathrm{d}x \\
       &\leq C \int_{\Sigma_\lambda\cap B_{2R}}\left[K(x) u^{p_{s,\sigma}-p} - K(x_\lambda) u_\lambda^{p_{s,\sigma}-p} \right]  \frac{1}{|x|^s} u^{p-1} (u - u_\lambda)_+ \mathrm{d}x \\
       &\leq C \int_{\Sigma_\lambda\cap B_{2R}} K(x) \left[ u^{p_{s,\sigma}-p} - u_\lambda^{p_{s,\sigma}-p} \right]  \frac{1}{|x|^s} u^{p-1} (u - u_\lambda)_+ \mathrm{d}x \\
              &+ C \int_{\Sigma_\lambda\cap B_{2R}}\left[K(x)- K(x_\lambda) \right]  u_\lambda^{p_{s,\sigma}-p} \frac{1}{|x|^s} u^{p-1} (u - u_\lambda)_+ \mathrm{d}x =: \hat{I}_1+ \hat{I}_2, \,\quad \forall \,\, \lambda<0,
\end{aligned}
\end{equation}
where $K(x)=|x|^{-\sigma}\ast u^{p_{s,\sigma}}$ and $C>0$ is independent of $R$ and $\lambda$. It follows from the sharp asymptotic estimates \eqref{eq0806}--\eqref{eq0806+} and regularity result in Theorem \ref{gth2} that the estimate \eqref{050403} in Lemma \ref{a} holds. By \eqref{050403}, \eqref{log-ineq} and Lagrange's mean value Theorem, we have, for $1<p_{s,\sigma}<2$,
\begin{equation}\label{eqmv-22}
\begin{aligned}
   \hat{I}_1 &\leq C \int_{\Sigma_\lambda} K(x) \left[ \frac{u^{p_{s,\sigma}-1}}{u^{p-1}} - \frac{u_\lambda^{p_{s,\sigma}-1}}{u_\lambda^{p-1}} \right]  \frac{1}{|x|^s} u^{p-1} (u - u_\lambda)_+ \mathrm{d}x \\
        &\leq C \int_{\Sigma_\lambda} K(x) \left[ u^{p_{s,\sigma}-1} - u_\lambda^{p_{s,\sigma}-1} \right]  \frac{1}{|x|^s}  (u - u_\lambda)_+ \mathrm{d}x \\
        &\leq C \int_{\Sigma_\lambda} K(x) u_\lambda^{p_{s,\sigma}-2} \frac{1}{|x|^s} \left[(u - u_\lambda)_+\right]^2 \mathrm{d}x\\
        &\leq C \int_{\Sigma_\lambda} K(x) \frac{u^{2-p_{s,\sigma}}}{u_\lambda^{2-p_{s,\sigma}}} \frac{u^{p_{s,\sigma}-2}}{|x|^s} u^2 \left[(\log u - \log u_\lambda)_+\right]^2 \mathrm{d}x \\
        &\leq C \tilde{c}^{p_{s,\sigma}-2} \int_{\Sigma_\lambda} K(x) \frac{ u^{p_{s,\sigma}}}{|x|^s} \left[(\log u - \log u_\lambda)_+\right]^2 \mathrm{d}x, \,\,\quad \forall \,\, \lambda<0,
\end{aligned}
\end{equation}
and for $p_{s,\sigma} \geq 2$,
\begin{equation}\label{eqmv-23}
\begin{aligned}
   \hat{I}_1
        &\leq C \int_{\Sigma_\lambda} K(x) \left[ u^{p_{s,\sigma}-1} - u_\lambda^{p_{s,\sigma}-1} \right]  \frac{1}{|x|^s}  (u - u_\lambda)_+ \mathrm{d}x \\
        &\leq C \int_{\Sigma_\lambda} K(x) u^{p_{s,\sigma}-2} \frac{1}{|x|^s} \left[(u - u_\lambda)_+\right]^2 \mathrm{d}x \\
        &\leq C \int_{\Sigma_\lambda} K(x) \frac{ u^{p_{s,\sigma}}}{|x|^s} \left[(\log u - \log u_\lambda)_+\right]^2 \mathrm{d}x, \,\,\quad \forall \,\, \lambda<0,
\end{aligned}
\end{equation}
where $C$ is independent of $R$ and $\lambda$.
\medskip

Now, we estimate $\hat{I}_2$. From Lemma \ref{lm.1} and the fact that $u^{p_{s,\sigma}}$ convex in $u$, one has that
\begin{equation*}
\begin{aligned}
   K(x)- K(x_\lambda)&= \int_{\R^N}\frac{u^{p_{s,\sigma}}(y)}{|x-y|^\sigma}\mathrm{d}y -  \int_{\R^N}\frac{u^{p_{s,\sigma}}(y)}{|x_\lambda-y|^\sigma}\mathrm{d}y\\
   &= \int_{\Sigma_\lambda} \left(\frac{1}{|x-y|^\sigma} -  \frac{1}{|x_\lambda-y|^\sigma} \right) \left[ u^{p_{s,\sigma}}(y) - u_\lambda^{p_{s,\sigma}}(y) \right]\mathrm{d}y\\
   &\leq p_{s,\sigma} \int_{\Sigma_\lambda} \frac{u^{p_{s,\sigma}-1}(y)}{|x-y|^\sigma}  \left(u-u_\lambda \right)_+(y) \mathrm{d}y,
\end{aligned}
\end{equation*}
where in the last inequality we used the fact that $|x-y|\leq|x_\lambda-y|$ for any $x,y\in\Sigma_\lambda$. Then, recalling $p_{s,\sigma}\geq p >1$ and \eqref{050403} in Lemma \ref{a}, we obtain
\begin{equation*}
\begin{aligned}
   \hat{I}_2 &\leq C \int_{\Sigma_\lambda\cap B_{2R}}\left[K(x)- K(x_\lambda) \right]_+  u_\lambda^{p_{s,\sigma}-1} \frac{1}{|x|^s} \frac{u^{p-1}}{u_\lambda^{p-1}} (u - u_\lambda)_+ \mathrm{d}x \\
        &\leq C \tilde{c}^{1-p} \int_{\Sigma_\lambda}\left[K(x)- K(x_\lambda) \right]_+  \frac{u_\lambda^{p_{s,\sigma}-1}}{|x|^s} (u - u_\lambda)_+ \mathrm{d}x \\
        &\leq  C \tilde{c}^{1-p} \int_{\Sigma_\lambda} \int_{\Sigma_\lambda} \frac{u^{ p_{s,\sigma}-1}(x)(u - u_\lambda)_+(x) \cdot u^{ p_{s,\sigma}-1}(y)(u - u_\lambda)_+(y)}{|x|^s |x-y|^\sigma} \mathrm{d}x\mathrm{d}y, \,\,\quad \forall \,\, \lambda<0.
\end{aligned}
\end{equation*}
From the double weighted Hardy-Littlewood-Sobolev inequality in Theorem \ref{HLSI+} with $\alpha=s$, $\beta=0$ and $r=t=2N/(2N-\sigma-s)$, H\"{o}lder' inequality and \eqref{log-ineq}, we obtain
\begin{equation}\label{eqmv-25}
\begin{aligned}
   \hat{I}_2 &\leq C \left\| u^{ p_{s,\sigma}-1} \cdot (u - u_\lambda)_+ \right\|^2_{ L^{\frac{2N}{2N-\sigma-s}} (\Sigma_\lambda) }\\
   &\leq C \left\| u^{ p_{s,\sigma}} \cdot (\log u - \log u_\lambda)_+ \right\|^2_{ L^{\frac{2N}{2N-\sigma-s}} (\Sigma_\lambda) }\\
   &\leq C \left\| u^{\frac{N-\sigma-s}{2(N-p)}p } \cdot u^{\frac{p^\star}{2}} \cdot (\log u - \log u_\lambda)_+ \right\|^2_{ L^{\frac{2N}{2N-\sigma-s}} (\Sigma_\lambda) }\\
   &\leq C \left( \int_{\Sigma_\lambda} u^{p^\star} \mathrm{d}x \right)^{\frac{N-\sigma-s}{N}}\int_{\Sigma_\lambda} u^{p^\star} \left[ (\log u - \log u_\lambda)_+ \right]^2 \mathrm{d}x, \,\,\quad \forall \,\, \lambda<0,
\end{aligned}
\end{equation}
where $C$ is independent of $R$ and $\lambda$. From \eqref{eqmv-22}, \eqref{eqmv-23} and \eqref{eqmv-25}, we can deduce that \eqref{eqmv-16} holds and conclude the proof of Lemma \ref{v3-int-esta}.
\end{proof}

Now we continue the proof of Theorem \ref{gth2}. From \eqref{eqmv-16} in Lemma \ref{v3-int-esta}, we get the key estimate \eqref{eqmv3-15-1}. Next, similar to \eqref{eq10.26.30.03+}, by (ii) of Lemma \ref{lm.6}, we have, for any $\tilde{R}>1$,
\begin{equation}\label{V-3-e-k-1}
\begin{aligned}
&K(x)=|x|^{-\sigma}*u^{p_{s,\sigma}}  \leq C |x|^{-\frac{\sigma-s}{2}} \,\,\,\quad\,\,\,\mbox{in}\,\, \R^N\setminus B_{2\tilde{R}}(0),\\
&K(x)=|x|^{-\sigma}\ast u^{p_{s,\sigma}} \leq C_1 \max\{|x|^{-\frac{\sigma-s}{2}},1\} \,\,\,\,\quad\,\,\mbox{in}\,\, B_{\tilde{R}}(0),
\end{aligned}
\end{equation}
where $C$ is independent of $\tilde{R}$ and $C_1=C_1(\tilde{R})>0$. Based \eqref{eqmv-I4-2}, \eqref{eqmv3-15-1} and \eqref{V-3-e-k-1}, we can deduce further that
\begin{equation}\label{eqmv3-15-2}
\begin{aligned}
   \hat{I}_4 &\leq C \frac{1}{|\lambda|^{\beta_2}} \int_{ \Sigma_\lambda^\prime \cap \{u\geq u_\lambda\} }  \frac{1}{|x|^{-2\theta+2}} ( \log u - \log u_\lambda )^2 \mathrm{d}x \\
    &+ C \frac{1}{|\lambda|^{\hat{\beta}_2}} \int_{ \Sigma_\lambda^\prime \cap \{u\geq u_\lambda\} }  \frac{1}{|x|^{-2\theta+2}} ( \log u - \log u_\lambda )^2 \mathrm{d}x, \,\,\,\,\ \forall \ \lambda<-R_1,
\end{aligned}
\end{equation}
where $\beta_2>0$ and $\theta$ are given by \eqref{eqmv-I4-2}, $\hat{\beta}_2:= \frac{\sigma-s}{2} + s + \gamma_2 p_{s,\sigma} +2\theta-2=\frac{\sigma+s}{2}+\gamma_2 (p_{s,\sigma}-p)-p>0$ due to $\gamma_{2}>\frac{N-p}{p}$, and $C>0$ is independent of $R$ and $\lambda$. In the proof of Theorem \ref{gth2}, the key estimates \eqref{eqmv3-15-1} and \eqref{eqmv3-15-2} will lead to similar inequalities as \eqref{050410}, \eqref{050420} and \eqref{050421}, and hence play the same crucial role as \eqref{eqmv-I4-1} and \eqref{eqmv-I4-2} in the proof of Theorem \ref{th2}.

\medskip

With these key estimates in hands, the rest of the proof of Theorem \ref{gth2} is entirely similar to that of Theorem \ref{th2} in Section \ref{sc4}, so we omit the details. We only need to note that, the assumption ``$\sigma+s\geq2$ or $\sigma+s<2$ and $\left(\frac{N-1}{p_{s,\sigma}}\right)^{p-1}\left[(p-1)\frac{N-1}{p_{s,\sigma}}-(N-p)\right]+\mu<0$" (i.e., ``$p_{s,\sigma}\gamma_{1}+1<N$") is necessary for the estimate \eqref{eq10.26.30.02+}, which guarantee that $u$ fulfills the ``condition ($I_{\alpha}$)" in Lemma \ref{lm.4} and hence Corollary \ref{re2333} holds for the generalized (weighted) doubly $D^{1,p}$-critical nonlocal quasi-linear equation \eqref{eq1.3}. Thus, we can derive $|Z_{u}|=0$ (see Remark \ref{re2334}), and can apply the weighted Poincar\'{e} type inequality in Lemma \ref{lm.3} with the weight $\rho=|\nabla u|^{p-2}$ when $p\geq2$ (see Remark \ref{rem0}). This completes the proof of Theorem \ref{gth2}.
\end{proof}

\begin{rem}
By the regularity and sharp asymptotic estimates in Remark \ref{loc-u-rem}, we can show the radial symmetry and strictly radial monotonicity for solutions to the weighted doubly $D^{1,p}$-critical quasi-linear equation \eqref{eq1.2} with local nonlinearity. Note that problem \eqref{eq1.2} can be written as
\begin{align*}
\left\{ \begin{array}{ll} \displaystyle
-\Delta_p u - \frac{\mu}{|x|^p} u^{p-1}=\bar{V}_2(x)u^{p-1}  \quad\,\,\,\,&\mbox{in}\,\, \R^N, \\ \\
u \in D^{1,p}(\R^N),\quad\,\,\,\,\,\,u\geq0 \qquad\,\,\,\,&  \mbox{in}\,\, \R^N,
\end{array}
\right.\hspace{1cm}
\end{align*}
where $0\leq\mu<\bar{\mu}$, $1<p<N$, $0\leq s <p$, $\bar{V}_2(x):=|x|^{-s} u^{p_s-p}$ and $p_s=\frac{(N-s)p}{N-p}$. In Remark \ref{loc-u-rem}, we have already proved that the solution $u\in C^{1,\alpha} (\R^N\setminus \{0\}) \cap L_{loc}^\infty(\R^N\setminus \{0\})$ for some $0<\alpha<\min\{1,\frac{1}{p-1}\}$ and $u$ satisfies the sharp asymptotic estimates \eqref{eq0806}-\eqref{eq0806+}.

Keeping those key ingredients in the proof of Theorem \ref{th2} in Section \ref{sc4} in mind, we define
\begin{equation}\label{eqmv-Ic-1}
\begin{aligned}
   I_4':&=\int_{\widetilde{B}_{2R}}\left( \bar{V}_2(x) - \bar{V}_2(x_\lambda) \right)\eta^\tau (u^p-u^p_\lambda)_+ \mathrm{d}x \\
     &=\int_{\Sigma_\lambda\cap B_{2R}} \left(\bar{V}_2(x)-\bar{V}_2(x_\lambda)\right) \eta^\tau (u^p-u^p_\lambda)_+ \mathrm{d}x.
\end{aligned}
\end{equation}
One can see clearly from \eqref{eqmv-I4-1} that the main key point for us is to prove, for any $\lambda<0$,
\begin{equation}\label{eqmv-15-1}
\begin{aligned}
   I_4'&=\int_{\Sigma_\lambda\cap B_{2R}} \left(\bar{V}_2(x)-\bar{V}_2(x_\lambda)\right)\eta^{\tau}(u^p-u^p_\lambda)_+ \mathrm{d}x \\
   &\leq C \int_{\Sigma_\lambda} \frac{1}{|x|^s} u^{p_s} \left[(\log u - \log u_\lambda)_+ \right]^2 \mathrm{d}x, \,\,\,\,\quad \forall \,\, R>0.
\end{aligned}
\end{equation}
Indeed, using the facts that $u\geq u_\lambda$ in $supp ((u^p-u^p_\lambda)_+)$ and $|x|>|x_\lambda|$ in $\Sigma_\lambda$, we have (recalling that $p_s>p>1$)
\begin{equation}\label{eqmv-15-I}
\begin{aligned}
   I_4'&=\int_{\Sigma_\lambda \cap B_{2R}} \left[\frac{u^{p_s-1}}{u^{p-1}|x|^s} - \frac{u_\lambda^{p_s-1}}{u_\lambda^{p-1}|x_\lambda|^s} \right] (u^p-u^p_\lambda)_+ \mathrm{d}x \\
     &\leq \int_{\Sigma_\lambda\cap B_{2R}} \frac{1}{u^{p-1}|x|^s} (u^{p_s-1} - u_\lambda^{p_s-1}) (u^p-u^p_\lambda)_+ \mathrm{d}x.
\end{aligned}
\end{equation}
Next, applying twice Lagrange's mean value theorem in \eqref{eqmv-15-I}, we derive, for $1<p_s\leq2$ (recalling \eqref{050403}),
\begin{equation*}
\begin{aligned}
   I_4' &\leq C\int_{\Sigma_\lambda} \frac{1}{|x|^s} u_\lambda^{p_s-2} \left[(u-u_\lambda)_+ \right]^2 \mathrm{d}x \\
     &= C \int_{\Sigma_\lambda} \frac{1}{|x|^s} \left(\frac{u}{u_\lambda}\right)^{2-p_s} u^{p_s-2} \left[(u-u_\lambda)_+ \right]^2 \mathrm{d}x \\
     &\leq C \tilde{c}^{p_s-2} \int_{\Sigma_\lambda} \frac{1}{|x|^s} u^{p_s-2} \left[(u-u_\lambda)_+ \right]^2 \mathrm{d}x,
\end{aligned}
\end{equation*}
and for $p_s>2$,
\begin{equation*}
   I_4'\leq C\int_{\Sigma_\lambda} \frac{1}{|x|^s} u^{p_s-2} \left[(u-u_\lambda)_+ \right]^2 \mathrm{d}x,
\end{equation*}
which together with the inequality \eqref{log-ineq} give the proof of the key estimate \eqref{eqmv-15-1}.

\medskip

Based on \eqref{eqmv-15-1}, we can deduce further that
\begin{equation}\label{eqmv-15-2}
I_4'\leq C \frac{1}{|\lambda|^{\tilde{\beta}_2}} \int_{ \Sigma_\lambda^\prime \cap \{u\geq u_\lambda\} }  \frac{1}{|x|^{-2\theta+2}} ( \log u - \log u_\lambda )^2 \mathrm{d}x, \,\,\,\quad\, \forall \,\, \lambda<-R_1,
\end{equation}
where $\tilde{\beta}_2 = s + \gamma_2 p_s +2\theta-2 =s+\gamma_2\left(1+\frac{p-s}{N-p}\right)p - \gamma_2p - p > 0$ due to $\gamma_{2}>\frac{N-p}{p}$, and $C>0$ is independent of $R$ and $\lambda$. The key estimates \eqref{eqmv-15-1} and \eqref{eqmv-15-2} will lead to similar inequalities as \eqref{050410}, \eqref{050420} and \eqref{050421}, and hence play the same crucial role as \eqref{eqmv-I4-1} and \eqref{eqmv-I4-2} in the proof of Theorem \ref{th2}.

\medskip

With these key estimates in hands, the rest of the proof of radial symmetry of positive solutions to \eqref{eq1.2} is entirely similar to that of Theorem \ref{th2} in Section \ref{sc4}, so we omit the details. In particular, we extend the radial symmetry result in \cite{OSV} from $s=0$ to general weighted cases $0\leq s<p$.
\end{rem}

\end{document}